\documentclass{article}
\usepackage[a4paper,left=2.5cm,right=2.5cm, top=3cm, bottom=3cm]{geometry}
\usepackage{cuted}

\usepackage{amsfonts}
\usepackage{stmaryrd}
\usepackage{comment}
\usepackage{graphicx}
\usepackage{epstopdf}
\usepackage{scalerel}
\usepackage{framed,multirow}
\usepackage{etoolbox}
\usepackage{amscd,amsmath,amssymb,mathrsfs,bbm,listings}
\usepackage{amsthm}
\usepackage[dvipsnames]{xcolor}
\usepackage{tikz} 
\usetikzlibrary{arrows,automata,calc}
\usepackage{epsfig}
\usepackage{multirow}
\usepackage{subfig}

\newtheorem{theorem}{Theorem}[section]
\newtheorem{assumption}[theorem]{Assumption}
\newtheorem{lemma}[theorem]{Lemma}
\newtheorem{proposition}[theorem]{Proposition}
\newtheorem{corollary}[theorem]{Corollary}
\newtheorem{remark}[theorem]{Remark}

\usepackage[hidelinks]{hyperref}
\definecolor{myblue}{rgb}{0,0,0.6}   
\hypersetup{
    colorlinks,
    linktoc=section,
    citecolor=red,
    linkcolor=myblue,
    pdfborder={0 0 0}
}
\usepackage{tabularray}
\usepackage{pifont}
\newcommand{\cmark}{\ding{51}}
\newcommand{\xmark}{\ding{55}}

\numberwithin{equation}{section}

\ifpdf
  \DeclareGraphicsExtensions{.eps,.pdf,.png,.jpg}
\else
  \DeclareGraphicsExtensions{.eps}
\fi

\usepackage{cancel}
\usepackage{soul}

\usepackage{amsmath,amsthm,amssymb,mathtools}
\usepackage{array}
\newcolumntype{L}[1]{>{\raggedright\let\newline\\\arraybackslash\hspace{0pt}}m{#1}}
\newcolumntype{C}[1]{>{\centering\let\newline\\\arraybackslash\hspace{0pt}}m{#1}}
\newcolumntype{R}[1]{>{\raggedleft\let\newline\\\arraybackslash\hspace{0pt}}m{#1}}
\usepackage{nicematrix}

\definecolor{cbBlue}{RGB}{0, 119, 187}
\definecolor{cbOrange}{RGB}{255, 158, 0}
\definecolor{cbGreen}{RGB}{0, 158, 115}

\newcommand{\lp}{\lesssim_p}
\newcommand{\eremk}{\hbox{}\hfill\rule{0.8ex}{0.8ex}}
\newcommand\dx{\,\mathrm{d}\bx} 
\newcommand\dS{\,\mathrm{d}S}
\newcommand\ds{\,\mathrm{d}s}
\newcommand\dt{\,\mathrm{d}t}
\newcommand\dV{\,\mathrm{d}V}

\newcommand\Deltax{\Delta_{\bx}}
\newcommand\nablax{\nabla_{\bx}}
\newcommand\nablaxh{\nabla_{\bx, h}}
\newcommand\dpth{\partial_{t,h}}
\newcommand\nablaxLDG{\hyperref[EQN::LDG-GRADIENT]{\nabla_{\bx}^{\scaleto{\text{LDG}}{3.5pt}}}}
\newcommand\QT{Q_T}

\newcommand\Pp[2]{\mathbb{P}^{#1}(#2)}
\newcommand\bdh{\uu{d}_h}

\newcommand\mh{m_h^t}
\newcommand\bh{b_h}
\newcommand\bq{\uu{b}_q}
\newcommand\bu{b_u}
\newcommand\luh{\ell_h^u}
\newcommand\lqh{\ell_h^q}

\newcommand\suh{s_h^u}

\newcommand\A{\mathcal{A}}

\newcommand\bk{\uu{\kappa}}

\newcommand\dpt{\partial_t}
\newcommand\dptt{\partial_{tt}}

\newcommand\Lu{\hyperref[EQN::LIFTING-Uh]{\mathcal{L}_h}}

\newcommand\Bh{\hyperref[EQN::REDUCED-VARIATIONAL-DG]{\mathcal{B}_h}}
\newcommand\Rh{\hyperref[EQN::DEF-Rh]{\mathcal{R}_h}}
\newcommand\Ah{\hyperref[DEF::Ah]{\mathcal{A}_h}}
\newcommand\Nh{\hyperref[EQN::DISCRETE-NEWTON]{\mathcal{N}_h}}
\newcommand\uh{u_h}
\newcommand\wh{w_h}
\newcommand\zh{z_h}
\newcommand\q{{\uu{q}}}

\def\r{{\uu{r}}}
\newcommand\p{{\uu{p}}}

\newcommand{\pK}{p_K}
\newcommand{\lK}{\lambda_K}
\newcommand{\lambdah}{\lambda_h}
\newcommand{\hpK}{\hat{p}_K}
\newcommand{\hhKt}{\hat{h}_{\Kt}}

\newcommand\Uh{U_h}
\newcommand\D{\uu{D}}

\newcommand\Su{S_u}
\newcommand\Qh{\uu{Q}_h}
\newcommand\MT{M_t}
\newcommand\vh{v_h}
\newcommand\qh{\q_h}

\newcommand\rh{\r_h}

\newcommand{\TPi}{\widetilde{\Pi}}
\newcommand{\TPix}{\widetilde{\Pi}^{\bx}}
\newcommand{\PiO}{\boldsymbol{\Pi}_0}
\newcommand\Vp{\mathcal{V}^{\text{\scriptsize{$\p$}}}}
\newcommand\VpK{V^{\text{\scriptsize{$p_K$}}}}
\newcommand\MpK{\boldsymbol{M}^{\text{\scriptsize{$p_K$}}}}
\newcommand\Mp{\boldsymbol{\mathcal{M}}^{\text{\scriptsize{$\p$}} }}
\newcommand{\EFC}[2]{\calC^{#1}\left({#2}\right)}

\newcommand{\jump}[1]{[\![#1]\!]}

\numberwithin{equation}{section}

\newcommand\vnOmega{{{\vn}_{\Omega}^{\bx}}}
\newcommand{\vnSpace}[1]{{{\vn}_{#1}^{\bx}}}
\newcommand{\vnTime}[1]{{n_{#1}^t}}
\newlength{\dhatheight}
\newcommand\UFLUX{{\widehat{u}_h}}
\newcommand\QFLUX{
    \settoheight{\dhatheight}{\ensuremath{\widehat{q_h}}}
    \addtolength{\dhatheight}{-0.40ex}
    \widehat{\vphantom{\rule{1pt}{\dhatheight}}
    \smash{\widehat{\q}}}_h
}

\allowdisplaybreaks[4]
\newcommand{\nF}{\bn_F}
\newcommand{\nFt}{n_F^t}
\newcommand{\nFx}{\bn_{F}^{\bx}}
\newcommand{\Ft}{F_t}
\newcommand{\Fx}{F_{\bx}}
\newcommand{\ts}{\mathtt{s}}

\newcommand{\tsK}{\mathtt{s}_K}
\newcommand{\ellK}{\ell_K}
\newcommand{\sKxF}{s_{K}^{\Fx}}

\newcommand{\mvec}[1]
{{\mathbf{#1}}}
\newcommand{\vn}{{\mvec n}}
\newcommand{\uu}[1]{\hbox{\boldmath$#1$}}
\newcommand{\Uu}[1]{{\mathbf{#1}}}

\newcommand{\bzero}{\Uu{0}}
\newcommand{\calC}{{\mathcal C}}

\newcommand{\calT}{{\mathcal T}}

\newcommand{\calF}{{\mathcal F}}

\newcommand{\Fh}{\calF_h}
\newcommand{\frakE}{\mathfrak{E}}
\newcommand{\frakEx}{\mathfrak{E}_{\bx}}
\newcommand{\calK}{\mathcal{K}}
\newcommand{\calKx}{\mathcal{K}_{\bx}}
\newcommand{\Kx}{K_{\bx}}
\newcommand{\Kt}{K_t}

\newcommand{\hK}{h_{K}}

\newcommand{\hcalKx}{h_{\calKx}}
\newcommand{\hKx}{h_{K_\bx}}

\newcommand{\hKt}{h_{K_t}}
\newcommand{\Tc}{\calT^{\#}}
\newcommand{\Th}{{\calT_h}}

\newcommand{\deK}{{\partial K}}

\newcommand{\IN}{\mathbb{N}}
\newcommand{\IR}{\mathbb{R}}

\newcommand{\gD}{{g_{\mathrm D}}}

\newcommand{\oon}{\;\text{on}\;}
\newcommand{\bx}{{\Uu x}}
\newcommand{\by}{{\Uu y}}
\newcommand{\bn}{{\Uu n}}
\newcommand{\bnOmega}{\Uu{n}_{\Omega}^{\bx}}

\newcommand{\bN}{{\sf{N}}}

\newcommand{\btau}{{\boldsymbol\tau}}		

\newcommand{\IP}{\mathbb{P}}
\newcommand{\IQ}{\mathbb{P}_{\otimes}}

 \newcommand{\mvl}[1]{\{ \!\!\{#1\}\!\!\}}  
 \newcommand{\FT}{{\Fh^T}}
\newcommand{\FO}{{\Fh^0}}

\newcommand{\FD}{{\Fh^{\mathrm D}}}

\newcommand{\rtime}{{\mathrm{time}}}
\newcommand{\rspace}{{\mathrm{space}}}
\newcommand{\CP}{C_{\mathrm{P}}}
\newcommand{\CN}{C_\mathcal{N}}
\newcommand{\CL}{C_\mathcal{L}}
\newcommand{\Cinv}{C_{\mathrm{inv}}}
\newcommand{\Ctr}{C_{\mathrm{tr}}}
\newcommand{\Fspa}{{\Fh^\rspace}}
\newcommand{\Ftime}{{\Fh^\rtime}}
\newcommand{\FKtime}{{\mathcal{F}_K^\rtime}}
\newcommand{\FKspace}{{\mathcal{F}_K^\rspace}}

\DeclareMathOperator{\diam}{diam} 

\DeclareMathOperator{\essinf}{ess\,inf}

\newcommand*{\abs}[1]{|#1|}
\newcommand{\Tnorm}[2]{|\!|\!|#1|\!|\!|_{#2}}

\newcommand*{\SemiNorm}[2]{\left|#1\right|_{#2}}
\newcommand*{\Norm}[2]{\|#1\|_{#2}}

\newcommand{\LDG}{_{\hyperref[EQN::LDG-NORM]{\mathrm{LDG}}}}
\newcommand{\J}{_{\hyperref[EQN::JUMP-FUNCTIONAL]{\mathrm{J}}}}

\newcommand{\LDGN}{_{\hyperref[EQN::DG-NORMS-3]{\mathrm{LDG,}\mathcal{N}}}}
\newcommand{\LDGp}{_{\hyperref[EQN::DG-NORMS-4]{\mathrm{LDG^+}}}}
\newcommand{\LDGs}{_{\hyperref[EQN::DG-NORMS-5]{\mathrm{LDG^{\star}}}}}
\newcommand{\LDGss}{_{\hyperref[EQN::DG-NORMS-6]{\mathrm{LDG^{\diamond}}}}}

\newcommand{\mj}{{\boldsymbol{j}}}

\usepackage[lined,boxed,commentsnumbered]{algorithm2e}
\SetSymbolFont{stmry}{bold}{U}{stmry}{m}{n}
\newcommand{\qT}{{\mathbb{Q\!T}}}
\newcommand{\eT}{{\mathbb{E\!T}}}
\newcommand{\mi}{{\boldsymbol{i}}}

\newcommand{\ej}{{\boldsymbol{e}_j}}
\newcommand{\e}{{\boldsymbol{e}}}
\definecolor{cpcol}{rgb}{0.0, 0.5, 1.0}

\definecolor{pscol}{rgb}{0,0.6,0}

\usepackage{stmaryrd} 
\usepackage{pgf,tikz}
\usetikzlibrary{spy,matrix, calc, arrows,shapes,decorations}
\usepackage{tkz-euclide}
\usetikzlibrary{positioning,arrows,calc,decorations.markings,math,arrows.meta}
\usepackage{pgfplots} 
\usepgfplotslibrary{groupplots}
\usepackage{csvsimple} 
\usepackage{siunitx} 
\usepackage{tabularx}
\usepackage{booktabs}
\usepackage{multirow}
\usepackage[normalem]{ulem}
\usepackage{orcidlink}
\useunder{\uline}{\ul}{}

\pgfplotscreateplotcyclelist{colorshp}{%
violet,every mark/.append style={solid,fill=violet},mark=square*,very thick,mark size=3pt\\%
orange, every mark/.append style={solid,fill=orange},mark=*,very thick,mark size=2.9pt\\%
teal,every mark/.append style={solid,fill=teal},mark=diamond*,very thick,mark size=2.8pt\\%
magenta, every mark/.append style={solid,fill=magenta},mark=star,very thick,mark size=2.7pt\\%
violet,  dashed,every mark/.append style={solid,fill=violet},mark=square*,very thick,mark size=3pt\\%
orange,  dashed,every mark/.append style={solid,fill=orange},mark=*,very thick,mark size=2.9pt\\%
teal, dashed,every mark/.append style={solid,fill=teal},mark=diamond*,very thick,mark size=2.8pt\\%
magenta,  dashed,every mark/.append style={solid,fill=magenta},mark=star,very thick,mark size=2.7pt\\%
}

\pgfplotscreateplotcyclelist{colorsh}{%
violet,every mark/.append style={solid,fill=violet},mark=square*,very thick,mark size=3pt\\%
orange, every mark/.append style={solid,fill=orange},mark=*,very thick,mark size=2.9pt\\%
teal,every mark/.append style={solid,fill=teal},mark=diamond*,very thick,mark size=2.8pt\\%
magenta, every mark/.append style={solid,fill=magenta},mark=star,very thick,mark size=2.7pt\\%
magenta, every mark/.append style={solid,fill=magenta},mark=star,very thick,mark size=2.7pt\\%
}
\pgfplotscreateplotcyclelist{colorsp}{%
violet,every mark/.append style={solid,fill=violet},mark=square*,very thick,mark size=3pt\\%
orange, every mark/.append style={solid,fill=orange},mark=diamond*,very thick,mark size=2.9pt\\%
teal,every mark/.append style={solid,fill=teal},mark=*,very thick,mark size=2.8pt\\%
magenta, every mark/.append style={solid,fill=magenta},mark=star,very thick,mark size=2.7pt\\%
}

\pgfplotscreateplotcyclelist{colorshsing}{%
violet,every mark/.append style={solid,fill=violet},mark=square*,very thick,mark size=3pt\\%
orange, every mark/.append style={solid,fill=orange},mark=*,very thick,mark size=2.9pt\\%
teal,every mark/.append style={solid,fill=teal},mark=diamond*,very thick,mark size=2.8pt\\%
magenta, every mark/.append style={solid,fill=magenta},mark=star,very thick,mark size=2.7pt\\%
violet,every mark/.append style={solid,fill=violet},mark=square*,very thick,mark size=3pt\\%
orange, every mark/.append style={solid,fill=orange},mark=*,very thick,mark size=2.9pt\\%
teal,every mark/.append style={solid,fill=teal},mark=diamond*,very thick,mark size=2.8pt\\%
magenta, every mark/.append style={solid,fill=magenta},mark=star,very thick,mark size=2.7pt\\%
}

\pgfplotscreateplotcyclelist{colorscond}{%
red!60,every mark/.append style={solid,fill=red!60},mark=square*,very thick,mark size=3pt\\%
red!60,every mark/.append style={solid,fill=red!60},mark=square*,very thick,mark size=3pt\\%
Cerulean, every mark/.append style={solid,fill=Cerulean},mark=*,very thick,mark size=3pt\\%
Cerulean, every mark/.append style={solid,fill=Cerulean},mark=*,very thick,mark size=3pt\\%
brown!100, every mark/.append style={solid,fill=brown!100},mark=diamond*,very thick,mark size=3pt\\%
brown!100, every mark/.append style={solid,fill=brown!100},mark=diamond*,very thick,mark size=3pt\\%
}

\pgfplotsset{
    discard if not/.style 2 args={
        x filter/.append code={
            \edef\tempa{\thisrow{#1}}
            \edef\tempb{#2}
            \ifx\tempa\tempb
            \else
                
            \fi
        }
    }
}
\pgfplotsset{compat=1.16}

\pgfplotsset{tick label style={font=\small},label style={font=\small},legend style={font=\small},}
\pgfplotsset{ width=.49\linewidth}

\usepackage{pgfplotstable} 
\usepackage{booktabs} 
\usepackage{colortbl}
\makeatletter
\pgfplotstableset{
    discard if not/.style 2 args={
        row predicate/.append code={
            \def\pgfplotstable@loc@TMPd{\pgfplotstablegetelem{##1}{#1}\of}
            \expandafter\pgfplotstable@loc@TMPd\pgfplotstablename
            \edef\tempa{\pgfplotsretval}
            \edef\tempb{#2}
            \ifx\tempa\tempb
            \else
                \pgfplotstableuserowfalse
            \fi
        }
    }
}

\pgfplotsset{
	discard if/.style 2 args={
		x filter/.append code={
			\edef\tempa{\thisrow{#1}}
			\edef\tempb{#2}
			\ifx\tempa\tempb
			
			\fi
		}
	},}     
\newcommand{\CycleNextGruoupPloth}[7]
{\nextgroupplot[title={#6}, ymode=log,xmode=log, ylabel={#5},xlabel={#7}]
	\foreach \m in {cart,poly,qtrefftz,embt}{
	\addplot+[discard if not={p}{#1},discard if not={method}{\m}] table [x=h, y=#2, col sep=comma] {#3};}
   \foreach \m in {embt}{
	\addplot+[discard if not={p}{#1},discard if not={method}{\m},discard if={#4}{0}, only marks,
	visualization depends on=\thisrow{#4} \as \labela,
	nodes near coords=\pgfmathprintnumber{\labela}
	,
	every node near coord/.append style={
		black,
		draw=yellow!30,
		ellipse,
		fill=yellow!30,
		inner sep=1pt,
		xshift=5.5ex,
		yshift=5.5ex,
		scale=0.7,/pgf/number format/fixed,
        /pgf/number format/precision=2,/pgf/number format/fixed zerofill}
	] table [x=h, y=#2, col sep=comma] {#3};
	\legend{$\IQ$,$\IP$,$\qT$,$\eT$}
}}

\newcommand{\CycleNextGruoupPlotdofs}[3]
{\nextgroupplot[ymode=log, ylabel={#3},xlabel={$\sqrt{\mathrm{N}_{\mathrm{DoFs}}}$}]
\foreach \m in {cart,poly,qtrefftz,embt}{
\addplot+[discard if not={method}{\m}] table [x=totdofs2, y=#1, col sep=comma] {#2};}
\legend{$\IQ$,$\IP$,$\qT$,$\eT$}
}

\newcommand{\CycleNextGruoupPlotcond}[6]
{\nextgroupplot[ymin=10^0,ymax=10^5,title={#6}, ymode=log,xmode=log,ylabel={#5},xlabel={$h$}]
	\foreach \p in {2,3,4}{
	\addplot+[discard if not={p}{\p},discard if not={method}{#1}] table [x=h, y=#2, col sep=comma] {#3};
	\addplot+[discard if not={p}{\p},discard if not={method}{#1},discard if={#4}{0}, only marks,
	visualization depends on=\thisrow{#4} \as \labela,
	nodes near coords=\pgfmathprintnumber{\labela}
	,
	every node near coord/.append style={
		black,
		draw=yellow!30,
		ellipse,
		fill=yellow!30,
		inner sep=1pt,
		xshift=3ex,
		yshift=-1.5ex,
		scale=0.45,/pgf/number format/fixed,
        /pgf/number format/precision=2,/pgf/number format/fixed zerofill}
	] table [x=h, y=#2, col sep=comma] {#3};}
	\legend{$p=2$,,$p=3$,,$p=4$}
}

\title{Inf-sup stable space--time Local Discontinuous Galerkin method for the heat equation}

\author{Sergio G\'omez\thanks{Department of Mathematics and Applications, University of Milano-Bicocca, 20125 Milan, Italy (\href{mailto:sergio.gomezmacias@unimib.it}{sergio.gomezmacias@unimib.it})} \thanks{IMATI-CNR ``Enrico Magenes", Via Ferrata 5, 27100, Pavia, Italy}\ \orcidlink{0000-0001-9156-5135} 
\and 
Chiara Perinati\thanks{Department of Mathematics, University of Pavia, 27100 Pavia, Italy (\href{mailto:chiara.perinati01@universitadipavia.it}{chiara.perinati01@universitadipavia.it})}
\orcidlink{0009-0002-8819-928X}
\and 
Paul Stocker\thanks{Faculty of Mathematics, University of Vienna, 1090 Vienna, Austria (\href{mailto:paul.stocker@univie.ac.at}{paul.stocker@univie.ac.at})} \orcidlink{0000-0001-5073-3366}
}
\date{}

\pgfplotsset{compat=1.14}
\begin{document}
\maketitle

\begin{abstract}
\noindent We propose and analyze a space--time Local Discontinuous Galerkin method for the  approximation of the solution to parabolic problems.
The method allows for very general discrete spaces and prismatic space--time meshes. 
Existence and uniqueness of a discrete solution are shown by means of an inf-sup condition, whose proof
does not rely on polynomial inverse estimates.
Moreover, for piecewise polynomial spaces satisfying an additional mild condition, we show a second inf-sup condition that provides additional control over the time derivative of the discrete solution.
We derive~$hp$-\emph{a priori} error bounds based on these inf-sup conditions, which we use to prove convergence rates for standard, tensor-product, 
and quasi-Trefftz polynomial spaces.
Numerical experiments validate our theoretical results.
\end{abstract}



\section{Introduction \label{INTRODUCTION}}
We are interested in the numerical approximation of the solution to a parabolic problem on a space--time cylinder~$\QT = \Omega \times (0, T)$, where~$\Omega$ is an open, bounded polytopal domain in~$\IR^d$ ($d \in \{1,2, 3\}$) with Lipschitz boundary~$\partial \Omega$, and~$T > 0$ is some final time.
Let~$\bk$ be a symmetric positive definite diffusion tensor such that, for some constant~$\theta  > 0$, 
\begin{equation}
\label{EQN::DIFFUSION}
\by^T \bk \by \geq \theta |\by|^2 \qquad \forall \by \in \IR^d,
\end{equation}
and let~$f: \QT \rightarrow \IR$, $u_0 :\Omega \rightarrow \IR$, and $\gD : \partial \Omega \times (0, T) \rightarrow \IR$ be prescribed source term, initial datum, and Dirichlet boundary datum, respectively.
The considered initial and boundary value problem (IBVP) reads: find~$u : \QT \rightarrow \IR$ such that
\begin{subequations}
\label{EQN::MODEL-PROBLEM}
\begin{alignat}{3}
\label{EQN::MODEL-PROBLEM-1}
    \dpt u - \nablax \cdot (\bk \nablax u) & = f & & \qquad \text{ in } \QT, \\
\label{EQN::MODEL-PROBLEM-2}
    u & = \gD & & \qquad \text{ on } \partial \Omega \times (0, T), \\
\label{EQN::MODEL-PROBLEM-3}
    u & = u_0 & & \qquad \text{ on } \Omega \times \{0\}.
\end{alignat}
\end{subequations}

\paragraph{Continuous weak formulation.}
For initial datum~$u_0 \in L^2(\Omega)$, source term~$f \in L^2(0, T; H^{-1}(\Omega))$, and homogeneous Dirichlet boundary conditions~$(\gD = 0)$, we define the following spaces:
\begin{equation*}
Y := L^2(0, T; H_0^1(\Omega)) \quad \text{ and } \quad X:= Y \cap H^1(0, T; H^{-1}(\Omega)),
\end{equation*}
and their associated norms
\begin{equation*}
\Norm{v}{Y}^2 := \int_0^T \Norm{\nablax v(\cdot, t)}{L^2(\Omega)^d}^2 \dt \quad \text{ and } \quad \Norm{v}{X}^2 := \frac12 \Norm{v(\cdot, T)}{L^2(\Omega)}^2 + \theta \Norm{v}{Y}^2 + \Norm{\dpt v}{L^2(0, T; H^{-1}(\Omega))}^2,
\end{equation*}
where
\begin{equation*}
\Norm{\varphi}{L^2(0, T; H^{-1}(\Omega))} := \sup_{0 \neq v \in Y} \frac{\langle \varphi, v\rangle}{\Norm{v}{Y}} \qquad \qquad \forall \varphi \in L^2(0, T; H^{-1}(\Omega)),
\end{equation*} 
with~$\langle \cdot, \cdot \rangle$ denoting the duality between~$L^2(0, T; H^{-1}(\Omega))$ and~$Y$.

The standard Petrov--Galerkin weak formulation of the IBVP~\eqref{EQN::MODEL-PROBLEM} with homogeneous Dirichlet boundary conditions is: find~$u \in X$, such that
\begin{subequations}
\label{EQN::CONTINUOUS-WEAK-FORMULATION}
\begin{alignat}{3}
\label{EQN::CONTINUOUS-WEAK-FORMULATION-1}
b(u, v) & = \langle f, v \rangle & & \qquad \forall v \in Y, \\
\label{EQN::CONTINUOUS-WEAK-FORMULATION-2}
\int_\Omega u(\bx, 0) w \dx & = \int_{\Omega} u_0 w \dx & & \qquad \forall w \in L^2(\Omega),
\end{alignat}
\end{subequations}
where
\begin{equation*}
b(u, v) := \langle \dpt u, v\rangle + a(u, v), \quad \text{ with } a(u, v) := \int_{\QT} \bk \nablax u \cdot \nablax v.
\end{equation*}

According to~\cite[Thm.~4.1 and \S4.7.1 in Ch.~3]{Lions_Magenes_Vol1_1972} and~\cite[Thm.~5.1]{Schwab_Stevenson_2009}, there exists a unique solution to the weak formulation~\eqref{EQN::CONTINUOUS-WEAK-FORMULATION}. 
Moreover, the inclusion~$X \subset C^0([0, T]; L^2(\Omega))$ (see, e.g., \cite[Thm.~3 in \S5.9.2]{Evans_1998}) guarantees that~\eqref{EQN::CONTINUOUS-WEAK-FORMULATION-2} makes sense.

Since the proposed method is nonconforming, additional regularity is required for the source term~$f$. More precisely, we henceforth assume that~$f \in L^2(\QT)$.

\paragraph{Previous works.} 
Space--time finite element methods treat time as an additional space dimension in a time-dependent PDE,
which results in many advantages, such as simultaneous high-order accuracy in space and time, 
the possibility of performing space--time adaptive refinements, 
the natural treatment of problems on moving domains,
and suitability to be combined with parallel-in-time solvers.
Motivated by such advantages, in the literature, several methods have been designed for the discretization of parabolic problems; see, e.g., the recent survey in~\cite{Langer_Steinbach_2019}.
In particular, space--time methods related to the Petrov--Galerkin weak formulation in~\eqref{EQN::CONTINUOUS-WEAK-FORMULATION} include conforming finite element~\cite{Aziz_Monk_1989,Steinbach_2015}, interior-penalty discontinuous Galerkin~\cite{Cangiani_Dong_Georgoulis:2017}, and virtual element~\cite{Gomez_Mascotto_Moiola_Perugia_2024,Gomez_Mascotto_Perugia_2024} methods.
Alternative approaches, such as wavelet~\cite{Schwab_Stevenson_2009} and finite element~\cite{Andreev_2013,Stevenson_Westerdiep_2021}
methods based on a minimal residual Petrov--Galerkin formulation, least-squares methods~\cite{Bochev_Gunzburger_1998,Fuhrer_Karkulik_2021,Gantner_Stevenson_2024}, a discontinuous Petrov--Galerkin method~\cite{Diening_Storn_2022}, an isogeometric method with time-upwind test functions~\cite{Langer_Moore_Stephen_Neumuller_2016}, and a coercive formulation based on a Hilbert transformation of the test functions~\cite{Steinbach_Zank_2020} have also been considered.

We focus on discontinuous Galerkin (DG) methods, which offer great flexibility in the choice of the discrete spaces and meshes that can be employed.
In particular, the Local Discontinuous Galerkin (LDG) method~\cite{Cockburn_Shu:1998} has shown better stability properties than other DG methods (see, e.g., the comparative study of DG methods for elliptic PDEs in~\cite{Castillo_2002}).

We can reduce the number of degrees of freedom (DoFs) while keeping the approximation properties of full polynomial spaces by using special discrete spaces. 
Trefftz spaces are based on solutions to the homogeneous PDE and they have been used in the context of space--time DG methods.
In particular, for 
wave problems in one space dimension~\cite{SpaceTimeTDG,KSTW2014,PFT09},
the acoustic wave equation~\cite{PSSW20,Moiola_Perugia_2018,bgl2016},
elasto-acoustics~\cite{bcds20},
time-dependent Maxwell's equations~\cite{EKSW15},
and the linear Schr\"odinger equation~\cite{Gomez_Moiola:2024,Gomez_Moiola_Perugia_Stocker:2023}. 
The uneven degree of the heat operator in different dimensions presents an interesting challenge in the construction of Trefftz spaces.
Quasi-Trefftz spaces consists of functions that are only approximated solutions to the considered PDE. A general approach to construct suitable quasi-Trefftz spaces for linear operators is presented in~\cite{IG_Moiola_Perinati_Stocker:2024}.
A method that avoids the explicit construction of Trefftz spaces is the embedded Trefftz method, see~\cite{LS_IJMNE_2023,Lozinski_2019}.
Numerical results of the embedded Trefftz method for the heat equation, and possible choices for the Trefftz-like space, have been discussed in~\cite{ma_ch}.
A recent unifying framework for Trefftz-like methods in \cite{LLSV_ARXIV_2024} provides an error analysis for the embedded Trefftz discontinuous Galerkin method applied to a
range of scalar elliptic PDEs. 
An in-depth comparison on the parameters for the efficiency of different finite element methods on polytopal meshes can be found in \cite{LSZ_PAMM_2024}, where the advantages of Trefftz methods were further evidenced.

\paragraph{Main contributions.} 
In this work, we present and analyze a space--time LDG method for the discretization of parabolic problems.
Below, we list the main theoretical and computational results in this work.
\begin{itemize}
\item 
We establish the well-posedness of the method for very general prismatic space--time meshes and discrete spaces.

\item We show two different inf-sup conditions, which require only that the stability parameter be strictly positive. The first one is valid for any choice of the discrete spaces satisfying a local compatibility condition and does not rely on polynomial inverse estimates. The second one holds for polynomial spaces that also satisfy  a local inclusion condition and provides additional control on the first-order time derivative of the discrete solution; see~Section~\ref{SECT::WELL-POSEDNESS} for more details.

\item We derive~$hp$-\emph{a priori} error bounds in some energy norms.
Moreover, we prove $hp$-error estimates for standard and tensor-product polynomial spaces, and~$h$-error estimates for 
quasi-Trefftz spaces. The latter space allows for a significant reduction of the number of degrees of freedom; see Section~\ref{SECT::QUASI-TREFFTZ}.

\item Optimal convergence rates of order~$\mathcal{O}(h^{p+1})$ (where~$h$ denotes the maximum diameter of the space--time prismatic elements) are numerically observed for the error in the~$L^2(\QT)$ norm when standard piecewise polynomials of 
uniform degree~$p$ are used. 
\end{itemize}

Our stability and convergence results are analogous to those obtained in~\cite{Cangiani_Dong_Georgoulis:2017} for the space--time interior-penalty DG method with standard polynomial spaces. However, in our analysis, we allow for prismatic meshes with hanging~\emph{time-like} facets (see Figure~\ref{fig:hanging-time-like-facets} and the notation introduced in Section~\ref{SECT::MESH-NOTATION} below), which naturally arise when using local time steps.
In particular, we show that stability and \emph{a priori} error estimates can be established for elements with an arbitrary number of both \emph{space-like} and~\emph{time-like} facets. Moreover, as in the elliptic case, the proposed LDG method avoids the typical requirement of \emph{primal} DG methods (such as the interior-penalty DG method) of a ``sufficiently large" stability parameter, which can have a negative impact on the conditioning of the associated matrix (see, e.g., \cite[\S4.2]{Castillo_2002}).

\begin{figure}[!ht]
    \centering
		\resizebox{6.5cm}{!}{
			\begin{tikzpicture}
				[rotate around z=-18,rotate around y=-30,rotate around x=20,grid/.style={very thin,gray},axis/.style={->,thick}]
				
				\draw[axis] (6,0,0) -- (7,0,0) node[anchor=west]{\Large $x_1$};
				\draw[axis] (0,3,0) -- (0,4,0) node[anchor=east,xshift=-2pt]{\Large $t$};
				\draw[axis] (0,0,6) -- (0,0,7) node[anchor=east]{\Large $x_2$};
				\draw[dashed,lightgray]  (0,0,0) -- (6,0,0);			
                \draw[dashed,lightgray]  (0,0,0) -- (0,3,0);			
                \draw[dashed,lightgray]  (0,0,0) -- (0,0,6);			
				
				\draw[dashed,lightgray]  (1,0,0) -- (1,3,0);				
				\draw[dashed,lightgray]  (2.4,0,0) -- (2.4,3,0);	
				\draw[dashed,lightgray]  (4.1,0,0) -- (4.1,3,0);
				\draw[dashed,lightgray]  (5.2,0,0) -- (5.2,3,0);						
				\draw[dashed,lightgray]  (0,0,4.8) -- (0,3,4.8);		
				\draw[dashed,lightgray]  (0,0,3.3) -- (0,3,3.3);		
				\draw[dashed,lightgray]  (0,0,1.7) -- (0,3,1.7);		
				\draw[dashed,lightgray]  (1.1,0,1.3) -- (1.1,3,1.3);		
				\draw[dashed,lightgray]  (1.5,0,1.8) -- (1.5,3,1.8);		
				\draw[dashed,lightgray]  (2.4,0,1.8) -- (2.4,3,1.8);		
				\draw[dashed,lightgray]  (2.7,0,1.1) -- (2.7,3,1.1);		
				\draw[dashed,lightgray]  (1.2,0,2.9) -- (1.2,3,2.9);		
				\draw[dashed,lightgray]  (1.6,0,3.4) -- (1.6,3,3.4);		
				\draw[dashed,lightgray]  (1.6,0,4.5) -- (1.6,3,4.5);		
				\draw[dashed,lightgray]  (1.2,0,4.8) -- (1.2,3,4.8);		
				\draw[dashed,lightgray]  (2.8,0,5.2) -- (2.8,3,5.2);		
				\draw[dashed,lightgray]  (3.3,0,4.5) -- (3.3,3,4.5);		
				\draw[dashed,lightgray]  (2.8,0,3.4) -- (2.8,3,3.4);		
				\draw[dashed,lightgray]  (3.1,0,2.6) -- (3.1,3,2.6);		
				\draw[dashed,lightgray]  (5,0,1.5) -- (5,3,1.5);		
				\draw[dashed,lightgray]  (4.4,0,1.8) -- (4.4,3,1.8);		
				\draw[dashed,lightgray]  (3.9,0,1.1) -- (3.9,3,1.1);		
				\draw[dashed,lightgray]  (4.2,0,2.6) -- (4.2,3,2.6);		
				\draw[dashed,lightgray]  (4.5,0,3.4) -- (4.5,3,3.4);		
				\draw[dashed,lightgray]  (5.2,0,3.3) -- (5.2,3,3.3);	
				\draw[dashed,lightgray]  (4.2,0,4.5) -- (4.2,3,4.5);		
				\draw[dashed,lightgray]  (4.7,0,5) -- (4.7,3,5);	
				\draw[dashed,lightgray] (0,0,1.7) -- (1.1,0,1.3);
				\draw[dashed,lightgray]  (1.5,0,1.8) -- (1.1,0,1.3);
				\draw[dashed,lightgray]  (1.2,0,2.9) -- (1.5,0,1.8);
				\draw[dashed,lightgray]  (1.2,0,2.9) -- (0,0,3.3);
				\draw[dashed,lightgray] (1,0,0) -- (1.1,0,1.3);
				\draw[dashed,lightgray]  (1.6,0,4.5) -- (1.6,0,3.4);
				\draw[dashed,lightgray]  (1.2,0,2.9) -- (1.6,0,3.4);
				\draw[dashed,lightgray]  (2.8,0,3.4) -- (3.3,0,4.5);
				\draw[dashed,lightgray]  (2.8,0,3.4) -- (1.6,0,3.4);
				\draw[dashed,lightgray]  (2.8,0,3.4) -- (3.1,0,2.6);
				\draw[dashed,lightgray] (4.4,0,1.8) -- (3.9,0,1.1);
				\draw[dashed,lightgray]  (4.1,0,0) -- (3.9,0,1.1);
				\draw[dashed,lightgray]  (2.7,0,1.1) -- (3.9,0,1.1);
				\draw[dashed,lightgray] (2.7,0,1.1) -- (2.4,0,0);
				\draw[dashed,lightgray]  (2.7,0,1.1) -- (2.4,0,1.8);
				\draw[dashed,lightgray]  (2.7,0,1.1) -- (2.4,0,1.8);
				\draw[dashed,lightgray]  (1.5,0,1.8) -- (2.4,0,1.8);
				\draw[dashed,lightgray] (2.4,0,1.8) -- (3.1,0,2.6);
				\draw[dashed,lightgray]  (4.2,0,2.6) -- (3.1,0,2.6);			
				\draw[dashed,lightgray]  (4.2,0,4.5) -- (3.3,0,4.5);			
				
				\draw[dashed,lightgray] (6,0,1.9)--(6,0,0)--(5.2,0,0)--(5,0,1.5)--(6,0,1.9);
				\draw[dashed,lightgray]	(5.2,0,3.3)--(4.5,0,3.4)--(4.2,0,2.6)--(4.4,0,1.8)--(5,0,1.5)--(6,0,1.9)--(6,0,3.3)--(5.2,0,3.3);
				\draw[dashed,lightgray] (6,1.5,3.3)--(5.2,1.5,3.3)--(4.5,1.5,3.4)--(4.2,1.5,2.6)--(4.4,1.5,1.8)--(5,1.5,1.5)--(6,1.5,1.9);
				\draw[dashed,lightgray] (6,0,3.3)--(5.2,0,3.3)--(4.5,0,3.4)--(4.2,0,4.5)--(4.7,0,5)--(6,0,4.9);
				\draw[dashed,lightgray] (6,1.5,3.3)--(5.2,1.5,3.3)--(4.5,1.5,3.4)--(4.2,1.5,4.5)--(4.7,1.5,5)--(6,1.5,4.9);
				\draw[dashed,lightgray] (6,0.75,3.3)--(5.2,0.75,3.3)--(4.5,0.75,3.4)--(4.2,0.75,4.5)--(4.7,0.75,5)--(6,0.75,4.9);
				\draw[dashed,lightgray] (6,2.25,3.3)--(5.2,2.25,3.3)--(4.5,2.25,3.4)--(4.2,2.25,4.5)--(4.7,2.25,5)--(6,2.25,4.9);		
				\draw[dashed,lightgray] (4.7,0.375,6)--(4.7,0.375,5)--(6,0.375,4.9);
				\draw[dashed,lightgray] (4.7,0.75,6)--(4.7,0.75,5)--(6,0.75,4.9);
				\draw[dashed,lightgray] (4.7,1.125,6)--(4.7,1.125,5)--(6,1.125,4.9);
				\draw[dashed,lightgray] (4.7,1.5,6)--(4.7,1.5,5)--(6,1.5,4.9);
				\draw[dashed,lightgray] (4.7,1.875,6)--(4.7,1.875,5)--(6,1.875,4.9);
				\draw[dashed,lightgray] (4.7,2.25,6)--(4.7,2.25,5)--(6,2.25,4.9);
				\draw[dashed,lightgray] (4.7,2.625,6)--(4.7,2.625,5)--(6,2.625,4.9);
				\draw[dashed,lightgray] (4.7,0,6)--(4.7,0,5)--(6,0,4.9);
				\draw[dashed,lightgray] (4.7,0,6)--(4.7,0,5)--(4.2,0,4.5)--(3.3,0,4.5)--(2.8,0,5.2)--(2.8,0,6);
				\draw[dashed,lightgray] (4.7,2.25,6)--(4.7,2.25,5)--(4.2,2.25,4.5)--(3.3,2.25,4.5)--(2.8,2.25,5.2)--(2.8,2.25,6);
				\draw[dashed,lightgray] (4.7,1.5,6)--(4.7,1.5,5)--(4.2,1.5,4.5)--(3.3,1.5,4.5)--(2.8,1.5,5.2)--(2.8,1.5,6);
				\draw[dashed,lightgray] (4.7,0.75,6)--(4.7,0.75,5)--(4.2,0.75,4.5)--(3.3,0.75,4.5)--(2.8,0.75,5.2)--(2.8,0.75,6);
				\draw[dashed,lightgray]  (2.8,3,6)--(2.8,3,5.2)--(1.6,3,4.5)--(1.2,3,4.8)--(1.5,3,6);
				\draw[dashed,lightgray]  (2.8,1.5,6)--(2.8,1.5,5.2)--(1.6,1.5,4.5)--(1.2,1.5,4.8)--(1.5,1.5,6);
				\draw[dashed,lightgray] (0,0,6)--(1.5,0,6)--(1.2,0,4.8)--(0,0,4.8);
				
				\draw[fill=blue,opacity=.1] (6,0,1.9)--(6,0,0)--(5.2,0,0)--(5,0,1.5);
				\draw[fill=blue,opacity=.1] (6,3,1.9)--(6,3,0)--(5.2,3,0)--(5,3,1.5);
				\draw[fill=blue,opacity=.1] (6,3,1.9)--(6,0,1.9)--(5,0,1.5)--(5,3,1.5);
				\draw[fill=blue,opacity=.1] (5.2,3,0)--(5.2,0,0)--(5,0,1.5)--(5,3,1.5);
				\draw[fill=blue,opacity=.1] (5.2,0,0)--(5.2,3,0)--(6,3,0)--(6,0,0);
				\draw[fill=blue,opacity=.1] (6,3,1.9)--(6,3,0)--(6,0,0)--(6,0,1.9);
				
				\draw[fill=blue,opacity=.1] (0,0,6)--(1.5,0,6)--(1.2,0,4.8)--(0,0,4.8);
				\draw[fill=blue,opacity=.1] (0,3,6)--(1.5,3,6)--(1.2,3,4.8)--(0,3,4.8);
				\draw[fill=blue,opacity=.1] (0,3,6)--(1.5,3,6)--(1.5,0,6)--(0,0,6);
				\draw[fill=blue,opacity=.1] (1.2,3,4.8)--(1.5,3,6)--(1.5,0,6)--(1.2,0,4.8);
				\draw[fill=blue,opacity=.1] (1.2,3,4.8)--(0,3,4.8)--(0,0,4.8)--(1.2,0,4.8);
				\draw[fill=blue,opacity=.1] (0,3,6)--(0,3,4.8)--(0,0,4.8)--(0,0,6);
				%
				%
				\draw[fill=yellow,opacity=.1]	(5.2,0,3.3)--(4.5,0,3.4)--(4.2,0,2.6)--(4.4,0,1.8)--(5,0,1.5)--(6,0,1.9)--(6,0,3.3);
				\draw[fill=yellow,opacity=.1]	(5.2,3,3.3)--(4.5,3,3.4)--(4.2,3,2.6)--(4.4,3,1.8)--(5,3,1.5)--(6,3,1.9)--(6,3,3.3);
				\draw[fill=yellow,opacity=.1] (6,0,3.3)--(6,3,3.3)--(6,3,1.9)--(6,0,1.9);
				\draw[fill=yellow,opacity=.1] (6,0,3.3)--(6,3,3.3)--(5.2,3,3.3)--(5.2,0,3.3);
				\draw[fill=yellow,opacity=.1] (4.5,0,3.4)--(4.5,3,3.4)--(5.2,3,3.3)--(5.2,0,3.3);
				\draw[fill=yellow,opacity=.1] (4.5,0,3.4)--(4.5,3,3.4)--(4.2,3,2.6)--(4.2,0,2.6);
				\draw[fill=yellow,opacity=.1] (4.4,0,1.8)--(4.4,3,1.8)--(4.2,3,2.6)--(4.2,0,2.6);
				\draw[fill=yellow,opacity=.1] (5,0,1.5)--(5,3,1.5)--(4.4,3,1.8)--(4.4,0,1.8);
				\draw[fill=yellow,opacity=.1] (5,0,1.5)--(5,3,1.5)--(6,3,1.9)--(6,0,1.9);
				\draw[thick] (6,1.5,1.9)--(6,1.5,3.3);
				%
				\draw[fill=yellow,opacity=.1] (1.5,0,6)--(2.8,0,6)--(2.8,0,5.2)--(1.6,0,4.5)--(1.2,0,4.8)--(1.5,0,6);
				\draw[fill=yellow,opacity=.1] (1.5,3,6)--(2.8,3,6)--(2.8,3,5.2)--(1.6,3,4.5)--(1.2,3,4.8)--(1.5,3,6);
				\draw[fill=yellow,opacity=.1] (1.5,3,6)--(2.8,3,6)--(2.8,0,6)--(1.5,0,6);
				\draw[fill=yellow,opacity=.1] (2.8,3,5.2)--(2.8,3,6)--(2.8,0,6)--(2.8,0,5.2);
				\draw[fill=yellow,opacity=.1] (2.8,3,5.2)--(1.6,3,4.5)--(1.6,0,4.5)--(2.8,0,5.2);
				\draw[fill=yellow,opacity=.1] (1.2,3,4.8)--(1.6,3,4.5)--(1.6,0,4.5)--(1.2,0,4.8);
				\draw[fill=yellow,opacity=.1] (1.2,3,4.8)--(1.5,3,6)--(1.5,0,6)--(1.2,0,4.8);
				\draw[thick] (1.5,1.5,6)--(2.8,1.5,6);		
				
				\draw[fill=red,opacity=.1] (6,0,4.9)--(6,0,3.3)--(5.2,0,3.3)--(4.5,0,3.4)--(4.2,0,4.5)--(4.7,0,5);
				\draw[fill=red,opacity=.1] (6,3,4.9)--(6,3,3.3)--(5.2,3,3.3)--(4.5,3,3.4)--(4.2,3,4.5)--(4.7,3,5);
				\draw[fill=red,opacity=.1] (6,0,3.3)--(6,3,3.3)--(5.2,3,3.3)--(5.2,0,3.3);
				\draw[fill=red,opacity=.1] (4.5,0,3.4)--(4.5,3,3.4)--(5.2,3,3.3)--(5.2,0,3.3);
				\draw[fill=red,opacity=.1] (6,3,3.3)--(6,3,4.9)--(6,0,4.9)--(6,0,3.3);
				\draw[fill=red,opacity=.1]
				(4.5,0,3.4)--(4.2,0,4.5)--(4.2,3,4.5)--(4.5,3,3.4);
				\draw[fill=red,opacity=.1]
				(4.7,0,5)--(4.2,0,4.5)--(4.2,3,4.5)--(4.7,3,5);
				\draw[fill=red,opacity=.1]
				(4.7,0,5)--(6,0,4.9)--(6,3,4.9)--(4.7,3,5);
				\draw[thick] (6,0.75,4.9)--(6,0.75,3.3);
				\draw[thick] (6,1.5,4.9)--(6,1.5,3.3);
				\draw[thick] (6,2.25,4.9)--(6,2.25,3.3);
				
				\draw[fill=red,opacity=.1] (4.7,0,6)--(4.7,0,5)--(4.2,0,4.5)--(3.3,0,4.5)--(2.8,0,5.2)--(2.8,0,6);
				\draw[fill=red,opacity=.1] (4.7,3,6)--(4.7,3,5)--(4.2,3,4.5)--(3.3,3,4.5)--(2.8,3,5.2)--(2.8,3,6);
				\draw[fill=green,opacity=.1]
				(4.7,0,6)--(4.7,0,5)--(4.7,3,5)--(4.7,3,6);
				\draw[fill=red,opacity=.1] (4.7,0,5)--(4.2,0,4.5)--(4.2,3,4.5)--(4.7,3,5);
				\draw[fill=red,opacity=.1] (3.3,0,4.5)--(4.2,0,4.5)--(4.2,3,4.5)--(3.3,3,4.5);
				\draw[fill=red,opacity=.1] (3.3,0,4.5)--(2.8,0,5.2)--(2.8,3,5.2)--(3.3,3,4.5);
				\draw[fill=red,opacity=.1] (2.8,0,6)--(2.8,0,5.2)--(2.8,3,5.2)--(2.8,3,6);
				\draw[fill=red,opacity=.1] (2.8,0,6)--(4.7,0,6)--(4.7,3,6)--(2.8,3,6);
				\draw[thick] (4.7,1.5,6)--(2.8,1.5,6);
				\draw[thick] (4.7,0.75,6)--(2.8,0.75,6);
				\draw[thick] (4.7,2.25,6)--(2.8,2.25,6);
				%
				\draw[fill=green,opacity=.1] (6,0,6)--(4.7,0,6)--(4.7,0,5)--(6,0,4.9);
				\draw[fill=green,opacity=.1] (6,3,6)--(4.7,3,6)--(4.7,3,5)--(6,3,4.9);
				\draw[fill=green,opacity=.1]
				(4.7,0,5)--(6,0,4.9)--(6,3,4.9)--(4.7,3,5);
				\draw[fill=green,opacity=.1] (6,0,6)--(6,3,6)--(4.7,3,6)--(4.7,0,6);
				\draw[fill=green,opacity=.1] (6,0,6)--(6,3,6)--(6,3,4.9)--(6,0,4.9);
				\draw[fill=green,opacity=.1]
				(4.7,0,6)--(4.7,0,5)--(4.7,3,5)--(4.7,3,6);
				
				\draw[thick] (6,2.625,4.9)--(6,2.625,6)--(4.7,2.625,6);		
				\draw[thick] (6,2.25,4.9)--(6,2.25,6)--(4.7,2.25,6);
				\draw[thick] (6,1.875,4.9)--(6,1.875,6)--(4.7,1.875,6);
				\draw[thick] (6,1.5,4.9)--(6,1.5,6)--(4.7,1.5,6);
				\draw[thick] (6,1.125,4.9)--(6,1.125,6)--(4.7,1.125,6);
				\draw[thick] (6,0.75,4.9)--(6,0.75,6)--(4.7,0.75,6);
				\draw[thick] (6,0.375,4.9)--(6,0.375,6)--(4.7,0.375,6);
				
				\foreach \c in {thick}{
					\draw[\c] (0,3,1.7) -- (1.1,3,1.3);
					\draw[\c]  (1.5,3,1.8) -- (1.1,3,1.3);
					\draw[\c]  (1.2,3,2.9) -- (1.5,3,1.8);
					\draw[\c]  (1.2,3,2.9) -- (0,3,3.3);
					\draw[\c] (1,3,0) -- (1.1,3,1.3);
					\draw[\c]  (0,3,4.8) -- (1.2,3,4.8);
					\draw[\c]  (1.5,3,6) -- (1.2,3,4.8);
					\draw[\c]  (1.6,3,4.5) -- (1.6,3,3.4);
					\draw[\c] (1.6,3,4.5) -- (1.2,3,4.8);
					\draw[\c]  (1.2,3,2.9) -- (1.6,3,3.4);
					\draw[\c] (2.8,3,5.2) -- (1.6,3,4.5);
					\draw[\c]  (2.8,3,5.2) -- (2.8,3,6);
					\draw[\c]  (2.8,3,5.2) -- (3.3,3,4.5);
					\draw[\c]  (2.8,3,3.4) -- (3.3,3,4.5);
					\draw[\c]  (2.8,3,3.4) -- (1.6,3,3.4);
					\draw[\c]  (2.8,3,3.4) -- (3.1,3,2.6);			
					\draw[\c]  (4.7,3,5) -- (4.7,3,6);
					\draw[\c]  (4.7,3,5) -- (6,3,4.9);
					\draw[\c]   (4.7,3,5) -- (4.2,3,4.5);
					\draw[\c]  (4.5,3,3.4) -- (4.2,3,4.5);
					\draw[\c]   (5.2,3,0) -- (5,3,1.5);
					\draw[\c] (6,3,1.9) -- (5,3,1.5);
					\draw[\c]  (4.4,3,1.8) -- (5,3,1.5);
					\draw[\c]  (4.4,3,1.8) -- (3.9,3,1.1);
					\draw[\c] (4.1,3,0) -- (3.9,3,1.1);
					\draw[\c]  (2.7,3,1.1) -- (3.9,3,1.1);
					\draw[\c] (2.7,3,1.1) -- (2.4,3,0);
					\draw[\c]  (2.7,3,1.1) -- (2.4,3,1.8);
					\draw[\c]  (2.7,3,1.1) -- (2.4,3,1.8);
					\draw[\c]  (1.5,3,1.8) -- (2.4,3,1.8);
					\draw[\c] (2.4,3,1.8) -- (3.1,3,2.6);
					\draw[\c] (4.2,3,2.6) -- (3.1,3,2.6);
					\draw[\c] (4.2,3,2.6) -- (4.4,3,1.8);
					\draw[\c]  (4.2,3,2.6) -- (4.5,3,3.4);
					\draw[\c]  (5.2,3,3.3) -- (4.5,3,3.4);		
					\draw[\c]  (5.2,3,3.3) -- (6,3,3.3);				
					\draw[\c]  (4.2,3,4.5) -- (3.3,3,4.5);
					\draw[\c]  (6,3,0) -- (6,3,6);
					\draw[\c]  (0,3,6) -- (6,3,6);	}
				
				\foreach \c in {thick}{
					\draw[\c]  (6,0,1.9) -- (6,3,1.9);		
					\draw[\c]  (6,0,3.3) -- (6,3,3.3);		
					\draw[\c]  (6,0,4.9) -- (6,3,4.9);		
					\draw[\c]  (4.7,0,6) -- (4.7,3,6);		
					\draw[\c]  (2.8,0,6) -- (2.8,3,6);	
					\draw[\c]  (1.5,0,6) -- (1.5,3,6);		
					\draw[\c]  (0,0,6) -- (0,3,6);		
					\draw[\c]  (6,0,0) -- (6,3,0);		
					\draw[\c]  (0,0,6) -- (0,3,6);		
					\draw[\c]  (6,3,6) -- (0,3,6);		
					\draw[\c]  (6,3,6) -- (6,3,0);		
					\draw[\c]  (6,3,6) -- (6,0,6); 
					\draw[\c]  (0,3,0) -- (0,3,6); 
					\draw[\c]  (0,3,0) -- (6,3,0); 
					\draw[\c]  (6,0,6) -- (0,0,6); 
					\draw[\c]  (6,0,6) -- (6,0,0); 
				}
			\end{tikzpicture}
		}
            \vspace{0.5cm}
    \caption{
    Example of a $(2+1)$-dimensional prismatic space--time mesh with polygonal \emph{space-like} facets and hanging \emph{time-like} facets.}
    \label{fig:hanging-time-like-facets}
\end{figure}

Finally, we highlight that the LDG method for polytopal meshes has previously been analyzed only in~\cite{Ye_Zhang_Zhu:2022} (for elliptic problems). Furthermore, some results from the present analysis can be adapted to study the~$hp$ version of the LDG method for elliptic problems on a broader class of polytopal meshes.

\paragraph{Notation.}
We denote by~$\dpt$ and~$\dptt$ the first- and second-order time derivatives, respectively, and by~$\nablax$ and~$\Deltax$ the spatial gradient and Laplacian operators. 

Standard notation for~$L^p$, Sobolev, and Bochner spaces is used. For a given bounded, Lipschitz domain~$\Upsilon \subset \IR^d$~$(d = 1, 2, 3)$ and a real number~$s$, we denote by~$H^s(\Upsilon)$ the Sobolev space of order~$s$, endowed with the inner product~$(\cdot,\cdot)_{H^s(\Upsilon)}$, the seminorm~$|\cdot|_{H^s(\Upsilon)}$, and the norm~$\|\cdot \|_{H^s(\Upsilon)}$. In
particular, $H^0(\Upsilon) := L^2(\Upsilon)$ is the space of Lebesgue square integrable
functions over~$\Upsilon$, 
and~$H_0^1(\Upsilon)$ is the closure of~$C_0^{\infty}(\Upsilon)$ in the~$H^1(\Upsilon)$ norm.
Given~$n \in \IN$, a time interval~$(a, b)$, 
and a Banach space~$(X, \Norm{\cdot}{X})$,
the corresponding Bochner space is denoted by~$H^n(a, b; X)$.

We use the following notation for the algebraic tensor product of two spaces, say~$V$ and~$W$:
\begin{equation*}
V \otimes W := \mathrm{span} \{vw \, :\, v \in V \text{ and } w \in W\}.
\end{equation*}
Moreover, given~$p \in \IN$ and an open, bounded set~$\mathcal{D}$, we denote by~$\Pp{p}{\mathcal{D}}$ the space of polynomials of total degree at most~$p$ defined on~$\mathcal{D}$.

\paragraph{Structure of the paper.} In Section~\ref{SECT::SPACE-TIME-DG}, we introduce the space--time notation used and describe the proposed method.
Section~\ref{SECT::WELL-POSEDNESS} is devoted to prove the inf-sup stability of the method.
In Section~\ref{SECT::CONVERGENCE-ANALYSIS}, we derive~\emph{a priori} error bounds in some 
discrete energy norms, which are used to derive~\emph{a priori} error estimates for different discrete polynomial spaces in Section~\ref{SECT::DISCRETE-SPACES}. 
In Section~\ref{SECT::NUMERICAL-EXPERIMETNS}, we present several numerical experiments to validate our theoretical results and assess some additional aspects of the proposed method. Some concluding remarks are presented in Section~\ref{SECT::CONCLUSIONS}.

\section{Definition of the method\label{SECT::SPACE-TIME-DG}}
In this section, we present the proposed space--time LDG method for the discretization of model~\eqref{EQN::MODEL-PROBLEM}.
In Section~\ref{SECT::MESH-NOTATION}, we introduce the notation used for prismatic space--time meshes, and for DG weighted averages and normal jumps.
A mixed space--time LDG formulation for model~\eqref{EQN::MODEL-PROBLEM} is introduced in~Section~\ref{SECT::LDG-METHOD} for variable degrees of accuracy and generic discrete spaces.
Such a mixed formulation is then reduced to one involving only the primary discrete unknown in Section~\ref{SECT::REDUCED-FORMULATION}.

\subsection{Space--time mesh and DG notation \label{SECT::MESH-NOTATION}}
Let~$\Th$ be a nonoverlapping prismatic partition of the space--time domain~$\QT$, i.e., any element~$K \in \Th$ can be written as~$K = \Kx \times \Kt$, for some~$d$-dimensional polytope~$\Kx \subset \Omega$ and some time interval~$\Kt \subset (0, T)$.
We use the notation~$\hKx = \diam(\Kx)$, $\hKt = \abs{\Kt}$, and $\hK = \diam(K)$.
We call ``mesh facet'' any intersection~$F \subset \deK_1\cap \deK_2$, for
$K_1,K_2\in\Th$, $F \subset (\partial \Kx \times \Kt) \cap\partial \QT$, or~$F \subset (\Kx \times \partial \Kt) \cap \partial\QT$ that has positive $d$-dimensional measure and is contained in a $d$-dimensional hyperplane.
For each mesh facet~$F$, 
we set~$\nF = (\nFx,
\nFt) \in \IR^{d+1}$ as one of the two unit normal vectors orthogonal to~$F$ (for prismatic elements, either~$\nFt = 0$ or $\nFt = 1$). We can then classify each mesh facet~$F$ as a (see, e.g., \cite[\S3]{Moiola_Perugia_2018})
\begin{equation*}
\begin{cases}
\text{\emph{space-like} facet} & \text{if }
\nFt = 1; \\
\text{\emph{time-like} facet} & \text{if } \nFt = 0.
\end{cases}
\end{equation*}
In Figure~\ref{fig:space-time-like-facets}, we illustrate these definitions, whose terminology is adopted from the relativistic geometry of space--time (see, e.g., \cite[Ch.~17]{Feynman:1964}).

\begin{figure}[!ht]
    \centering
		\resizebox{6cm}{!}{
			\begin{tikzpicture}
				[rotate around z=0,rotate around y=-270,rotate around x=0,grid/.style={very thin,gray},axis/.style={->,thick}]
                
				\draw[dashed] (0,0,0) -- (11.8,0,0);
                \draw[axis] (11.8,0,0) -- (13,0,0) node[anchor=west,yshift=-9pt]{\large $x_1$};
				\draw[axis] (0,0,0) -- (0,5,0) node[anchor=east,xshift=-2pt]{ \large$t$};
				\draw[axis] (0,0,0) -- (0,0,6) node[anchor=north,yshift=-2pt]{\large $x_2$};
				\draw[fill=gray!50,opacity=0.5]  (0,3,0) -- (0,3,3) -- (1.8,3,5) -- (4.5,3,2.5) -- (4.2,3,0) -- (0,3,0) ;
                \draw[fill=cyan!60,opacity=0.2] (1.8,3,5) -- (1.8,0,5) -- (4.5,0,2.5) -- (4.5,3,2.5) ;
                \draw[axis] (2.1,3,2.1)--(2.1,4.5,2.1) node[anchor=west, yshift=7pt]{\Large $\bn_{F^{\rspace}} =  (\mathbf{0}, 1) 
                $};
                \draw[axis] (3.15,1.5,3.75)--(4.1694,1.5,4.8504) node[anchor=west]{\Large $\bn_{F^{\rtime}} = (\bn^\bx_{F^\mathrm{time}}, 0)$};
                \draw  (3.15,1,3.2)  node[anchor=west]{\Large $F^\mathrm{time}$};
                \draw (2.1,2.6,1.2) node[anchor=west]{\Large $F^\mathrm{space}$};
				\draw (2.1,0,1.3) node[anchor=west]{\Large $K$};

				\foreach \c in {thick}{
					\draw[\c] (0,0,0) -- (0,0,3);
                    \draw[\c] (0,0,3) -- (1.8,0,5);
                    \draw[dashed] (1.8,0,5) -- (4.5,0,2.5) ;
                    \draw[dashed] (4.5,0,2.5) -- (4.2,0,0) ;
                     \draw[dashed] (4.2,0,0) -- (0,0,0) ;
                     \draw[\c] (0,3,0) -- (0,3,3);
                    \draw[\c] (0,3,3) -- (1.8,3,5);
                    \draw[\c] (1.8,3,5) -- (4.5,3,2.5) ;
                    \draw[\c] (4.5,3,2.5) -- (4.2,3,0) ;
                    \draw[\c] (4.2,3,0) -- (0,3,0) ;
                     \draw[\c] (0,3,0) -- (0,0,0) ;
                    \draw[\c] (0,3,3) -- (0,0,3);
                    \draw[\c] (1.8,3,5) -- (1.8,0,5) ;
                    \draw[dashed] (4.5,3,2.5) -- (4.5,0,2.5) ;
                     \draw[dashed] (4.2,3,0) -- (4.2,0,0);
                    }
			\end{tikzpicture}
		}
        \vspace{0.5cm}
\caption{Example of a $(2+1)$-dimensional prismatic element.
A space-like facet ($F^{\rspace}$) and a time-like facet~($F^{\rtime}$) are highlighted in gray and blue, respectively.
}
    \label{fig:space-time-like-facets}
\end{figure}

For each~$K \in \Th$, we denote by~$\FKspace$ and~$\FKtime$ the sets of space-like and time-like facets of~$K$, respectively. We define the mesh skeleton and its
parts as follows:
\begin{alignat*}{3}
\Fh & := \bigcup_{K \in \Th} \big\{ F : F \in \FKspace \cup \FKtime \big\}, & \ 
\FO & := \bigcup_{K \in \Th} \big\{ F : F \in \FKspace \text{ and } F \subset \Omega \times \{0\} \big\}, \\
\FT & := \bigcup_{K \in \Th} \big\{ F : F \in \FKspace \text{ and } F \subset \Omega \times \{T\} \big\}, & \
\Fspa & := \bigcup_{K \in \Th} \big\{ F :  F \in \FKspace \text{ and }  F \not\in \FT \cup \FO \big\}, \\
\FD & := \bigcup_{K \in \Th} \big\{ F :  F \in \FKtime \text{ and } F \subset \partial \Omega \times (0, T) \big\}, 
& \
\Ftime & := \bigcup_{K \in \Th} \big\{ F :  F \in \FKtime \text{ and } F \not\in \FD \big\}.
\end{alignat*}
Henceforth, we use the following shorthand notation for~$\star \in \{0,\, T,\,  \rspace\}$ and~$\diamond \in \{\mathrm{D}, \rtime\}$:
\begin{equation*}
\int_{\Fh^{\star}} \varphi \dx = \sum_{F \in \Fh^{\star}} \int_F \varphi \dx \quad \text{and} \quad \int_{\Fh^{\diamond}} \varphi \dS = \sum_{F \in \Fh^{\diamond}} \int_F \varphi \dS.
\end{equation*}

Boldface will be used to denote~$d$-dimensional vector fields.
We further employ the following standard DG notation for weighted averages~($\mvl{\cdot}_{\alpha}$,\  with~$\alpha \in [0, 1]$), spatial normal jumps~$(\jump{\cdot}_{\bN})$, and temporal normal jumps~$(\jump{\cdot}_t)$ for 
piecewise scalar functions and vector fields:
\begin{align*}
&\begin{cases}
\mvl{w}_{\alpha} : = (1 - \alpha) w_{|_{K_1}} + \alpha w_{|_{K_2}}\\
\mvl{\btau}_{\alpha} : = (1 - \alpha) \btau_{|_{K_1}} + \alpha \btau_{|_{K_2}}
\end{cases}
&&\oon F \in \mathcal{F}_{K_1}^{\rtime} \cap \mathcal{F}_{K_2}^{\rtime} \text{ for neighboring~$K_1,\, K_2  \in \Th$}, 
\\
&\begin{cases}
\jump{w}_\bN : = w_{|_{K_1}} \vnSpace{K_1} + w_{|_{K_2}} \vnSpace{K_2}\\
\jump{\btau}_\bN : = \btau_{|_{K_1}} \cdot \vnSpace{K_1} + \btau_{|_{K_2}} \cdot \vnSpace{K_2}
\end{cases}
&&
\oon F \in \mathcal{F}_{K_1}^{\rtime} \cap \mathcal{F}_{K_2}^{\rtime}  \text{ for neighboring~$K_1,\, K_2  \in \Th$}, 
\\
&
\ \jump{w}_t : = w_{|_{K_1}} \vnTime{K_1} + w_{|_{K_2}} \vnTime{K_2} = w^- - w^+
&& \oon F \in \mathcal{F}_{K_1}^{\rspace} \cap \mathcal{F}_{K_2}^{\rspace} \text{ for neighboring~$K_1,\, K_2  \in \Th$},
\end{align*}
where~$\vnSpace{K}\in\IR^d$ and~$\vnTime{K} \in \IR$ are, respectively, the space and time components of the outward-pointing unit 
normal vectors on~$\deK\cap\Ftime$ and~$\deK \cap \Fspa$.  The superscripts ``$-$" and ``$+$" are used to indicate the traces of a function on a space-like facet from the elements ``before" ($-$) and ``after" ($+$) such a facet. 

Finally, we define the following space--time broken function spaces of order~$r$:
\begin{equation*}
\begin{split}
H^r(\Th)& :=\{v\in L^2(Q_T) : \; v_{|_{K}}\in H^r(K)\; \forall K\in \Th\} \qquad \text{for } r \in \IR_0^+,\\
\EFC{r}\Th& :=\{v:Q_T\to\IR : \; v_{|_{K}} \in \EFC{r}K\; \forall K\in \Th\} \qquad \text{for }r\in\IN_0. 
\end{split}
\end{equation*}

\subsection{Space--time LDG formulation
\label{SECT::LDG-METHOD}}
Let~$\p = (\pK)_{K \in \Th}$ be a degree vector, which assigns an approximation degree~$\pK \geq 1$ to each element~$K \in \Th$. We consider two generic broken discrete spaces
\begin{equation*}
    \Vp(\Th) := \prod_{K \in \Th} \VpK(K) \quad \text{ and } \quad \Mp(\Th) := \prod_{K \in \Th} \MpK(K),
\end{equation*}
where~$\VpK(K) \subset H^1(K)$ and~$\MpK(K) \subset H^1(K)^d$ are local finite element spaces, which are assumed to satisfy the following condition.
\begin{assumption}[Local compatibility condition]
\label{ASM::COMPATIBILITY-CONDITION}
For each element~$K \in \Th$, the pair~$(\VpK(K), \MpK(K))$ satisfies
\begin{equation*}
\nablax \VpK(K) \subset \MpK(K).
\end{equation*}
\end{assumption}

Our method is derived similarly to the LDG method in~\cite{Castillo_Cockburn_Perugia_Schotzau:2000} for elliptic problems. First, introducing the heat flux density variable~$\q : \QT \rightarrow \IR^d$, defined as~$\q := - \bk \nablax u$, equation~\eqref{EQN::MODEL-PROBLEM-1} can be rewritten as
\begin{equation*}
    \dpt u + \nablax \cdot \q = f \qquad \text{in } \QT.
\end{equation*}

The proposed space--time LDG formulation then reads: find~$(\uh, \qh) \in \Vp(\Th) \times \Mp(\Th)$ such that, for all~$K = \Kx \times \Kt \in \Th$, the following equations are satisfied:
\begin{subequations}
\label{EQN::LDG-K}
\begin{alignat}{3}
\int_{K} \bk^{-1} \qh \cdot \rh \dV & = - \int_{\Kt} \int_{\partial \Kx} \UFLUX \rh \cdot \vnSpace{K} \dS \dt  + \int_{K} \uh \nablax \cdot  \rh \dV & & \qquad   \forall \rh \in \Mp(\Th),\\
\nonumber
\int_K \dpt \uh \vh \dV & + \int_{\partial \Kt} \int_{\Kx} \vh (\UFLUX - \uh) \vnTime{K} \dx \ds  \\
    + \int_{\Kt} \int_{\partial \Kx} & \vh \QFLUX \cdot \vnSpace{K} \dS \dt - \int_K \qh \cdot \nablax \vh \dV   = \int_K f \vh \dV & & \qquad \forall \vh \in \Vp(\Th),
\end{alignat} 
\end{subequations}
where the \emph{numerical fluxes}~$\UFLUX$ and~$\QFLUX$ are, respectively, approximations of the traces of~$u$ and~$\q$ on~$\Fh$. For each mesh facet~$F \in \Fh$, we choose the numerical fluxes as follows:
\begin{align}
\label{EQN::NUMERICAL-FLUXES}
\UFLUX{}_{|_{F}} & := 
\begin{cases}
    \uh^- & \text{ if } F\in \Fspa, \\
    \uh & \text{ if } F \in \FT,\\
    u_0 & \text{ if } F \in \FO,\\
    \mvl{\uh}_{\alpha_F}
    & \text{ if } F \in \Ftime, \\
    \gD & \text{ if } F \in \FD,
\end{cases} \qquad \QFLUX{}_{|_{F}} := 
\begin{cases}
    \mvl{\qh}_{1 - \alpha_F} + \eta_F \jump{\uh}_{\bN} & \text{ if } F \in \Ftime, \\
    \qh + \eta_F (\uh - \gD) \bnOmega & \text{ if } F \in \FD,
\end{cases}
\end{align}
where~$\alpha_F \in [0, 1]$, 
and~$\eta_F \in L^{\infty}(\Ftime\cup\FD)$ is a stabilization function with~$\essinf_{\Ftime \cup \FD} \eta_F > 0$. 
The choices~$\alpha_F = 0$ or~$\alpha_F = 1$ for all time-like facets correspond to directional numerical fluxes, which are expected to increase the sparsity of the matrix associated with the LDG discretization of the spatial gradient operator~$\nablax (\cdot)$, provided that the unit normal vectors~$\nF$ are chosen appropriately; cf. \cite{Castillo_2002-stencil}.
Moreover, weighted averages in the definition of the numerical fluxes can be used to improve the robustness of the method with respect to strong local variations in meshsize, polynomial degree, and diffusion coefficients (see~\cite{Dong_Georgoulis:2022}).

In order to write the space--time LDG formulation~\eqref{EQN::LDG-K} in operator form, we define the following bilinear forms and linear functionals:
\begin{alignat*}{3}
\mh(\uh, \vh) & := \sum_{K \in \Th} \int_K \dpt \uh \vh \dV - \int_{\Fspa} \vh^+ \jump{\uh}_t \dx + \int_{\FO} \uh \vh \dx & & \quad \forall (\uh, \vh) \in \Vp(\Th) \times \Vp(\Th), \\
\bdh(\qh, \rh) & := \sum_{K \in \Th}  \int_K \bk^{-1} \qh \cdot \rh \dV & & \quad \forall (\qh, \rh) \in \Mp(\Th) \times \Mp(\Th), \\
\bh(\uh, \rh) & := \sum_{K \in \Th} \int_K \nablax \uh \cdot \rh \dV \\
& \quad -\int_\Ftime \jump{\uh}_{\bN} \cdot \mvl{\rh}_{1 - \alpha_F}  \dS - \int_{\FD} \uh \rh \cdot \vnSpace{\Omega} \dS  
& & \quad  \forall (\uh, \rh) \in \Vp(\Th) \times \Mp(\Th), \\
\suh(\uh, \vh) & := \int_{\Ftime} \eta_F \jump{\uh}_{\bN} \cdot \jump{\vh}_{\bN} \dS + \int_{\FD} \eta_F \uh \vh  \dS & & \quad \forall (\uh, \vh) \in \Vp(\Th) \times \Vp(\Th), \\
\lqh(\rh) & := -\int_{\FD} \gD \rh \cdot \vnSpace{\Omega} \dS  & & \quad  \forall \rh \in \Mp(\Th), \\
\luh(\vh) & := \int_{\QT} f \vh \dV + \int_{\FD} \eta_F \gD \vh \dS + \int_{\FO} u_0 \vh \dx & & \quad  \forall \vh \in \Vp(\Th).
\end{alignat*}

Substituting the definition of the numerical fluxes into~\eqref{EQN::LDG-K}, summing up over all the elements~$K \in \Th$, and using the average-jump identity
$$\mvl{\vh}_{\alpha_F} \jump{\rh}_{\bN} + \mvl{\rh}_{1 - \alpha_F} \cdot \jump{\vh}_{\bN} = \jump{\vh \rh}_{\bN},$$
the following variational problem is obtained: find~$(\uh, \qh) \in \Vp(\Th) \times \Mp(\Th)$ such that
\begin{subequations}
\label{EQN::VARIATIONAL-DG}
\begin{alignat}{3}
\label{EQN::VARIATIONAL-DG-1}
\bdh(\qh, \rh)
+ \bh(\uh, \rh)
& =  \lqh(\rh) & & \qquad \forall \rh \in \Mp(\Th), \\ 
\label{EQN::VARIATIONAL-DG-2}
 \mh(\uh, \vh) - \bh(\vh, \qh) + \suh(\uh, \vh) & = \luh(\vh)  & & \qquad \forall \vh \in \Vp(\Th).
\end{alignat}
\end{subequations}
\begin{remark}[Flux formulation]
The variational formulation~\eqref{EQN::LDG-K} can be seen as an extension of the unified DG flux formulation for elliptic problems in~\cite{Arnold-Brezzi-Cockburn-Marini:2001} to the parabolic problem in~\eqref{EQN::MODEL-PROBLEM}. 
Our choice in~\eqref{EQN::NUMERICAL-FLUXES} corresponds to using upwind fluxes for~$\uh$ on the space-like facets (in accordance with causality in time), and standard LDG numerical fluxes based on weighted averages on the time-like facets for both~$\uh$ and~$\qh$ (see, e.g., \cite[\S2.1]{Castillo_Cockburn_Perugia_Schotzau:2000} and~\cite[\S3.2]{Castillo_Sequeira:2013}).
 The space--time interior-penalty DG method in~\cite{Cangiani_Dong_Georgoulis:2017} can be recovered by setting $\alpha_F = 1/2$ in the definition of~$\UFLUX$ in~\eqref{EQN::NUMERICAL-FLUXES} for all 
the internal time-like facets, and defining~$\QFLUX$ as
\begin{equation*}
    \QFLUX{}_{|F} := \begin{cases}
    -\mvl{\bk \nablax \uh} + \eta_F \jump{\uh}_\bN & \text{if } F \in \Ftime, \\
    -\bk \nablax \uh + \eta_F (\uh - \gD) \vnOmega & \text{if } F \in \FD.
    \end{cases}
\end{equation*}
\eremk
\end{remark}
\subsection{Reduced formulation \label{SECT::REDUCED-FORMULATION}}
For the analysis, we rewrite the space--time LDG formulation~\eqref{EQN::VARIATIONAL-DG} as one involving only the primal unknown~$\uh$. 
We first define the lifting operator~$\Lu : \Vp(\Th) \rightarrow \Mp(\Th)$ and the LDG spatial gradient~$\nablaxLDG : \Vp(\Th) \rightarrow \Mp(\Th)$ for all~$\vh \in \Vp(\Th)$ as follows:\footnote{The lifting operator~$\Lu (\cdot)$ and the LDG spatial gradient~$\nablaxLDG (\cdot)$ can also be defined for functions in the space~$Y = L^2(0, T; H_0^1(\Omega))$. 
Any function~$v$ in such a space satisfies~$\jump{v}_{\sf N} = 0$ 
on~$\Ftime$ and~$v = 0$ on~$\FD$ almost everywhere. Therefore, for all~$v \in Y$, $\Lu v = 0$ and~$\nablaxLDG v = \nablax v$.}
\begin{equation}
\label{EQN::LIFTING-Uh}
\int_{\QT} \Lu \vh \cdot \rh \dV  = \int_\Ftime \jump{\vh}_{\bN} \cdot \mvl{\rh}_{1 - \alpha_F} \dS + \int_{\FD} \vh \rh \cdot \vnOmega \dS \qquad \forall \rh \in \Mp(\Th),
\end{equation}
and
\begin{equation}
\label{EQN::LDG-GRADIENT}
(\nablaxLDG \vh){}_{|_{K}} := \nablax \vh{}_{|_K} - (\Lu \vh)_{|_K} \qquad \forall K \in \Th.
\end{equation}

The following identity then follows from equation~\eqref{EQN::VARIATIONAL-DG-1}:
\begin{equation*}
\sum_{K \in \Th} \int_K \bk^{-1} \qh \cdot \rh \dV = -\sum_{K \in \Th}  \int_K \nablaxLDG \uh \cdot \rh \dV - \int_{\FD} \gD \rh \cdot \vnOmega \dS \qquad \forall \rh \in \Mp(\Th).
\end{equation*}
Consequently, for all~$\vh \in \Vp(\Th)$, we have
\begin{align*}
\bh(\vh, \qh) & \ = \sum_{K \in \Th} \int_K  \nablaxLDG \vh \cdot \qh \dV = \sum_{K \in \Th} \int_K \bk \nablaxLDG \vh \cdot \bk^{-1} \qh \dV \\
& 
= - \sum_{K \in \Th} \int_K \bk \nablaxLDG \uh \cdot \nablaxLDG \vh \dV - \int_{\FD} \gD \bk \nablaxLDG \vh \cdot \vnSpace{\Omega} \dS.
\end{align*}
Therefore, the mixed variational formulation~\eqref{EQN::VARIATIONAL-DG} reduces to: find~$\uh \in \Vp(\Th)$ such that
\begin{equation}
\label{EQN::REDUCED-VARIATIONAL-DG}
    \Bh(\uh, \vh):= \mh(\uh, \vh) + \Ah(\uh, \vh) = \ell_h(\vh) \qquad \forall \vh \in \Vp(\Th),
\end{equation}
where~$\Ah : \Vp(\Th) \times \Vp(\Th) \rightarrow \IR$ is the bilinear form associated with the LDG discretization of the spatial 
operator~$-\nablax \cdot (\bk \nabla (\cdot))$ given by
\begin{align}
\label{DEF::Ah}
\Ah(\uh, \vh) & : = \sum_{K \in \Th} \int_K \bk \nablaxLDG \uh 
 \cdot \nablaxLDG \vh \dV + \int_{\Ftime} \eta_F \jump{\uh}_{\bN} \cdot \jump{\vh}_{\bN} \dS + \int_{\FD} \eta_F \uh \vh \dS,
\end{align} 
and~$\ell_h : \Vp(\Th) \rightarrow \IR$ is the linear functional
\begin{equation*}
\label{DEF::lh}
    \ell_h(\vh) := \int_{\QT} f \vh \dV + \int_{\FO} u_0 \vh \dx + \int_{\FD} \gD( \eta_F \vh - \bk \nablaxLDG \vh \cdot \vnSpace{\Omega}) \dS. 
\end{equation*}

\begin{remark}[Lifting operator]
The lifting operator~$\Lu(\cdot)$ is just a tool for the analysis of the method, so it does not need to be implemented.
\eremk
\end{remark}

\begin{remark}[Reduced formulation]
The reduced formulation~\eqref{EQN::REDUCED-VARIATIONAL-DG} allows us to carry out an~$hp$-\textit{a priori} error analysis in the spirit of the one in~\cite{Perugia_Schotzau_2002} for elliptic problems.
\eremk
\end{remark}

\section{Inf-sup stability of the method\label{SECT::WELL-POSEDNESS}}
For the sake of simplicity, we henceforth assume homogeneous Dirichlet boundary conditions~$(\gD = 0)$. 

Although the coercivity in Lemma~\ref{PROP::COERCIVITY-Bh} of the bilinear form~$\Bh(\cdot, \cdot)$ is enough to guarantee the existence and uniqueness of a discrete solution, at least for pure Dirichlet boundary conditions (see Remark~\ref{REM::BOUNDARY-CONDITIONS} for other boundary conditions), 
the inf-sup theory in this section allows us to prove continuous dependence on the data and \emph{a priori} error estimates in stronger norms.

This section is devoted to studying the stability properties of the proposed method. 
More precisely, we show two different discrete inf-sup 
conditions, which ensure the existence and uniqueness of a solution to the space--time LDG formulation~\eqref{EQN::REDUCED-VARIATIONAL-DG}.
In Section~\ref{SECT::INF-SUP-NEWTON}, we prove an  inf-sup stability estimate that is valid for any choice of the discrete spaces satisfying the local compatibility condition in Assumption~\ref{ASM::COMPATIBILITY-CONDITION}, whereas, in Section~\ref{SECT::INF-SUP-POLYNOMIAL}, we present an inf-sup stability estimate that is valid for piecewise polynomial spaces that satisfy the following additional mild condition.
\begin{assumption}[Local inclusion condition]
\label{ASM::TIME-DERIVATIVE-ASSUMPTION}
For every~$K \in \Th$, the local discrete space~$\VpK(K)$ is closed under first-order differentiation in time, i.e.,
\begin{equation}
\label{EQN::DPT-CONDITION}
\dpt \VpK(K) \subset \VpK(K) \subset \Pp{\pK}{K}.
\end{equation}
\end{assumption}
\noindent This condition endows the method with additional stability properties.

We introduce the upwind-jump functional
\begin{equation}
\label{EQN::JUMP-FUNCTIONAL}
\SemiNorm{v}{\J}^2 := \frac12 \big(\Norm{v}{L^2(\FT)}^2 + \Norm{\jump{v}_t}{L^2(\Fspa)}^2 + \Norm{v}{L^2(\FO)}^2\big),
\end{equation}
and consider the following mesh-dependent seminorm in~$\Vp(\Th) + Y$:
\begin{align}
\label{EQN::LDG-NORM}
\Tnorm{v}{\LDG}^2 & := \Norm{\sqrt{\bk} \nablaxLDG v}{L^2(\QT)^d}^2 + \Norm{\eta_F^{\frac12}\jump{v}_{\bN}}{L^2(\Ftime)^{d}}^2 + \Norm{\eta_F^{\frac12} v}{L^2(\FD)}^2,
\end{align}
where~$\sqrt{\bk}$ is the symmetric positive definite matrix such that~$\sqrt{\bk} \cdot \sqrt{\bk} = \bk$.

\begin{lemma}[LDG norm]
\label{LEMMA::LDG-NORMS}
The LDG seminorm in~\eqref{EQN::LDG-NORM} is a norm in the space~$\Vp(\Th) + Y$.
\end{lemma}
\begin{proof}
Let~$v \in \Vp(\Th) + Y$ such that~$\Tnorm{v}{\LDG} = 0$.
Then, the spatial normal jump~$\jump{v}_{\sf N} = 0$ on~$\Ftime$ and~$v = 0$ on~$\FD$ almost everywhere, which imply
that~$v \in Y$ and~$\Lu v = 0$. Moreover, the nondegeneracy condition in~\eqref{EQN::DIFFUSION} of the diffusion tensor~$\bk$ implies that~$\Norm{\nablax v}{L^2(\QT)^d} = 0$. Therefore, since~$v = 0$ on~$\FD$, we conclude that~$v = 0$. 
\end{proof}

We conclude this section proving some properties of the discrete bilinear forms~$\Ah(\cdot, \cdot)$ and~$\Bh(\cdot, \cdot)$.

\begin{lemma}[Coercivity and continuity of~$\Ah$]\label{LEMMA::COERCIVITY-CONTINUITY-Ah}
For all~$u, v \in \Vp(\Th) + Y$, it holds
\begin{subequations}
\begin{align}
\label{EQN::COERCIVITY-Ah}
    \Ah(u, u) = & \ \Tnorm{u}{\LDG}^2, \\
\label{EQN::CONTINUITY-Ah}
\Ah(u, v) \leq &\  \Tnorm{u}{\LDG} \Tnorm{v}{\LDG}.
\end{align}
\end{subequations}
\end{lemma}
\begin{proof}
The coercivity identity~\eqref{EQN::COERCIVITY-Ah} is an immediate consequence of the definition of~$\Ah(\cdot, \cdot)$ in~\eqref{DEF::Ah}, whereas the continuity bound~\eqref{EQN::CONTINUITY-Ah} follows by using the Cauchy--Schwarz inequality.
\end{proof}

\begin{lemma}[Coercivity of~$\Bh$] 
\label{PROP::COERCIVITY-Bh}
The bilinear form~$\Bh(\cdot, \cdot)$ satisfies the following identity:
\begin{equation*}
\Bh(v, v) = \SemiNorm{v}{\J}^2 + \Tnorm{v}{\LDG}^2 \qquad \qquad  \hspace*{\fill}  \forall v \in \Vp(\Th) + H^1(\Th).
\end{equation*}
\end{lemma}
\begin{proof}
Integration by parts in time and the identity~$\frac12 \jump{v^2}_t - v^+ \jump{\vh}_t = \frac12 \jump{v}_t^2$ on~$\Fspa$ give
\begin{align}
\nonumber
\mh(v, v) & = \sum_{K \in \Th} \int_K \dpt v v \dV - \int_{\Fspa} v^+ \jump{v}_t \dx + \Norm{v}{L^2(\FO)}^2 \\
\nonumber
& = \frac{1}{2} \Big(\Norm{v}{L^2(\FT)}^2 + \int_{\Fspa} \jump{v^2}_t \dx - \Norm{v}{L^2(\FO)}^2 \Big) - \int_{\Fspa} v^+ \jump{v}_t \dx + \Norm{v}{L^2(\FO)}^2 \\
\label{EQN::IDENTITY-Mh}
& = \SemiNorm{v}{\J}^2,
\end{align}
which, combined with the coercivity identity~\eqref{EQN::COERCIVITY-Ah} for~$\Ah(\cdot, \cdot)$, completes the proof.
\end{proof}

\subsection{First inf-sup stability estimate \label{SECT::INF-SUP-NEWTON}}
We prove that a discrete inf-sup stability estimate holds for any choice of the discrete spaces~$\Vp(\Th) \times \Mp(\Th)$ satisfying the local compatibility condition in Assumption~\ref{ASM::COMPATIBILITY-CONDITION}.

We introduce a discrete Newton potential operator~$\Nh : \Vp(\Th) + H^1(\Th) \rightarrow \Vp(\Th)$, which we define for any~$v \in \Vp(\Th) + H^1(\Th)$ as the solution to the following variational problem:
\begin{equation}
\label{EQN::DISCRETE-NEWTON}
    \Ah( \Nh v, \wh) = \mh(v, \wh)  \qquad \forall \wh \in \Vp(\Th).
\end{equation} 
The variational problem~\eqref{EQN::DISCRETE-NEWTON} has a unique solution due to Lemma~\ref{LEMMA::LDG-NORMS} and the coercivity of the bilinear form~$\Ah(\cdot, \cdot)$ in Lemma~\ref{LEMMA::COERCIVITY-CONTINUITY-Ah}.

We define the following mesh-dependent norm in~$\Vp(\Th) + H^1(\Th)\cap Y$:
\begin{align}
\label{EQN::DG-NORMS-3}
\Tnorm{v}{\LDGN}^2 & := \SemiNorm{v}{\J}^2 + \Tnorm{v}{\LDG}^2  + \Tnorm{\Nh v}{\LDG}^2.
\end{align}

Next theorem shows that the bilinear form~$\Bh(\cdot, \cdot)$ is inf-sup stable for any choice of the discrete spaces satisfying the local compatibility condition in Assumption~\ref{ASM::COMPATIBILITY-CONDITION}.
\begin{theorem}[Inf-sup stability] \label{THM::INF-SUP-NEWTON}
For any discrete spaces~$(\Vp(\Th), \Mp(\Th)
)$ satisfying Assumption~\ref{ASM::COMPATIBILITY-CONDITION}, it holds
\begin{align}
\label{EQN::INF-SUP-NEWTON}
\frac{1}{2\sqrt{2}}\Tnorm{\uh}{\LDGN} & \leq \sup_{\vh \in \Vp(\Th) \setminus \{0\}} \frac{\Bh(\uh, \vh)}{\Tnorm{\vh}{\LDG}} \qquad \forall \uh \in \Vp(\Th).
\end{align}
\end{theorem}
\begin{proof}
Let~$\uh \in \Vp(\Th)$ and~$\wh := \Nh \uh$. 
By the triangle inequality and the definition of~$\Tnorm{\cdot}{\LDGN}$ in~\eqref{EQN::DG-NORMS-3}, we have
\begin{equation}
\label{EQN::Cauchy--Schwarz-UH-WH}
\Tnorm{\uh + \wh}{\LDG} \leq \sqrt{2}\left(\Tnorm{\uh}{\LDG}^2 + \Tnorm{\wh}{\LDG}^2\right)^{\frac12}\leq \sqrt{2} \Tnorm{\uh}{\LDGN}.
\end{equation}
Moreover, using identity~\eqref{EQN::IDENTITY-Mh}, the definition of the discrete Newton potential in~\eqref{EQN::DISCRETE-NEWTON}, the coercivity and continuity bounds in Proposition~\ref{LEMMA::COERCIVITY-CONTINUITY-Ah} for~$\Ah(\cdot, \cdot)$, and the Young inequality, we get
\begin{align*}
\Bh(\uh, \uh + \wh) & = \Bh(\uh, \uh) + \mh(\uh, \wh) + \Ah(\uh, \wh)  \\
& = \SemiNorm{\uh}{\J}^2  + \Tnorm{\uh}{\LDG}^2 + \Ah(\wh, \wh) + \Ah(\uh, \wh) \\
& \geq \SemiNorm{\uh}{\J}^2  + \Tnorm{\uh}{\LDG}^2 + \Tnorm{\wh}{\LDG}^2 - \Tnorm{\uh}{\LDG} \Tnorm{\wh}{\LDG} \\
& \geq \SemiNorm{\uh}{\J}^2  + \frac12 \left(\Tnorm{\uh}{\LDG}^2 + \Tnorm{\wh}{\LDG}^2\right).
\end{align*}
Then, using bound~\eqref{EQN::Cauchy--Schwarz-UH-WH}, we deduce that
\begin{align*}
\Bh(\uh, \uh + \wh) \geq &\ \frac12 \Tnorm{\uh}{\LDGN}^2 \geq \frac{1}{2\sqrt{2}} \Tnorm{\uh}{\LDGN} \Tnorm{\uh + \wh}{\LDG},
\end{align*}
which completes the proof of~\eqref{EQN::INF-SUP-NEWTON}.
\end{proof}

The following result is an immediate consequence of the inf-sup stability estimate in Theorem~\ref{THM::INF-SUP-NEWTON}.
\begin{corollary}[Existence and uniqueness of a discrete solution]
\label{COR::WELL-POSEDNESS}
For any discrete spaces~$(\Vp(\Th),\Mp(\Th))$ satisfying Assumption~\ref{ASM::COMPATIBILITY-CONDITION}, there exists a unique solution~$(\uh, \qh) \in \Vp(\Th) \times \Mp(\Th)$ to the space--time LDG variational formulation~\eqref{EQN::VARIATIONAL-DG}. 
\end{corollary}

We conclude this section by showing a continuity bound for the bilinear form~$\Bh(\cdot, \cdot)$.
\begin{lemma}[Continuity of $\Bh$] \label{PROP::CONT-NEWTON}
For all~$v \in \Vp(\Th) + H^1(\Th)$ and~$\vh\in\Vp(\Th)$, 
it holds
\begin{equation*}
\Bh(v, \vh) \leq\sqrt{2}\Tnorm{v}{\LDGN}\Tnorm{\vh}{\LDG}.
\end{equation*}
\end{lemma}
\begin{proof}
Let~$v\in \Vp(\Th) + H^1(\Th)$ and~$\vh\in \Vp(\Th)$.
Using the definition of the bilinear form~$\Bh(\cdot, \cdot)$, the definition of the discrete Newton potential in~\eqref{EQN::DISCRETE-NEWTON}, the continuity of~$\Ah(\cdot, \cdot)$ in~\eqref{EQN::CONTINUITY-Ah}, and the Cauchy--Schwarz inequality, we obtain
\begin{align*}
\Bh(v, \vh) = \mh(v, \vh)+\Ah(v, \vh)=\Ah( \Nh v, \vh)+\Ah(v, \vh) 
\leq \sqrt{2}\Tnorm{v}{\LDGN} \Tnorm{\vh}{\LDG}.
\end{align*} 
\end{proof}

\begin{remark}[Stability parameter~$\eta_u$] 
The space--time interior-penalty DG method in~\cite{Cangiani_Dong_Georgoulis:2017} requires a ``sufficiently large" stabilization parameter in order to guarantee well posedness,
whereas the proposed method requires only that~$\essinf_{\Ftime \cup \FD} \eta_F > 0$.
This is to be expected, as the stability term is related only to the discretization of the spatial operator~$-\nabla(\bk \nabla (\cdot))$.
\eremk
\end{remark}

\subsection{Second inf-sup stability estimate
\label{SECT::INF-SUP-POLYNOMIAL}}
We now consider piecewise polynomial spaces satisfying Assumptions~\ref{ASM::COMPATIBILITY-CONDITION} and~\ref{ASM::TIME-DERIVATIVE-ASSUMPTION}.
In order to avoid the restriction of a uniformly bounded number of time-like facets for the elements in~$\Th$, 
we make the following assumption, which extends the one in~\cite[Asm. 2.1]{Cangiani_Dong_Georgoulis:2017} to general prismatic space--time meshes.

\begin{assumption}[Mesh assumption]
\label{ASSUMPTION::PRISMATIC-MESH} If~$d > 1$, we assume that, for any~$K = \Kx \times \Kt \in \Th$, the boundary~$\partial \Kx$ can be subtriangulated into nonoverlapping~$(d - 1)$-dimensional simplices~$\mathcal{T}_{\partial \Kx} := \{\Fx^i\}_{i = 1}^n$, with~$n \in \IN$. Moreover, there exists a set of nonoverlapping~$d$-dimensional simplices~$\{s_K^{\Fx^i}\}_{i = 1}^n$ contained in~$\Kx$ such that, for~$i = 1, \dots, n$,
$\partial s_K^{\Fx^i} \cap \partial \Kx = \Fx^i$ and
\begin{equation*}
\hKx \leq C_s \frac{d \abs{s_{K}^{\Fx^i}}}{\abs{\Fx^i}},
\end{equation*}
for some constant~$C_s > 0$ independent of the discretization parameters, the number of time-like facets per element, and the facet measures. 
\end{assumption}

Fixing a constant~$\eta^\star > 0$, for all time-like facets~$F \in \Ftime \cup \FD$, we define the stabilization function as follows:
\begin{equation}
\label{EQN::STABILIZATION-TERM}
    \eta_F := \eta^\star \max_{
        K \in \Th : F \in \FKtime
        } \left\{\frac{\|\sqrt{\bk}_{|K}\|_{\IR^{d \times d}}^2(\pK + 1)(\pK + d)}{\hKx}\right\},
\end{equation}
where~$\|\cdot\|_{\IR^{d\times d}}$ is the~$2$-norm.

For each~$K \in \Th$, we define the following constants:
\begin{subequations}
\begin{align}
\hhKt & := \min \{h_{K'_t} \ : \ K' \in \Th \text{ and } \FKtime \cap \calF_{K'}^{\rtime} \neq \emptyset \},\\
\hpK & := \max \{p_{K'} \ : \ K' \in \Th \text{ and } \FKtime \cap \calF_{K'}^{\rtime} \neq \emptyset \}, \\
\label{EQN::LambdaK}
\lK & := \hhKt/\hpK^2. 
\end{align}
\end{subequations}
For convenience, we denote by~$\lambda_h$ the piecewise constant function defined as
$$\lambda_h{}_{|_K} = \lK \qquad  \forall K \in \Th.$$

Then, we define the following norm in~$\Vp(\Th) + H^1(\Th)\cap Y$:
\begin{equation}
\label{EQN::DG-NORMS-4}
\Tnorm{v}{\LDGp}^2  := \Tnorm{v}{\LDG}^2 + \SemiNorm{v}{\J}^2  + \sum_{K \in \Th} \lK \Norm{\dpt v}{L^2(K)}^2.
\end{equation}

Before proving the inf-sup stability estimate in Theorem~\ref{THM::INF-SUP-TWO} below, we recall some useful polynomial trace and inverse inequalities from \cite[\S 4.1]{Cangiani_Dong_Georgoulis:2017}. 

\begin{lemma}[Polynomial trace inequalities] \label{LEMMA::TRACE-INEQUALITY}
Let~$K = \Kx \times \Kt \in \Th$ and let Assumption~\ref{ASSUMPTION::PRISMATIC-MESH} hold. 
For all~$\vh \in \Pp{\pK}{K}$, we have
\begin{equation}
\label{EQN::TRACE-INEQUALITY-1}
\Norm{\vh}{L^2(F^*)}^2 \leq \frac{(\pK+1)(\pK+d)}{d} \frac{\abs{\Fx}}{\abs{\sKxF}} \Norm{\vh}{L^2(\Ft; L^2(\sKxF))}^2 \qquad \forall F^* = \Fx \times \Ft,
\end{equation}
where~$\Ft \subset \Kt$, $\Fx \in \mathcal{T}_{\partial \Kx}$ with $\Fx  = \partial \Kx \cap \partial \sKxF$, and~$\sKxF$ as in Assumption~\ref{ASSUMPTION::PRISMATIC-MESH} sharing~$\Fx$ with~$\Kx$. 
Moreover, there exists a positive constant~$\Ctr$ independent of~$K$ and the degree of approximation~$\pK$ such that
\begin{equation}
\label{EQN::TRACE-INEQUALITY-2}
\Norm{\vh}{L^2(\Kx \times \partial \Kt)}^2 \leq \Ctr \frac{\pK^2}{\hKt} \Norm{\vh}{L^2(K)}^2 \qquad \forall \vh \in \Pp{\pK}{K}.
\end{equation}
\end{lemma}

\begin{lemma}[Polynomial inverse estimates] 
\label{LEMMA::INVERSE-ESTIMATE}
Let~$\Th$ satisfy Assumption~\ref{ASSUMPTION::PRISMATIC-MESH} and let~$K = \Kx \times \Kt \in \Th$. There exists a positive constant~$\Cinv$ independent of~$K$ and the degree of approximation~$\pK$ such that, for all~$\vh \in \Pp{\pK}{K}$, it holds
\begin{subequations}
\begin{align}
\label{EQN::INVERSE-ESTIMATE-1}
\Norm{\dpt \vh}{L^2(K)}^2 & \leq \Cinv \frac{\pK^4}{\hKt^2}\Norm{\vh}{L^2(K)}^2, \\
\label{EQN::INVERSE-ESTIMATE-2}
\Norm{\dpt \vh}{L^2(F)}^2 & \leq \Cinv \frac{\pK^4}{|F_t|^2}\Norm{\vh}{L^2(F)}^2 \qquad \qquad \forall F = \Fx \times \Ft,
\end{align}
\end{subequations}
where~$\Ft \subset \Kt$, $\Fx = \partial \Kx \cap \partial \sKxF \in \mathcal{T}_{\partial \Kx}$, and~$\sKxF$ as in Assumption~\ref{ASSUMPTION::PRISMATIC-MESH}.
\end{lemma}

We are now in a position to prove the main result in this section.
\begin{theorem}[Inf-sup stability]\label{THM::INF-SUP-TWO}
Let~$(\Vp(\Th), \Mp(\Th))$ be discrete spaces satisfying Assumptions~\ref{ASM::COMPATIBILITY-CONDITION} and~\ref{ASM::TIME-DERIVATIVE-ASSUMPTION}, and let~$\Th$ satisfy Assumption~\ref{ASSUMPTION::PRISMATIC-MESH}. Let also the stabilization function~$\eta_F$ be given by~\eqref{EQN::STABILIZATION-TERM}. Then, there exists a positive constant~$\gamma_I$ independent of the meshsize~$h$ and the degree vector~$\p$ such that
\begin{align}
\label{EQN::INF-SUP-DERIVATIVE}
\gamma_I \Tnorm{\uh}{\LDGp} & \leq  \sup_{\vh \in \Vp(\Th) \setminus \{0\}} \frac{\Bh(\uh, \vh)}{\Tnorm{\vh}{\LDGp}} \qquad \forall \uh \in \Vp(\Th).
\end{align}
\end{theorem}
\begin{proof}
Let~$\uh \in \Vp(\Th)$ and~$\delta$ be a positive constant that will be chosen later. We define~$\wh \in \Vp(\Th)$ as~$\wh{}_{|_K} :=  \lK \dpt \uh{}_{|_K}$ with
$\lK$ as in~\eqref{EQN::LambdaK} for all~$K \in \Th$.

By the definition of~$\wh$, we have
\begin{equation}
\label{EQN::IDENTITY-WH-LDG}
\begin{split}
\Tnorm{\wh}{\LDG}^2 & = \sum_{K \in \Th} \lK^2 \Norm{\sqrt{\bk} \nablaxLDG (\dpt \uh)}{L^2(K)^d}^2 + \Norm{\eta_F^{\frac12} \jump{\lambdah \dpt \uh}_{\sf N} }{L^2(\Ftime)^d}^2 + \Norm{\eta_F^{\frac12} \lambdah \dpt \uh}{L^2(\FD)}^2 \\
& =: M_1 + M_2 + M_3.
\end{split}
\end{equation}
We now bound each term~$M_i$, $i = 1, 2, 3$, separately.

\paragraph{Bound on~$M_1$.} Using the commutativity of the LDG spatial gradient~$\nablaxLDG (\cdot)$ and the first-order time derivative~$\dpt(\cdot)$, the polynomial inverse estimate in~\eqref{EQN::INVERSE-ESTIMATE-1}, and the definition of~$\lK$ in~\eqref{EQN::LambdaK}, we get
\begin{equation}
\label{EQN::BOUND-M1}
M_1 \leq \sum_{K \in \Th} \Cinv \lK^2 \frac{\pK^4}{\hKt^2} \Norm{\sqrt{\bk} \nablaxLDG \uh}{L^2(K)^d}^2 \leq \Cinv \Norm{\sqrt{\bk} \nablaxLDG \uh}{L^2(\QT)^d}^2.
\end{equation}

\paragraph{Bound on~$M_2$.} Using the commutativity of the spatial normal jump~$\jump{\cdot}_{\sf N}$ and the first-order time derivative~$\dpt(\cdot)$, the definition of~$\lK$ in~\eqref{EQN::LambdaK}, and the polynomial inverse estimate in~\eqref{EQN::INVERSE-ESTIMATE-2}, we obtain
\begin{equation}
\label{EQN::BOUND-M2}
M_2 \leq \Cinv \!\!\!\! \sum\limits_{
\substack{
F \in \Ftime, \\
F \subset \partial K \cap \partial K'
}
}
\!\!\!\!
\max\{\lK^2, \lambda_{K'}^2\} \frac{\max\{\pK^4, p_{K'}^4\}}{\min\{\hKt^2, h_{K'_t}^2\}} \Norm{\eta_F^{\frac12} \jump{\uh}_{\sf N}}{L^2(F)^d}^2
\leq \Cinv \Norm{\eta_F^{\frac12} \jump{\uh}_{\sf N}}{L^2(\Ftime)^d}^2.
\end{equation}

\paragraph{Bound on~$M_3$.} Similar steps as for the bound on~$M_2$ yield
\begin{equation}
\label{EQN::BOUND-M3}
M_3 \leq \Cinv \Norm{\eta_F^{\frac12} \uh}{L^2(\FD)}^2.
\end{equation}

Therefore, combining bounds~\eqref{EQN::BOUND-M1}, \eqref{EQN::BOUND-M2}, and~\eqref{EQN::BOUND-M3} with identity~\eqref{EQN::IDENTITY-WH-LDG}, we obtain
\begin{equation}
\label{EQN::BOUND-WH-LDG}
\Tnorm{\wh}{\LDG}^2 \leq \Cinv \Tnorm{\uh}{\LDG}^2.
\end{equation}

Using the triangle inequality, the polynomial trace inequality in~\eqref{EQN::TRACE-INEQUALITY-2}, and the definition of~$\lK$ in~\eqref{EQN::LambdaK}, the following bound is obtained:
\begin{align}
\nonumber
\SemiNorm{\wh}{\J}^2 & = \frac12 \Norm{\lambdah \dpt \uh}{L^2(\FT)}^2 + \frac12 \Norm{\jump{\lambdah \dpt \uh}_t}{L^2(\Fspa)}^2 + \frac12 \Norm{\lambdah \dpt \uh}{L^2(\FO)}^2 \\\
\nonumber
& \leq \sum_{
K = \Kx \times \Kt \in \Th
} \Norm{\lK \dpt \uh}{L^2(\Kx \times \partial \Kt)}^2 \leq \Ctr \sum_{
K \in \Th
} \lK^2 \frac{\pK^2}{\hKt} \Norm{\dpt \uh}{L^2(K)}^2 \\
\label{EQN::BOUND-WH-UPWIND}
& \leq \Ctr \sum_{K \in \Th} \lK \Norm{\dpt \uh}{L^2(K)}^2.
\end{align}

Moreover, the polynomial inverse estimate in~\eqref{EQN::INVERSE-ESTIMATE-1} and the definition of~$\lK$ in~\eqref{EQN::LambdaK} lead to
\begin{alignat}{3}
\nonumber
\sum_{K \in \Th} \lK \Norm{\dpt \wh}{L^2(K)}^2 = \sum_{K \in \Th} \lK^3 \Norm{\dptt \uh}{L^2(K)}^2 
& \leq \Cinv \sum_{K \in \Th} \lK^3 \frac{\pK^4}{\hKt^2} \Norm{\dpt \uh}{L^2(K)}^2 \\
\label{EQN::BOUND-WH-TIME-DER}
& \leq \Cinv \sum_{K \in \Th} \lK \Norm{\dpt \uh}{L^2(K)}^2.
\end{alignat}

The following estimate follows by adding bounds~\eqref{EQN::BOUND-WH-LDG}, \eqref{EQN::BOUND-WH-UPWIND}, and~\eqref{EQN::BOUND-WH-TIME-DER}:
\begin{equation*}
\label{EQN::BOUND-WH-LDGp}
\Tnorm{\wh}{\LDGp}^2 \leq 2\max\{\Cinv, \Ctr\} \Tnorm{\uh}{\LDGp}^2,
\end{equation*}
which, combined with the Cauchy--Schwarz inequality, gives
\begin{equation}
\label{EQN::BOUND-Uh-Wh}
\Tnorm{\uh + \delta \wh}{\LDGp}^2 \leq 2 \left(\Tnorm{\uh}{\LDGp}^2 + \delta^2 \Tnorm{\wh}{\LDGp}^2 \right) \leq \mu \Tnorm{\uh}{\LDGp}^2,
\end{equation}
with~$\mu = 2 (1 + 2\delta^2 \max\{\Cinv, \Ctr\})$.

The identity in Lemma~\ref{PROP::COERCIVITY-Bh} yields
\begin{align}
\label{EQN::FIRST-BOUND-Lh}
\Bh(\uh, \uh + \delta \wh) = &\ \SemiNorm{\uh}{\J}^2 + \Tnorm{\uh}{\LDG}^2 + \delta \mh(\uh, \wh) + \delta \Ah(\uh, \wh).
\end{align}
We consider first the third term on the right-hand side of~\eqref{EQN::FIRST-BOUND-Lh}. 
Using the Young inequality, the polynomial trace inequality in~\eqref{EQN::TRACE-INEQUALITY-2}, and the definition of~$\lK$ in~\eqref{EQN::LambdaK}, for all~$\epsilon > 0$, we have
\begin{align}
\nonumber
\mh(\uh, \wh) & = \sum_{K \in \Th} \lK \Norm{\dpt \uh}{L^2(K)}^2
- \int_{\Fspa} \lambdah^+ \dpt\uh^+ \jump{\uh}_t \dx + \int_{\FO} 
\lambdah \uh \dpt \uh \dx \\
\nonumber
& \geq \sum_{K \in \Th} \lK \Norm{\dpt \uh}{L^2(K)}^2 - \frac{\epsilon}{2} \left(
\Norm{\lambdah^+ \dpt \uh^+}{L^2(\Fspa)}^2
+ \Norm{\lambdah \dpt \uh}{L^2(\FO)}^2
\right) 
- \frac{1}{\epsilon} \SemiNorm{\uh}{\J}^2 \\
\nonumber
& \geq \sum_{K \in \Th} \lK \Norm{\dpt \uh}{L^2(K)}^2 - \frac{\epsilon \Ctr}{2} \sum_{K \in \Th} \lK^2 \frac{\pK^2}{\hKt} \Norm{\dpt \uh}{L^2(K)}^2
- \frac{1}{\epsilon} \SemiNorm{\uh}{\J}^2 \\
\label{EQN::SECOND-BOUND-Mh}
& \geq 
\Big(1 - \frac{\epsilon \Ctr}{2}\Big) \sum_{K \in \Th} \lK \Norm{\dpt \uh}{L^2(K)}^2 
- \frac{1}{\epsilon} \SemiNorm{\uh}{\J}^2.
\end{align}

As for the fourth term on the right-hand side of~\eqref{EQN::FIRST-BOUND-Lh}, bound~\eqref{EQN::BOUND-WH-LDG} and the continuity of the bilinear form~$\Ah(\cdot, \cdot)$ in~\eqref{EQN::CONTINUITY-Ah} give
\begin{align}
\label{EQN::BOUND-Ah}
\Ah(\uh, \wh) \geq - \Tnorm{\uh}{\LDG} \Tnorm{\wh}{\LDG} \geq -\Cinv^{\frac12} \Tnorm{\uh}{\LDG}^2.
\end{align}

Combining~\eqref{EQN::FIRST-BOUND-Lh}, \eqref{EQN::SECOND-BOUND-Mh}, and~\eqref{EQN::BOUND-Ah}, we get
\begin{align*}
\Bh(\uh, \uh + \delta \wh) & \geq \Big(1 - \frac{\delta}{\epsilon}\Big)\SemiNorm{\uh}{\J}^2 + \Big(1 - \delta \Cinv^{\frac12}\Big) \Tnorm{\uh}{\LDG}^2 + \delta \Big(1 - \frac{\epsilon \Ctr}{2}\Big) \sum_{K \in \Th} \lK \Norm{\dpt \uh}{L^2(K)}^2.
\end{align*}
Therefore, choosing~$\delta$ and~$\epsilon$ such that
\begin{equation*}
0 < \epsilon < \frac{2}{\Ctr}\quad \text{ and } \quad 0 < \delta < \min\left\{\epsilon,\ \Cinv^{-\frac12} \right\},
\end{equation*}
defining
\begin{equation*}
  \beta := \min\left\{1 - \frac{\delta}{\epsilon}, \ 1 - \delta \Cinv^{\frac12},\ \delta \Big(1 - \frac{\epsilon \Ctr}{2}\Big)\right\},
\end{equation*}
and using bound~\eqref{EQN::BOUND-Uh-Wh}, we obtain
\begin{equation*}
\Bh(\uh, \uh + \delta \wh) \geq \beta \Tnorm{\uh}{\LDGp}^2 \geq \frac{\beta}{\sqrt{\mu}} \Tnorm{\uh}{\LDGp} \Tnorm{\uh + \delta \wh}{\LDGp},
\end{equation*}
which completes the proof of~\eqref{EQN::INF-SUP-DERIVATIVE} with~$\gamma_I = \beta/\sqrt{\mu}$.
\end{proof}

For the convergence analysis in Section~\ref{SECT::CONVERGENCE-ANALYSIS}, it is useful to introduce the following auxiliary norms:
\begin{subequations}
\begin{align}
\label{EQN::DG-NORMS-5}
\Tnorm{v}{\LDGs}^2 & := \Tnorm{v}{\LDG}^2+\sum_{K \in \Th} \lK^{-1} \Norm{v}{L^2(K)}^2 + \Norm{v^{-}}{L^2(\Fspa)}^2  +
\Norm{v}{L^2(\FT)}^2,\\
\label{EQN::DG-NORMS-6}
\Tnorm{v}{\LDGss}^2 & := \Tnorm{v}{\LDG}^2 + \sum_{K \in \Th} \lK^{-1} \Norm{v}{L^2(K)}^2 + \Norm{v^{+}}{L^2(\Fspa)}^2  +
\Norm{v}{L^2(\FO)}^2. 
\end{align}
\end{subequations}

In next lemma, we prove two continuity bounds for the discrete bilinear form~$\Bh(\cdot, \cdot)$.
\begin{lemma}[Continuity of $\Bh$] \label{PROP::CONT-LDGP}
\begin{subequations}
For all~$v, w \in \Vp(\Th) + H^1(\Th)$, 
the following continuity bounds hold:
\begin{align}
\label{EQN::CONT-LDGP}
\Bh(v, w) & \leq \sqrt{2} \Tnorm{v}{\LDGs}\Tnorm{w}{\LDGp}, \\
\label{EQN::CONT-LDGP-2}
\Bh(v, w) & \leq \sqrt{2} \Tnorm{v}{\LDGp}\Tnorm{w}{\LDGss}.
\end{align}
\end{subequations}
\end{lemma}
\begin{proof}
Let~$v, w\in \Vp(\Th) + H^1(\Th)$.
Using the definition of the bilinear form~$\Bh(\cdot, \cdot)$ and the continuity of~$\Ah(\cdot, \cdot)$ in~\eqref{EQN::CONTINUITY-Ah}, we obtain
\begin{equation}
\label{EQN::AUX-BOUND-Bh}
\Bh(v, w)=\mh(v, w)+\Ah(v, w)\leq
\mh(v, w)+\Tnorm{v}{\LDG}\Tnorm{w}{\LDG}.
\end{equation}
Integration by parts in time, the identity~$v^+ \jump{w}_t +w^-\jump{v}_t = \jump{v w}_t$ on~$\Fspa$, and the Cauchy--Schwarz inequality yield
\begin{align}
\nonumber
	\mh(v, w)=&\sum_{K \in \Th} \int_K \dpt v w \dV - \int_{\Fspa} w^+ \jump{v}_t \dx + \int_{\FO} v w \dx \\ 
 \nonumber
 =&
    -\sum_{K \in \Th} \int_K  v \dpt w \dV +\int_{\Fspa} (\jump{v w}_t - w^+ \jump{v}_t ) \dx + \int_{\FT} v w \dx  
	 \\
\nonumber
=&
   -\sum_{K \in \Th} \int_K  v \dpt w \dV +\int_{\Fspa} v^- \jump{w}_t \dx +\int_{\FT} v w \dx\\ 
\nonumber
\leq &
	\bigg(\sum_{K \in \Th} \lK^{-1} \Norm{v}{L^2(K)}^2 \bigg)^{\frac12} \bigg(\sum_{K \in \Th} \lK \Norm{\dpt w}{L^2(K)}^2 \bigg)^{\frac12} + \Norm{v^-}{L^2(\Fspa)} \Norm{\jump{w}_t}{L^2(\Fspa)}\\
 \nonumber
 & + 
\Norm{v}{L^2(\FT)}\Norm{w}{L^2(\FT)},
\end{align}
which, combined with~\eqref{EQN::AUX-BOUND-Bh} and the Cauchy--Schwarz inequality, gives~\eqref{EQN::CONT-LDGP}.
The proof of bound~\eqref{EQN::CONT-LDGP-2} is similar, so we omit it here.
\end{proof}

\begin{remark}[More general boundary conditions]
\label{REM::BOUNDARY-CONDITIONS}
In Lemma~\ref{LEMMA::LDG-NORMS}, we have proven that~$\Tnorm{\cdot}{\LDG}$ is a norm on the space~$\Vp(\Th) + Y$, which holds only for pure Dirichlet boundary conditions.
For such boundary conditions, one could deduce existence and uniqueness of a solution to~\eqref{EQN::REDUCED-VARIATIONAL-DG} directly from the coercivity identity in Lemma~\ref{PROP::COERCIVITY-Bh} for the bilinear form~$\Bh(\cdot, \cdot)$.
On the other hand, if pure Neumann or mixed boundary conditions were to be considered, then the corresponding seminorm would not be a norm, and the discrete Newton potential in~\eqref{EQN::DISCRETE-NEWTON} would not be well defined without further modifications.
Nonetheless, the norm~$\Tnorm{\cdot}{\LDGp}$ defined in~\eqref{EQN::DG-NORMS-4} would be a norm in any case, so the inf-sup condition in Theorem~\ref{THM::INF-SUP-TWO} guarantees the well-posedness of the method also for pure Neumann or mixed boundary conditions.
\eremk
\end{remark}

\section{Convergence analysis\label{SECT::CONVERGENCE-ANALYSIS}}
This section is devoted to deriving error bounds in some energy norms (Section~\ref{SECT::A-PRIORI-ENERGY})
that can be used to obtain \emph{a priori} error estimates for different discrete spaces; see Section~\ref{SECT::DISCRETE-SPACES} below.
In Section~\ref{SECT::A-PRIORI-L2}, we discuss the difficulties of deriving error estimates in the mesh-independent norm~$L^2(\QT)$.

\subsection{\emph{A priori} error bounds in energy norms \label{SECT::A-PRIORI-ENERGY}}
We define the inconsistency bilinear form~$\Rh : H^{\frac32 + \varepsilon}(\Th) \times (\Vp(\Th) + Y) \rightarrow \IR$ ($\varepsilon > 0$) as follows:
\begin{equation}
\label{EQN::DEF-Rh}
\Rh(u, v) := \int_{\Ftime} \mvl{\bk \nablaxh u - \bk \PiO \nablaxh u}_{1 - \alpha_F} \cdot \jump{v}_{\sf N} \dS + \int_{\FD} v (\bk \nablaxh u - \bk \PiO\nablaxh u) \cdot \bnOmega \dS,
\end{equation}
with~$\nablaxh$ and~$\PiO$ denoting the piecewise spatial gradient in~$\Th$ and the~$L^2(\QT)^d$-orthogonal projection in~$\Mp(\Th)$, respectively.

If the solution~$u$ to the continuous weak formulation in~\eqref{EQN::CONTINUOUS-WEAK-FORMULATION} belongs to~$H^{\frac32 + \varepsilon}(\Th) \cap X$ for some~$\varepsilon > 0$, then the following identity can be proven using integration by parts in space and the definition of the lifting operator~$\Lu$:
\begin{equation}
\label{EQN::IDENTITY-Rh}
\Rh(u, \vh) = \Bh(u - \uh, \vh) = \Bh(u, \vh) - \ell_h(\vh) \qquad \forall \vh \in \Vp(\Th).
\end{equation}

We are now in a position to show some \emph{a priori} bounds for the errors in the norms~$\Tnorm{\cdot}{\LDGN}$ and~$\Tnorm{\cdot}{\LDGp}$, which are defined in~\eqref{EQN::DG-NORMS-3} and~\eqref{EQN::DG-NORMS-4}, respectively.
\begin{theorem}[\emph{A priori} error bounds in the energy norms] \label{THM::ERROR-BOUND-ENERGY}
Let the discrete spaces~$(\Vp(\Th), \Mp(\Th))$ satisfy the compatibility condition in Assumption~\ref{ASM::COMPATIBILITY-CONDITION}.
Let~$u \in H^{\frac32 + \varepsilon}(\Th) \cap X$ for some~$\varepsilon > 0$ be the solution to the continuous weak formulation in~\eqref{EQN::CONTINUOUS-WEAK-FORMULATION}, and~$\uh \in \Vp(\Th)$ be the unique solution to the space--time LDG variational formulation~\eqref{EQN::REDUCED-VARIATIONAL-DG}. Then, the following bound holds:
\begin{subequations}
\begin{equation}
\label{EQ::STRANG_NEWTON}
\Tnorm{u - \uh}{\LDGN} \leq 5\inf_{\vh \in \Vp(\Th)}\Tnorm{u - \vh}{\LDGN} + 
2\sqrt{2}
\sup_{\vh \in \Vp(\Th)\setminus \{0\}} \frac{|\Rh(u, \vh)|}{\Tnorm{\vh}{\LDG}}.
\end{equation}
Moreover, if the discrete space~$\Vp(\Th)$ satisfies the inclusion condition in  Assumption~\ref{ASM::TIME-DERIVATIVE-ASSUMPTION}, Assumption~\ref{ASSUMPTION::PRISMATIC-MESH} on~$\Th$ holds, and the stabilization function~$\eta_F$ is given by~\eqref{EQN::STABILIZATION-TERM}, then
\begin{equation}
\label{EQ::STRANG_PLUS}
\begin{split}
\Tnorm{u - \uh}{\LDGp} \leq \inf_{\vh \in \Vp(\Th)} \Tnorm{u - \vh}{\LDGp} & +\sqrt{2}\gamma_I^{-1} \inf_{\vh \in \Vp(\Th)}\Tnorm{u - \vh}{\LDGs} \\
& + 
\gamma_I^{-1}
\sup_{\vh \in \Vp(\Th)\setminus \{0\}} \frac{|\Rh(u, \vh)|}{\Tnorm{\vh}{\LDGp}}.
\end{split}
\end{equation}
\end{subequations}
\end{theorem}

\begin{proof}
The proof follows from the Strang lemma, identity~\eqref{EQN::IDENTITY-Rh}, the inf-sup stability estimates in Theorems~\ref{THM::INF-SUP-NEWTON} and~\ref{THM::INF-SUP-TWO}, and the continuity bounds for~$\Bh(\cdot, \cdot)$ in~Lemmas~\ref{PROP::CONT-NEWTON} and~\ref{PROP::CONT-LDGP}.
\end{proof}

Next lemma provides a bound for the inconsistency term in the \emph{a priori} error bounds~\eqref{EQ::STRANG_NEWTON} and~\eqref{EQ::STRANG_PLUS}.
\begin{lemma}[Inconsistency bound]\label{LEMMA::INCONSISTENCY-BOUND}
For all~$u\in H^{\frac32 + \varepsilon}(\Th)$ with~$\varepsilon > 0$ and~$v \in \Vp(\Th)+ Y$, the following bound holds:
\begin{equation*}
\begin{split}
|\Rh(u, v)|
& \leq \Big(\Norm{\eta_F^{-\frac12} \mvl{\bk(\nabla_{\bx, h} u - \PiO\nabla_{\bx, h} u)}_{1 - \alpha_F}}{L^2(\Ftime)^d} \\
& \quad + \Norm{\eta_F^{-\frac12} \bk(\nabla_{\bx, h} u - \PiO\nabla_{\bx, h} u) \cdot \bnOmega}{L^2(\FD)} \Big) \Tnorm{v}{},
\end{split}
\end{equation*}
where~$\Tnorm{\cdot}{}$ denotes either~$\Tnorm{\cdot}{\LDG}$ or~$\Tnorm{\cdot}{\LDGp}$.
\end{lemma}
\begin{proof}
The result readily follows from the definition in~\eqref{EQN::DEF-Rh} of~$\Rh(\cdot, \cdot)$, the Cauchy--Schwarz inequality, and the definitions of~$\Tnorm{\cdot}{\LDG}$ and~$\Tnorm{\cdot}{\LDGp}$.
\end{proof}

Interpolation estimates for the term~$\Norm{\Nh(u - \vh)}{\LDG}$ in~\eqref{EQ::STRANG_NEWTON}
are not immediate.
In the rest of this section, we show that, for piecewise polynomial spaces, such a term can be bounded by some more standard error terms; see Proposition~\ref{PROP::BOUND-Nh-LDG} below.

We first show a bound for the lifting of discrete functions and a discrete Poincar\'e inequality for meshes satisfying Assumption~\ref{ASSUMPTION::PRISMATIC-MESH}.

\begin{lemma}[Bound on~$\Lu$]
\label{LEMMA::BOUND-Lh}
Let~$\Th$ satisfy Assumption~\ref{ASSUMPTION::PRISMATIC-MESH} and the stability function~$\eta_F$ be chosen as in~\eqref{EQN::STABILIZATION-TERM}. If the discrete spaces~$(\Vp(\Th), \Mp(\Th))$ satisfy the local compatibility condition in Assumption~\ref{ASM::COMPATIBILITY-CONDITION} and~$\VpK(K) \subset \Pp{\pK}{K}$ for all~$K \in \Th$, then there exists a constant~$\CL >0$ independent of the meshsize~$h$ and the degree vector~$\p$ such that
    \begin{equation*}
        \Norm{\Lu \vh}{L^2(\QT)^d} \leq \CL \big(\Norm{\eta_F^{\frac12} \jump{\vh}_{\sf N}}{L^2(\Ftime)^d} + \Norm{\eta_F^{\frac12} \vh}{L^2(\FD)}\big) \qquad \forall \vh \in \Vp(\Th).
    \end{equation*}
\end{lemma}
\begin{proof} 
Let~$\vh \in \Vp(\Th)$. Using the Cauchy--Schwarz inequality and the definition in~\eqref{EQN::LIFTING-Uh} of the lifting operator~$\Lu$, we get
\begin{align}
\nonumber
\Norm{\Lu \vh}{L^2(\QT)^d}^2 & = \int_\Ftime \jump{\vh}_{\bN} \cdot \mvl{\Lu \vh}_{1 - \alpha_F} \dS + \int_\FD \vh (\Lu \vh) \cdot \vnOmega \dS \\
\nonumber
& \leq \Norm{\eta_F^{\frac12} \jump{\vh}_{\bN}}{L^2(\Ftime)^d}\Norm{\eta_F^{-\frac12} \mvl{\Lu \vh}_{1 - \alpha_F}}{L^2(\Ftime)^d} \\
& \quad + \Norm{\eta_F^{\frac12} \vh}{L^2(\FD)} \Norm{\eta_F^{-\frac12} \Lu \vh}{L^2(\FD)^d}.
\label{EQN::AUX-BOUND-Lu}
\end{align}
Moreover, using the polynomial trace inequality in~\eqref{EQN::TRACE-INEQUALITY-1}, Assumption~\ref{ASSUMPTION::PRISMATIC-MESH} on~$\Th$, the definition in~\eqref{EQN::STABILIZATION-TERM} of the stabilization function~$\eta_F$, and the ellipticity condition~\eqref{EQN::DIFFUSION}, we obtain
\begin{align*}
\Norm{\eta_F^{-\frac12} \mvl{\Lu \vh}_{1 - \alpha_F}}{L^2(\Ftime)^d}^2 &  + \Norm{\eta_F^{-\frac12} \Lu \vh}{L^2(\FD)^d}^2  \\
&  \leq 2 \sum_{K \in \Th} 
\sum\limits_{
\substack{
F = \Fx \times \Ft, \\
\Fx \in \mathcal{T}_{\partial \Kx}
}
}
\!\!\!\!\!\!
\eta_F^{-1} \Norm{ \Lu \vh}{L^2(F)^d}^2  \\
& \leq 2 \sum_{K \in \Th} 
\sum\limits_{
\substack{
F = \Fx \times \Ft, \\
\Fx \in \mathcal{T}_{\partial \Kx}
}
}
\!\!\!\!\!\!
\eta_F^{-1}\frac{(\pK + 1)(\pK + d)}{d} \frac{\abs{\Fx}}{\abs{\sKxF}} \Norm{\Lu \vh }{L^2(\Ft; L^2(\sKxF)^d)}^2 \\
& \leq \frac{2  C_s}{\eta^{\star}\theta} \Norm{\Lu \vh}{L^2(\QT)^d}^2,
\end{align*}
which, combined with~\eqref{EQN::AUX-BOUND-Lu}, gives the desired result with~$\CL= \sqrt{\frac{2 C_s}{\eta^{\star}\theta}}$.
\end{proof}

The proof of the discrete Poincar\'e inequality in Lemma~\ref{LEMMA::POINCARE} below follows the ideas used in~\cite[Lemma~2.15]{DiPietro_Droniou:2020}. First, we recall the following multiplicative trace inequality for simplices from~\cite[Eq. (1.52) in Ch.~1]{DiPietro_Droniou:2020}:
for any simplex~$T$ and any facet~$F$ of~$T$, there exists a positive constant~$\Ctr^{\star}$ depending only on~$d$ such that
\begin{equation}
\label{EQN::TRACE-INEQUALITY-CONTINUOUS}
\begin{split}
\Norm{v}{L^2(F)}^2 & \le \frac{|F|}{|T|} \Norm{v}{L^2(T)}^2 + 2 h_T \frac{|F|}{d |T|} \Norm{v}{L^2(T)} \Norm{\nabla v}{L^2(T)^d} \\
& \le \Ctr^{\star} \frac{|F|}{|T|} \big(\Norm{v}{L^2(T)}^2 + h_T^2 \Norm{\nabla v}{L^2(T)^d}^2 \big) \qquad \forall v \in H^1(T).
\end{split}
\end{equation}

\begin{lemma}[Discrete Poincar\'e inequality]
\label{LEMMA::POINCARE}
Let the assumptions of Lemma~\ref{LEMMA::BOUND-Lh} hold. There exists a positive constant~$C_P$ independent of the meshsize~$h$, the degree vector~$\p$, and the number of  facets per element such that
 \begin{equation}\label{EQ::POINCARE}
     \Norm{\vh}{L^2(\QT)}\leq  \CP \Tnorm{\vh}{\LDG} \qquad \forall \vh \in \Vp(\Th).
 \end{equation}    
\end{lemma}
\begin{proof}
Let~$\vh \in \Vp(\Th)$. From~\cite[Lemma 8.3 in Ch.~8]{DiPietro_Droniou:2020}, we deduce that there exists a function~$\btau \in L^2(0, T; H^1(\Omega)^d)$, such that 
\begin{equation}
\label{EQN::STOKES-BOUNDS}
    \nablax \cdot \btau = \vh \quad \text{ and } \quad \Norm{\btau}{L^2(\QT)^d} + \Norm{\btau}{L^2(0, T; H^1(\Omega)^{d})} \le C_{\Omega} \Norm{\vh}{L^2(\QT)},
\end{equation}
for some positive constant~$C_{\Omega}$ depending only on~$\Omega$.

Integrating by parts in space, and using the Cauchy--Schwarz inequality, the continuity of the spatial normal component of~$\btau$, the nondegeneracy of~$\bk$ in~\eqref{EQN::DIFFUSION}, the bound in Lemma~\ref{LEMMA::BOUND-Lh} on~$\Norm{\Lu \vh}{L^2(\QT)}$, and the definition of the norm~$\Tnorm{\cdot}{\LDG}$ in~\eqref{EQN::LDG-NORM}, we get
\begin{alignat}{3}
\nonumber
\Norm{\vh}{L^2(\QT)}^2 
& = \int_{\QT} \vh \nablax \cdot \btau \dV \\
\nonumber
& = - \sum_{K \in \Th} \int_{K} \nablax \vh \cdot \btau \dV + \sum_{K \in \Th} \int_{\Kt} \int_{\partial \Kx} \vh \btau \cdot \bn_{\Kx} \dS \dt \\
\nonumber
& = - \sum_{K \in \Th} \int_K \nablax \vh \cdot \btau \dV + \int_{\Ftime} \btau \cdot \jump{\vh}_{\sf N} \dS + \int_{\FD} \vh \btau \cdot \bnOmega \dS \\
\nonumber
& \le \Norm{(\sqrt{\bk})^{-1}\sqrt{\bk} \nablax^h \vh}{L^2(\QT)^d} \Norm{\btau}{L^2(\QT)^d} + \Norm{\eta_F^{\frac12} \jump{\vh}_{\sf N}}{L^2(\Ftime)^d} \Norm{\eta_F^{-\frac12} \btau}{L^2(\Ftime)^d}  \\
\nonumber
& \quad + \Norm{\eta_F^{\frac12} \vh}{L^2(\FD)} \Norm{ \eta_F^{-\frac12} \btau}{L^2(\FD)^d} \\
\nonumber
& \le \theta^{-\frac12}\Norm{\sqrt{\bk} \nablaxLDG \vh}{L^2(\QT)^d} \Norm{\btau}{L^2(\QT)^d} + \theta^{-\frac12}\Norm{\sqrt{\bk} \Lu \vh}{L^2(\QT)^d}\Norm{\btau}{L^2(\QT)^d} \\
\nonumber
& \quad + \Norm{\eta_F^{\frac12} \jump{\vh}_{\sf N}}{L^2(\Ftime)^d} \Norm{\eta_F^{-\frac12} \btau}{L^2(\Ftime)^d} + \Norm{\eta_F^{\frac12} \vh}{L^2(\FD)} \Norm{ \eta_F^{-\frac12} \btau}{L^2(\FD)^d} \\
\nonumber
& \le \theta^{-\frac12} \Norm{\sqrt{\bk} \nablaxLDG \vh}{L^2(\QT)^d} \Norm{\btau}{L^2(\QT)^d} \\
\nonumber
& \quad + \CL \theta^{-\frac12} |\bk|_2 \big(\Norm{\eta_F^{\frac12} \jump{\vh}_{\sf N}}{L^2(\Ftime)^d} + \Norm{\eta_F^{\frac12} \vh}{L^2(\FD)} \big) \Norm{\btau}{L^2(\QT)^d} \\
\nonumber
& \quad + \Norm{\eta_F^{\frac12} \jump{\vh}_{\sf N}}{L^2(\Ftime)^d} \Norm{\eta_F^{-\frac12} \btau}{L^2(\Ftime)^d} + \Norm{\eta_F^{\frac12} \vh}{L^2(\FD)} \Norm{ \eta_F^{-\frac12} \btau}{L^2(\FD)^d} \\
\nonumber
& \le \sqrt{2}\max\{1, \theta^{-\frac12} (1 + \CL |\bk|_2)\} \Tnorm{\vh}{\LDG} \big(\Norm{\btau}{L^2(\QT)^d}^2 + \Norm{\eta_F^{-\frac12} \btau}{L^2(\Ftime)^d}^2 \\
\label{EQN::AUX-BOUND-POINCARE}
& \quad +  \Norm{\eta_F^{-\frac12} \btau}{L^2(\FD)^d}^2\big)^{\frac12}.
\end{alignat}
Therefore, it only remains to bound the last terms on the right-hand side of~\eqref{EQN::AUX-BOUND-POINCARE}.

Using the trace inequality~\eqref{EQN::TRACE-INEQUALITY-CONTINUOUS}, Assumption~\ref{ASSUMPTION::PRISMATIC-MESH}, the nondegeneracy of~$\bk$ in~\eqref{EQN::DIFFUSION}, and the definition of the stabilization parameter~$\eta_F$ in~\eqref{EQN::STABILIZATION-TERM}, we obtain
\begin{alignat}{3}
\nonumber
\Norm{\eta_F^{-\frac12} \btau}{L^2(\Ftime)^d}^2 & + 
\Norm{\eta_F^{-\frac12} \btau}{L^2(\FD)^d}^2 \\
\nonumber
& \le \Ctr^{\star}\sum_{K \in \Th} 
\sum\limits_{
\substack{
F = \Fx \times \Ft, \\
\Fx \in \mathcal{T}_{\partial \Kx}
}
}
\!\!\!\!\!\!
\eta_F^{-1} \frac{|\Fx|}{|\sKxF|} \big(\Norm{\btau}{L^2(\Ft; L^2(\sKxF)^d)}^2 + \diam(\sKxF)^2 \Norm{\nabla \btau}{L^2(\Ft; L^2(\sKxF)^{d \times d})}^2 \big) \\
\nonumber
& \le \Ctr^{\star}\sum_{K \in \Th} 
\sum\limits_{
\substack{
F = \Fx \times \Ft, \\
\Fx \in \mathcal{T}_{\partial \Kx}
}
}
\!\!\!\!\!\!
\eta_F^{-1} \frac{C_s d}{\hKx}\big(\Norm{\btau}{L^2(\Ft; L^2(\sKxF)^d)}^2 + \diam(\Omega)^2 \Norm{\nabla \btau}{L^2(\Ft; L^2(\sKxF)^{d \times d})}^2 \big) \\
\nonumber
& \le \frac{\Ctr^{\star} C_s d}{\eta^{\star} \theta } \Big(\Norm{\btau}{L^2(\QT)^d}^2 + \diam(\Omega)^2 \Norm{\nabla \btau}{L^2(\QT)^{d \times d}}^2\Big),
\end{alignat}
which, combined with~\eqref{EQN::AUX-BOUND-POINCARE} and~\eqref{EQN::STOKES-BOUNDS}, gives the desired bound~\eqref{EQ::POINCARE}.
\end{proof}

\begin{proposition}[Bound on~$\Tnorm{\Nh(\cdot)}{\LDG}$]
\label{PROP::BOUND-Nh-LDG}
Let the assumptions of Lemma~\ref{LEMMA::BOUND-Lh} hold. There exists a constant~$\CN > 0$ independent of the meshsize~$h$ and the degree vector~$\p$ such that the following bound holds for all~$v \in \Vp(\Th) + H^1(\Th)$:
\begin{equation}
\label{EQN::BOUND-Nh-LDG}
 \Tnorm{\Nh v}{\LDG} \leq  \CN\Bigg[ \left(\sum_{K \in \Th} \Norm{\dpt v}{L^2(K)}^2 \right)^{\frac12} 
 + \Norm{\jump{\lambdah^{-\frac12 } v}_t}{L^2(\Fspa)} + \Norm{\lambdah^{-\frac12} v}{L^2(\FO)}
 \Bigg].
\end{equation}
\end{proposition}
\begin{proof}
Let~$v \in \Vp(\Th) + H^1(\Th)$. Using the definition in~\eqref{EQN::DISCRETE-NEWTON} of the discrete Newton potential, and the coercivity of~$\Ah(\cdot, \cdot)$ in Lemma~\ref{LEMMA::COERCIVITY-CONTINUITY-Ah}, we get
\begin{equation}
\label{EQN::AUX-BOUND-Nh}
\begin{split}
\Tnorm{\Nh v}{\LDG}^2 = \Ah(\Nh v, \Nh v) & = \mh(v,\Nh v) \\
& =  \sum_{K \in \Th} \int_K \dpt v \Nh v \dV - \int_{\Fspa} (\Nh v)^+ \jump{v}_t \dx + \int_{\FO} v \Nh v \dx.
\end{split}
\end{equation}

The first term on the right-hand side of identity~\eqref{EQN::AUX-BOUND-Nh} can be bounded using the Cauchy--Schwarz inequality and the Poincar\'e inequality in~\eqref{EQ::POINCARE} as follows:
\begin{align}
\nonumber
 \sum_{K \in \Th} \int_K \dpt v \Nh v \dV 
 \leq \sum_{K \in \Th} \Norm{\dpt v}{L^2(K)}  \Norm{\Nh v}{L^2(K)} & \leq
\left(\sum_{K \in \Th} \Norm{ \dpt v}{L^2(K)}^2 \right)^{\frac12} \Norm{ \Nh v}{L^2(\QT)} \\
\label{EQN::FIRST-TERM-BOUND-Nh}
& \leq
C_P	\left(\sum_{K \in \Th} \Norm{\dpt v}{L^2(K)}^2\right)^{\frac12}  \Tnorm{\Nh v}{\LDG}.
 \end{align}

As for the second and third terms on the right-hand side of identity~\eqref{EQN::AUX-BOUND-Nh}, we use the Cauchy--Schwarz inequality, the polynomial trace inequality in~\eqref{EQN::TRACE-INEQUALITY-2}, and the Poincar\'e inequality in~\eqref{EQ::POINCARE} to obtain
\begin{align}
\nonumber
- \int_{\Fspa} & (\Nh v)^+ \jump{v}_t \dx + \int_{\FO} v \Nh v \dx \\
\nonumber
&\leq \left(\Norm{\jump{\lambdah^{-\frac12} v}_t}{L^2(\Fspa)} + 
\Norm{\lambdah^{-\frac12} v}{L^2(\FO)}\right)\left(\Norm{ (\lambdah^{\frac12} \Nh v)^+ }{L^2(\Fspa)}^2 +\Norm{\lambdah^{\frac12} \Nh v}{L^2(\FO)}^2 \right)^{\frac12}\\
\nonumber
&\leq
\left(\Norm{\jump{ \lambdah^{-\frac12} v}_t }{L^2(\Fspa)}
+ \Norm{\lambdah^{-\frac12} v}{L^2(\FO)}\right)
\Big(\sum_{K \in \Th} \Ctr \lK \frac{\pK^2}{\hKt} \Norm{\Nh v}{L^2(K)}^2 \Big)^{\frac12} 
\\
\nonumber
& \leq 
 \Ctr^{\frac12} \left(\Norm{\jump{\lambdah^{-\frac12} v}_t}{L^2(\Fspa)} + 
\Norm{\lambdah^{-\frac12} v}{L^2(\FO)}\right) \Norm{\Nh v}{L^2(\QT)}
 \\
 \label{EQN::SECOND-THIRD-TERMS-BOUND-Nh}
 &\leq 
 C_P \Ctr^{\frac12} \left(\Norm{\jump{\lambdah^{-\frac12} v}_t}{L^2(\Fspa)} + 
\Norm{\lambdah^{-\frac12} v}{L^2(\FO)}\right) \Tnorm{\Nh v}{\LDG}.
 \end{align}
Bound~\eqref{EQN::BOUND-Nh-LDG} then follows with~$\CN = C_P \Ctr^{\frac12}$ by combining identity~\eqref{EQN::AUX-BOUND-Nh} with bounds~\eqref{EQN::FIRST-TERM-BOUND-Nh} and~\eqref{EQN::SECOND-THIRD-TERMS-BOUND-Nh}.
\end{proof}

\begin{remark}[Continuous dependence on the data]
If the space~$\Vp(\Th)$ is such that the discrete Poincar\'e inequality~\eqref{EQ::POINCARE} is valid, the solution~$\uh \in \Vp(\Th)$ to the space--time LDG formulation~\eqref{EQN::REDUCED-VARIATIONAL-DG} satisfies the following continuous dependence on the data:
\begin{equation*}
\Tnorm{\uh}{\LDGN} \le 2\sqrt{2} \big(\CP \Norm{f}{L^2(\QT)} + \sqrt{2} \Norm{u_0}{L^2(\Omega)} \big) \quad \text{ and }  \quad 
\Tnorm{\uh}{\LDGp} \le \gamma_I^{-1} \big(\CP \Norm{f}{L^2(\QT)} + \sqrt{2} \Norm{u_0}{L^2(\Omega)} \big),
\end{equation*}
which follow from the inf-sup conditions in Theorems~\ref{THM::INF-SUP-NEWTON} and~\ref{THM::INF-SUP-TWO}, the definition of the method in~\eqref{EQN::REDUCED-VARIATIONAL-DG}, and the Cauchy--Schwarz and the triangle inequalities.
\eremk
\end{remark}

\subsection{Difficulties in proving optimal convergence rates in the~\texorpdfstring{$L^2(\QT)$}{} norm} \label{SECT::A-PRIORI-L2}
In this section, we briefly discuss the difficulties of deriving~\emph{a priori} error estimates in the mesh-independent norm~$L^2(\QT)$ by means of a duality argument.
We consider polynomial discrete spaces satisfying the local inclusion in Assumption~\ref{ASM::TIME-DERIVATIVE-ASSUMPTION}.

For a given~$\phi \in L^2(\QT)$, we consider the following adjoint problem:
\begin{subequations}
\label{EQN::ADJOINT-PROBLEM}
\begin{align}
\label{EQN::ADJOINT-PROBLEM-1}
    -\dpt z - \nablax \cdot (\bk \nablax z) = \phi& \qquad \text{ in } \QT, \\
\label{EQN::ADJOINT-PROBLEM-2}
    z = 0 & \qquad \text{ on } \partial \Omega \times (0,T], \\
\label{EQN::ADJOINT-PROBLEM-3}
    z = 0 & \qquad \text{ on } \Omega \times \{T\},
\end{align}
\end{subequations}
and assume that~$\bk$ and~$\Omega$ are such that the following parabolic regularity estimate holds (cf.~\cite[Thm.~5 in \S7.1]{Evans_1998}):
\begin{equation}
\label{EQN::PARBOLIC-REGULARITY}
    \Norm{z}{L^{\infty}(0,T;H^1_0(\Omega))}+
     \Norm{z}{L^{2}(0,T;H^2(\Omega))}+
      \Norm{z}{H^{1}(0,T;L^2(\Omega))}\leq C_R \Norm{\phi}{L^{2}(\QT)},
\end{equation}
for some~$C_R > 0$ depending only on~$\Omega$, $T$, and~$\bk$.
\begin{theorem}[\emph{A priori} error bounds in the~$L^2(\QT)$ norm] 
\label{THM::A-PRIORI-L2}
Let~$u \in H^{\frac32 + \varepsilon}(\Th) \cap X$ be the solution to the continuous weak formulation~\eqref{EQN::CONTINUOUS-WEAK-FORMULATION}, and let~$\uh \in \Vp(\Th)$ be the solution to the space--time LDG formulation~\eqref{EQN::REDUCED-VARIATIONAL-DG}. If~$\Omega$ is such that the parabolic regularity estimate~\eqref{EQN::PARBOLIC-REGULARITY} holds, and the discrete spaces~$(\Vp(\Th), \Mp(\Th))$ satisfy Assumptions \ref{ASM::COMPATIBILITY-CONDITION} and~\ref{ASM::TIME-DERIVATIVE-ASSUMPTION},
then the following bound holds:
\begin{equation}
\label{EQN::BOUND-L2}
\Norm{u - \uh}{L^2(\QT)}^2 \leq \sqrt{2} \Tnorm{u - \uh}{\LDGp} \Tnorm{z - \zh}{\LDGss} + |\Rh(z, u - \uh)| + |\Rh(u, z - \zh)| \quad \forall \zh \in \Vp(\Th),
\end{equation}
where~$z$ is the solution to the adjoint problem~\eqref{EQN::ADJOINT-PROBLEM} with~$\phi = u - \uh$.
\end{theorem}
\begin{proof}
We set~$\phi=u-\uh$ in~\eqref{EQN::ADJOINT-PROBLEM} and multiply~\eqref{EQN::ADJOINT-PROBLEM-1} by~$\phi$. Integration by parts, the regularity of~$z$, the orthogonality properties of~$\PiO$, the definition in~\eqref{EQN::DEF-Rh} of the inconsistency bilinear form~$\Rh(\cdot, \cdot)$, and the definition in~\eqref{EQN::LIFTING-Uh} of the lifting operator~$\Lu$ yield
\begin{align*}
    \Norm{u-\uh}{L^{2}(\QT)}^2 & = \sum_{K\in\Th}\int_K ( -\dpt z - \nablax \cdot (\bk \nablax z))(u-\uh)\dV\\
& =
    \sum_{K\in\Th}\int_K \Big(\dpt (u-\uh) z  + \bk \nablax(u - \uh) \cdot \nablax z \Big)\dV 
    - \int_{\Fspa} z^+ \jump{u-\uh}_t \dx \\
& \quad 
    + \int_{\FO} z(u - \uh) \dx
    -\int_{\Ftime} \mvl{\bk\nablax z}_{1 - \alpha_F} \cdot \jump{u-\uh}_\bN \dS 
    - \int_{\FD}(u-\uh)\bk\nablax z \cdot \vnOmega \dS \\
& =
    \sum_{K\in\Th}\int_K \Big(\dpt (u-\uh) z  + \bk \nablax(u - \uh) \cdot \nablax z \Big)\dV 
    - \int_{\Fspa} z^+ \jump{u-\uh}_t \dx \\
& \quad 
    + \int_{\FO} z(u - \uh) \dx
    -\int_{\Ftime} \mvl{\bk\nablax z - \bk \PiO \nablax z }_{1 - \alpha_F} \cdot \jump{u-\uh}_\bN \dS
    \\
& \quad 
    - \int_{\FD}(u-\uh) (\bk\nablax z - \bk \PiO\nablax z) \cdot \vnOmega \dS + \int_{\QT} \bk \Lu \uh \cdot \nablax z \dV \\
& =
    \sum_{K\in\Th}\int_K \Big(\dpt (u-\uh) z  + \bk \nablaxLDG (u - \uh) \cdot \nablaxLDG z \Big)\dV 
    - \int_{\Fspa} z^+ \jump{u-\uh}_t \dx \\
& \quad 
    + \int_{\FO} z(u - \uh) \dx
    - \int_{\Ftime} \mvl{\bk\nablax z - \bk \PiO \nablax z }_{1 - \alpha_F} \cdot \jump{u - \uh}_\bN \dS
    \\
& \quad 
    - \int_{\FD} (u - \uh) (\bk\nablax z - \bk \PiO\nablax z) \cdot \vnOmega \dS \\
& = \Bh(u - \uh, z) - \Rh(z, u - \uh) \\
& = \Bh(u - \uh, z - \zh) - \Rh(z, u - \uh) - \Rh(u, z - \zh),
\end{align*}
where, in the last equation, we have used identity~\eqref{EQN::IDENTITY-Rh}.

Bound~\eqref{EQN::BOUND-L2} then follows from the continuity bound in~\eqref{EQN::CONT-LDGP-2} for the bilinear form~$\Bh(\cdot, \cdot)$.
\end{proof}

\begin{remark}[Limited regularity of the adjoint solution]
The last two terms on the right-hand side of~\eqref{EQN::BOUND-L2} can be bounded using Lemma~\ref{LEMMA::INCONSISTENCY-BOUND}.
The main difficulty in deriving~\emph{a priori} error estimates in the~$L^2(\QT)$ norm is the limited regularity in time of the continuous solution~$z$ to the adjoint problem~\eqref{EQN::ADJOINT-PROBLEM}. This issue was overcome in the corrected version~\cite{Cangiani_Dong_Georgoulis:2017ARXIV} of reference~\cite{Cangiani_Dong_Georgoulis:2017} by using a continuous-in-time interpolant operator. However, extending this approach to general prismatic meshes is not straightforward.
\eremk
\end{remark}

\section{Some choices of discrete spaces\label{SECT::DISCRETE-SPACES}}

We introduce four different choices for the local discrete space~$\VpK(K)$ and discuss their properties.
More precisely, we consider some standard~$\Pp{\pK}{K}$, tensor-product~$\IQ^{\pK}(K)$, 
quasi-Trefftz~$\qT^{\pK}(K)$, and embedded Trefftz~$\eT^{\pK}(K)$ polynomial spaces. 
Moreover, using the approximation properties of each space and the error bounds from the previous section, we derive some~\emph{a priori} error estimates.

In Table~\ref{tab:spaces}, we compare the different spaces in terms of their dimension, a choice of the discrete space~$\MpK(K)$ guaranteeing the validity of the local compatibility condition in Assumption~\ref{ASM::COMPATIBILITY-CONDITION}, and the validity of the local inclusion condition in Assumption~\ref{ASM::TIME-DERIVATIVE-ASSUMPTION}. The error estimates in Theorems~\ref{THM::ERROR-ENERGY-QK}, \ref{THM::ERROR-ENERGY-PK}, and~\ref{THM::ERROR-ENERGY-QTK} predict the same convergence rates in the energy norms for the corresponding spaces, but under different regularity assumptions. More precisely, the solution~$u$ to~\eqref{EQN::MODEL-PROBLEM} is assumed to belong to local Bochner spaces (separating space and time regularity) for tensor-product polynomials, to local space–time Sobolev spaces $H^{\ell}$ for standard polynomials, and to local space--time $\mathcal{C}^{\ell}$ spaces for quasi-Trefftz polynomials.
\begin{table}[!ht]
\centering
\small
\begin{tabular}{lcccc}
    \toprule
    \multirow{2}{*}{Space} & \multirow{2}{*}{$\VpK(K)$} & \multirow{2}{*}{$\MpK(K)$} & \multirow{2}{*}{Dimension} & Inclusion \\
    & & & & condition~\eqref{EQN::DPT-CONDITION} \\
\hline \\[0.01em]
Tensor-product~\S\ref{SECT::TENSOR-PRODUCT} & $\IQ^{\pK}(K)$ in~\eqref{DEF::QK} &
$\IQ^{\pK}(K)^d$ & $(\pK + 1) \cdot \binom{\pK + d}{\pK}\approx p^{d + 1}$ & \cmark \\[1em]
   Standard~\S\ref{SECT::STANDARD-POLY} & $\Pp{\pK}{K}$ in~\eqref{DEF::PK} 
   & $\Pp{\pK}{K}^d$ & $\binom{\pK + d + 1}{\pK}\approx p^{d+1}$ & \cmark \\[1em]
    Quasi-Trefftz~\S\ref{SECT::QUASI-TREFFTZ} & $\qT^{\pK}(K)$ in~\eqref{DEF::QTK} & $\Pp{\pK}{K}^d$ & $\binom{\pK + d}{\pK} + \binom{\pK - 1 + d}{\pK - 1} \approx p^{d}$ & \xmark \\[1em]
    Embedded-Trefftz~\S\ref{SECT::EMBEDDED-TREFFTZ} & $\eT^{\pK}(K)$ in~\eqref{DEF::ETK} & $\Pp{\pK}{K}^d$ & $\binom{\pK + d}{\pK} + \binom{\pK - 1 + d}{\pK - 1} \approx p^{d}$ & \xmark \\[1em]
    \bottomrule
\end{tabular}
\vspace{3ex}
\caption{Comparison of different choices for the local discrete space~$\VpK(K)$: a choice for the local discrete space~$\MpK(K)$ satisfying the compatibility condition in Assumption~\ref{ASM::COMPATIBILITY-CONDITION} (\textbf{third column}), the dimension of~$\VpK(K)$ (\textbf{fourth column}), and the validity of the local inclusion condition in Assumption~\ref{ASM::TIME-DERIVATIVE-ASSUMPTION} (\textbf{fifth column}).}
\label{tab:spaces}
\end{table}

In the rest of this section, we write~$a \lesssim b$ meaning that there exists a positive constant~$C$ independent of meshsize~$h$, the degree vector~$\p$, and the maximum number of facets such that~$a \le C b$.
Similarly, we use~$a \lp b$ to indicate the possibility that the constant~$C$ depends on the degree vector~$\p$.
Moreover, we write~$a \simeq b$ whenever~$a \lesssim b$ and~$b \lesssim a$.

\subsection{Tensor-product polynomials \label{SECT::TENSOR-PRODUCT}}
We consider the following tensor-product piecewise polynomial space:
\begin{equation}
\label{DEF::QK}
\VpK(K) = \IQ^{\pK}(K) := \Pp{\pK}{\Kx} \otimes \Pp{\pK}{\Kt}.
\end{equation}

Let~$\{\Th\}_{h>0}$ be a family of prismatic space--time meshes for the domain~$\QT = \Omega \times (0, T)$.
Given a partition~$\Th$, as an extension of~\cite[Def. 5.2]{Cangiani_Dong_Georgoulis:2017}, we call \emph{covering} a set~$\Tc = \{\calK\}$ of shape-regular~$(d+1)$-dimensional prisms, whose bases are~$d$-dimensional simplices or hypercubes such that, for each~$K = \Kx \times \Kt \in \Th$, there exists~$\calK = \calKx \times \Kt \in \Tc$ with~$K \subset \calK$.

\begin{assumption}[Covering of~$\Th$]
\label{ASM::COVERING}
There exists a positive integer~$N_{\QT}$ independent of the mesh parameters such that, for any mesh~$\Th \in \{\Th\}_{h > 0}$, there exists a covering~$\Tc$ of~$\Th$ satisfying
\begin{equation*}
\mathrm{card} \{K' \in \Th \, : \, K' \cap \calK \neq \emptyset \text{ for some } \calK \in \Tc \text{ with } K \subset \calK\}  \le N_{\QT} \qquad \forall K \in \Th,
\end{equation*}
which implies that~$\hcalKx := \diam(\calKx) \lesssim \hKx$ for each pair~$K =  \Kx  \times\Kt\in \Th$ and~$\calK = \calKx\times\Kt   \in \Tc$ with~$K \subset \calK$.
\end{assumption}

For any Lipschitz domain~$\Upsilon \subset \IR^{d}$ and~$s \in \IN$, 
the Stein's \emph{extension} operator~$\frakEx : H^s(\Upsilon) \rightarrow H^s(\IR^d)$ 
is a linear operator with the following properties (see~\cite[Thm.~5 in Ch.~VI]{Stein_1970}): 
for all~$v \in H^s(\Upsilon)$, it holds
\begin{equation}
\label{EQN::STEIN}
\frakEx v_{|_{\Upsilon}} = v  \quad \text{ and } \quad \Norm{\frakEx v}{H^s(\IR^d)} \lesssim \Norm{v}{H^s(\Upsilon)},
\end{equation}
where the hidden constant depends only on~$s$ and the shape of~$\Upsilon$.

We now recall the following approximation results from~\cite[Lemmas 23 and 33]{Cangiani_Dong_Georgoulis_Houston_2017}.

\begin{lemma}[Estimates of~$\TPix$]
\label{LEMMA::ESTIMATE-TPI}
Let Assumptions~\ref{ASSUMPTION::PRISMATIC-MESH} and~\ref{ASM::COVERING} 
on~$\Th$ hold. Let also~$\Th \in \{\Th\}_{h > 0}$ and~$\Tc$ be its corresponding covering from Assumption~\ref{ASM::COVERING}.
For any~$K = \Kx \times \Kt \in \Th$ and~$v_{|_K} \in L^2(\Kt; H^{l_K}(\Kx))$ ($l_K > \frac12$), there exists~$\TPix v_{|_K} \in L^2(\Kt; \Pp{\pK}{\Kx})$, such that
\begin{subequations}
\begin{alignat}{3}
\label{EQN::ESTIMATE-TPI-VOLUME}
\Norm{v - \TPix v}{L^2(\Kt; H^q(\Kx))} & \lesssim \frac{\hKx^{\ellK - q}}{\pK^{l_K - q}} \Norm{\frakEx v}{L^2(\Kt; H^{l_K}(\calKx))} & & \quad \text{ for all } 0 \le q \le l_K, \\
\label{EQN::ESTIMATE-TPI-STABILITY}
\Norm{\TPix v}{L^2(K)} & \lesssim \frac{\hKx}{\pK} \Norm{\frakEx v}{L^2(\Kt; H^1(\calKx))} + \Norm{v}{L^2(K)}, \\
\label{EQN::ESTIMATE-TPI-BOUNDARY}
\Norm{v - \TPix v}{L^2(\Kt; L^2(\partial \Kx))} & \lesssim \frac{\hKx^{\ellK-\frac12}}{\pK^{l_K-\frac12}}\Norm{\frakEx v}{L^2(\Kt; H^{l_K}(\calKx))},  
\end{alignat}
\end{subequations}
where~$\ellK = \min\{\pK + 1, l_K\}$.
\end{lemma}
\begin{proof}
The estimate~\eqref{EQN::ESTIMATE-TPI-STABILITY} follows by using the triangle inequality and~\eqref{EQN::ESTIMATE-TPI-VOLUME}.
\end{proof}

Given~$q \in \IN$ and a time interval~$(a, b) \subset \IR$, we denote by~$\pi^t_{q}$ the~$L^2(a, b)$-orthogonal projection in the space~$\Pp{q}{a, b}$.
In the proof of Theorem~\ref{THM::ERROR-ENERGY-QK} below, the operators~$\pi_q^t$ and~$\TPix$ are to be understood as applied pointwise in space and time, respectively.

Lemma~\ref{LEMMA::ESTIMATES-PI-TIME} below concerns some standard properties of~$\pi_q^t$ (see, e.g., \cite[Lemma 3.3 and Lemma 3.5]{Houston_Schwab_Suli_2002}, \cite[Lemma 2.4]{Canuto_Quarteroni_1982}, and \cite[Thm.~2]{Warburton_Hesthaven_2003}).

\begin{lemma}[Estimates of~$\pi_{q}^t$]
\label{LEMMA::ESTIMATES-PI-TIME}
Given an integer~$q \geq 1$ and a time interval~$(a, b) \subset \IR$, the following estimates hold for all~$v \in H^{s}(a, b)$ $(s \geq 1)$:
\begin{subequations}
\begin{alignat}{3}
\label{EQN::ESTIMATE-PI-TIME-VOLUME}
\Norm{v - \pi_q^t v}{L^2(a, b)} & \lesssim \frac{(b-a)^{\ts}}{q^s} \SemiNorm{v}{H^{s}(a, b)}, \\
\label{EQN::ESTIMATE-PI-TIME-INVERSE}
\Norm{v - \pi_q^t v}{H^1(a, b)} & \lesssim \frac{q^2}{(b - a)} \Norm{v}{L^2(a, b)}, \\
\label{EQN::ESTIMATE-PI-TIME-L2-Linfty}
\Norm{v - \pi_q^t v}{L^2(a, b)} & \lesssim (b - a)^{\frac12} \Norm{v}{L^{\infty}(a, b)}, \\
\label{EQN::ESTIMATE-PI-TIME-POLYNOMIAL-TRACE}
|\pi_q^t v(a)| + |\pi_q^t v(b)| & \lesssim \frac{q}{(b - a)^{\frac12}} \Norm{v}{L^2(a, b)}, \\
\label{EQN::ESTIMATE-PI-TIME-BOUNDARY}
|(v - \pi_q^t v)(a)| + |(v - \pi_q^t v)(b)| & \lesssim \frac{(b-a)^{\ts - \frac12} }{q^{s - \frac12} }\SemiNorm{v}{H^{s}(a, b)},
\end{alignat}
\end{subequations}
where~$\ts = \min\{q + 1, s\}$.
\end{lemma}

Let~$q$ be an integer with~$q \geq 1$. The multiplicative trace inequality in one dimension reads:
\begin{equation*}
|v(a)|^2 + |v(b)|^2 \le (b - a)^{-1} \Norm{v}{L^2(a, b)}^2 + 2 \Norm{v}{L^2(a, b)} \Norm{v'}{L^2(a, b)} \qquad \forall v \in H^1(a, b),
\end{equation*}
which, together with the Young inequality, implies
\begin{equation}
\label{EQN::TRACE-INEQUALITY-CONTINUOUS-1D}
|v(a)|^2 + |v(b)|^2 \le \frac{q + 1}{b - a} \Norm{v}{L^2(a, b)}^2 + \frac{b - a}{2q} \Norm{v'}{L^2(a, b)}^2 \qquad \forall v \in H^1(a, b).
\end{equation}

\begin{theorem}[Error estimates in the energy norms]
\label{THM::ERROR-ENERGY-QK}
Let~$\Vp(\Th)$ be chosen as in~\eqref{DEF::QK}, and~$\Mp(\Th)$ be such that Assumption~\ref{ASM::COMPATIBILITY-CONDITION} holds.
Let also Assumptions~\ref{ASSUMPTION::PRISMATIC-MESH} and 
\ref{ASM::COVERING} on the space--time mesh~$\Th$ be satisfied, and the stabilization function~$\eta_F$ be given by~\eqref{EQN::STABILIZATION-TERM}. 
Assume that the exact solution~$u$ to the continuous weak formulation~\eqref{EQN::CONTINUOUS-WEAK-FORMULATION} satisfies: $u \in X$, and, for each~$K = \Kx \times \Kt \in \Th$, $u_{|_K}$ belongs to
\begin{equation*}
L^2(\Kt; H^{l_K}(\Kx)) \cap H^1(\Kt; H^{l_K - 1}(\Kx)) \cap H^{s_K}(\Kt; L^2(\Kx)) \cap H^{s_K - 1}(\Kt; H^1(\Kx)) \cap H^{\vartheta_K}(\Kt; H^2(\Kx)), 
\end{equation*}
where~$s_K > 3/2$, $l_K > 3/2$, and~$\vartheta_K = \max\{s_K - 2, 0\}$. 
Let~$\uh \in \Vp(\Th)$ be the solution to the space--time LDG formulation~\eqref{EQN::REDUCED-VARIATIONAL-DG}. Then, the following estimates hold:
\begin{alignat}{2}
\label{EQN::ERROR-LDGN-QK}
\Tnorm{u - \uh}{\LDGN}^2 & \lesssim \sum_{K \in \Th} \frac{\hKx^{2\ellK - 2}}{\pK^{2 l_K- 3}} \mathfrak{A}_K + \sum_{K \in \Th} \frac{\hKt^{2\tsK - 2}}{\pK^{2 s_K - 3}} \mathfrak{B}_K,
\\
\label{EQN::ERROR-LDGp-QK}
\Tnorm{u - \uh}{\LDGp}^2 & \lesssim \sum_{K \in \Th} \frac{\hKx^{2\ellK - 2}}{\pK^{2 l_K- 3}} \mathfrak{C}_K + \sum_{K \in \Th} \frac{\hKt^{2\tsK - 2}}{\pK^{2 s_K - 3}} \mathfrak{D}_K,
\end{alignat}
where~$\ell_K := \min\{\pK + 1,\, l_K\}$ and~$\tsK := \min\{\pK + 1,\, s_K\}$ for all~$K \in \Th$, and 
\begin{alignat*}{3}
\mathfrak{A}_K & := 
\Big(\frac{1}{\pK} + \frac{\hKx^2}{\hKt \pK^2} + \frac{\hpK^2}{\pK^2} \cdot \frac{\hKx^2}{\hKt \hhKt} + \big(\max_{F \in \FKtime} \eta_F\big) \frac{\hKx}{\pK^2} \Big) \Norm{\frakEx u}{L^2(\Kt; H^{l_K}(\calKx))}^2 \\
\nonumber
& \quad + \left(\frac{1}{\pK} + \frac{\hKt}{\pK^2} + \frac{\hpK^2}{\pK^2} \cdot \frac{\hKt}{\hhKt} \right)\Norm{\frakEx u}{H^1(\Kt; H^{l_K - 1}(\calKx))}^2 + \Big( \frac{1}{\pK} + \frac{1}{\pK^2}\Big) \Norm{\frakEx \nablax u}{L^2(\Kt; H^{l_K - 1}(\calKx)^d)}^2, \\
\mathfrak{B}_K & :=  \Big(\frac{\hKt}{\pK^2} + \frac{1}{\pK} + \frac{\hKt}{\hhKt} \cdot \frac{\hpK^2}{\pK^2} + \big(\max_{F \in \FKtime} \eta_F \big) \frac{\hKt^2}{\hKx \pK^2 }\Big) \Norm{u}{H^{s_K}(\Kt; L^2(\Kx))}^2   \\
\nonumber
& \quad + \Big(\frac{1}{\pK} + \big(\max_{F \in \FKtime} \eta_F \big) \frac{\hKx}{\pK^2} + \frac{1}{\pK^2}\Big) \Norm{u}{H^{s_K - 1}(\Kt; H^1(\Kx))}^2  \\
& \quad + \left(\frac{\hKx^2}{\hKt \pK^2} + \frac{\hKx^2}{\hKt^2 \pK} \right)\Norm{\frakEx u}{H^{s_K- 1}(\Kt; H^1(\calKx))}^2 + \frac{\hKx^2 \hKt^{2\theta_K}}{\hKt^{2\tsK - 2}} \cdot \frac{\pK^{2 s_K - 3}}{\pK^{2\vartheta_K + 3}} \Norm{u}{H^{\vartheta_K}(\Kt; H^2(\Kx))}^2, \\
\mathfrak{C}_K & := 
\Big( \frac{1}{\pK} + \frac{\hKx^2}{\hKt \pK^2} + \big(\max_{F \in \FKtime} \eta_F\big) \frac{\hKx}{\pK^2} + \frac{\hKx^2}{\hhKt}\cdot\frac{\hpK^2}{\pK^3}\Big) \Norm{\frakEx u}{L^2(\Kt; H^{l_K}(\calKx))}^2 \\
\nonumber
& \quad + \Big(\frac{\hKt}{\pK^3} + \frac{\hKt}{\pK^2} \Big) \Norm{\frakEx u}{H^1(\Kt; H^{l_K - 1}(\calKx))}^2 + \Big( \frac{1}{\pK} + \frac{1}{\pK^2}\Big) \Norm{\frakEx \nablax u}{L^2(\Kt; H^{l_K - 1}(\calKx)^d)}^2, \\
\nonumber
\mathfrak{D}_K & := \Big(\frac{\hKt}{\pK^2} + \frac{1}{\pK} + \big(\max_{F \in \FKtime} \eta_F \big) \frac{\hKt^2}{\hKx \pK^2} + \frac{\hKt^2}{\hhKt} \cdot \frac{\hpK^2}{\pK^3} +\frac{\hKt}{\pK^3} \Big) \Norm{u}{H^{s_K}(\Kt; L^2(\Kx))}^2   \\
\nonumber
& \quad + \Big(\frac{1}{\pK} + \big(\max_{F \in \FKtime} \eta_F \big) \frac{\hKx}{\pK^2} + \frac{1}{\pK^2}\Big) \Norm{u}{H^{s_K - 1}(\Kt; H^1(\Kx))}^2  \\
& \quad + \left(\frac{\hKx^2}{\hKt \pK^2} +\frac{\hKx^2}{\hKt \pK^3} \right)\Norm{\frakEx u}{H^{s_K- 1}(\Kt; H^1(\calKx))}^2 + \frac{\hKx^2 \hKt^{2\theta_K}}{\hKt^{2\tsK - 2}} \cdot \frac{\pK^{2 s_K- 3}}{\pK^{2\vartheta_K + 3}} \Norm{u}{H^{\vartheta_K}(\Kt; H^2(\Kx))}^2,
\end{alignat*}
with $\theta_K := \min\{\pK + 1,\, \vartheta_K\}$.
\end{theorem}
\begin{proof} 
For the sake of clarity, we postpone the proof of this theorem to Appendix~\ref{APP::PROOF-ENERGY-QK}.
\end{proof}

\begin{remark}[Conditions on the mesh and the degree vector]
The error estimate~\eqref{EQN::ERROR-LDGN-QK} in the norm $\Tnorm{\cdot}{\LDGN}$ suggests the need of the following local quasi-uniformity conditions:
\begin{equation}
\label{EQN::LOCAL-QUASI-UNIFORMITY}
\pK \simeq p_{K'}, \quad \hKt \simeq h_{K'_t}, \quad \hKx \simeq h_{K'_{\bx}} \qquad \text{ for all } K, K' \in \Th \text{ sharing a time-like facet},
\end{equation}
and the orthotropic scaling
\begin{equation}
\label{EQN::ORTHOTROPIC-RELATION}
\hKt \simeq \hKx \qquad \forall K = \Kx \times \Kt \in \Th.
\end{equation}

On the other hand, the error estimate~\eqref{EQN::ERROR-LDGp-QK} in the norm~$\Tnorm{\cdot}{\LDGp}$ suggests that the orthotropic scaling~\eqref{EQN::ORTHOTROPIC-RELATION} can be relaxed to
\begin{equation}
\label{EQN::RELAXED-ORTHOTROPIC-TENSOR}
\hKx^2 \lesssim \hKt \lesssim \hKx \qquad \forall K = \Kx \times \Kt \in \Th,
\end{equation}
as the last two terms on the right-hand side of~\eqref{EQN::ERROR-LDGp-QK} can be improved by requiring some extra regularity on the continuous solution~$u$.

In the absence of hanging time-like facets, the condition~$\hKx^2 \lesssim \hKt$ in~\eqref{EQN::RELAXED-ORTHOTROPIC-TENSOR} can be further relaxed using the analysis in~\cite[\S5.2]{Ern_Schieweck:2016}, which is based on some composed lifting and projection operators.
Moreover, the condition~$\hKt \lesssim \hKx$ in~\eqref{EQN::RELAXED-ORTHOTROPIC-TENSOR} is a consequence of the fact that, for non-uniform degrees of approximation or meshes with hanging time-like facets, it does not hold that~$\jump{\pi_{p}^t v}_{\sf N} = 0$ on~$\Ftime$.
\eremk
\end{remark}

\begin{corollary}
Let the hypotheses of Theorem \ref{THM::ERROR-ENERGY-QK} hold and assume also uniform
elemental polynomial degrees $p_K=p\ge1$ for all $K\in\Th$, the local quasi-uniformity condition \eqref{EQN::LOCAL-QUASI-UNIFORMITY}, and the orthotropic scaling
\eqref{EQN::ORTHOTROPIC-RELATION}. If the continuous solution~$u$ to~\eqref{EQN::CONTINUOUS-WEAK-FORMULATION} belongs to~$L^2(0, T; H^l(\Omega)) \cap H^1(0, T; H^{l - 1}(\Omega)) \cap H^s(0, T; L^2(\Omega)) \cap H^{s - 1}(0, T; H^1(\Omega)) \cap H^{\max\{s - 2, 0\}}(0, T; H^2(\Omega))$ for~$s > 3/2$ and~$l > 3/2$, then the following estimates hold:
\begin{alignat*}{2}
\nonumber
\Tnorm{u - \uh}{\LDGN} & \lesssim  \frac{h^{\ell - 1}}{p^{l - \frac32}} \big( \Norm{u}{L^2(0, T; H^{l}(\Omega))} + \Norm{u}{H^1(0, T; H^{l - 1}(\Omega))} \big) \\
\nonumber
& \quad + \frac{h^{\ts - 1}}{p^{s- \frac32}} \big( \Norm{u}{H^{s}(0, T; L^2(\Omega))}  +  \Norm{u}{H^{s - 1}(0, T; H^1(\Omega))} + \Norm{u}{H^{\vartheta}(0, T; H^2(\Omega))} \big),\\
\nonumber
\Tnorm{u - \uh}{\LDGp} 
& \lesssim \frac{h^{\ell - 1}}{p^{l- \frac32}} 
 \big(\Norm{u}{L^2(0,T; H^{l}(\Omega))} + \Norm{u}{H^1(0, T; H^{l - 1}(\Omega))} \big) \\
 & \quad  + \frac{h^{\ts - 1}}{p^{s - \frac32}} \left(  \Norm{u}{H^{s}(0,T; L^2(\Omega))}  + \Norm{u}{H^{s - 1}(0,T; H^1(\Omega))}+\Norm{u}{H^{\vartheta}(0, T; H^2(\Omega))} \right),
\end{alignat*}
where~$h = \max_{K\in\Th} \hK$, $\ts = \min\{p + 1, s\}$, $\ell = \min\{p + 1, l\}$, and~$\vartheta = \max\{s- 2, 0\}$.
\end{corollary}

\subsection{Standard polynomials \label{SECT::STANDARD-POLY}}
We now consider the following piecewise polynomial space:
\begin{equation}
\label{DEF::PK}
\VpK(K) = \Pp{\pK}{K} \qquad \forall K \in \Th.
\end{equation}

In order to derive \emph{a priori} error estimates for the standard polynomial space in~\eqref{DEF::PK}, we assume that the elements in the family~$\{\Th\}_{h > 0}$ of space--time meshes satisfy the orthotropic relation~\eqref{EQN::ORTHOTROPIC-RELATION}.

\begin{assumption}[Orthotropic space--time elements]
\label{ASM::SHAPE-REGULARITY}
The family~$\{\Th\}_{h > 0}$ of prismatic space--time meshes satisfies
\begin{equation*}
        \hKx \simeq \hKt,
\end{equation*}
uniformly for all~$K \in \Th$.
\end{assumption}

We denote by~$\frakE$ the space--time version in~\cite[Thm. 5.4]{Cangiani_Dong_Georgoulis:2017} of the Stein's extension operator.
Moreover, in next Lemma, we recall the approximation results from~\cite[Lemma~5.5]{Cangiani_Dong_Georgoulis:2017}.

\begin{lemma}[Estimates of~$\TPi$]
\label{LEMMA::ESTIMATE-TPI-XT}
Let Assumptions~\ref{ASSUMPTION::PRISMATIC-MESH}, \ref{ASM::COVERING}, and \ref{ASM::SHAPE-REGULARITY} 
hold. Let also~$\Th \in \{\Th\}_{h > 0}$ and~$\Tc$ be its corresponding covering from Assumption~\ref{ASM::COVERING}.
For any~$K \in \Th$ and~$v_{|_K} \in H^{l_K}(K)$ ($l_K > \frac12$), there exists~$\TPi v_{|_K} \in \Pp{\pK}{K}$, such that
\begin{subequations}
\begin{alignat*}{3}
\Norm{v - \TPi v}{H^q(K)} & \lesssim \frac{\hK^{\ellK - q} }{\pK^{l_K- q} }\Norm{\frakE v}{H^{l_K}(\calK)} & & \quad \text{ for all } 0 \le q \le l_K, \\
\Norm{v - \TPi v}{L^2(\partial K)} & \lesssim \frac{\hK^{\ellK - \frac12}}{\pK^{l_K - \frac12}} \Norm{\frakE v}{H^{l_K}(\calK)},
\end{alignat*}
\end{subequations}
where~$\ellK = \min\{\pK + 1, l_K\}$.
\end{lemma}

\begin{theorem}[Error estimates in the energy norms]
\label{THM::ERROR-ENERGY-PK}
Let~$\Vp(\Th)$ be chosen as in~\eqref{DEF::PK}, and~$\Mp(\Th)$ be such that Assumption~\ref{ASM::COMPATIBILITY-CONDITION} holds.
Let also Assumptions~\ref{ASSUMPTION::PRISMATIC-MESH}, \ref{ASM::COVERING}, and~\ref{ASM::SHAPE-REGULARITY} on the space--time mesh~$\Th$ be satisfied, and the stabilization function~$\eta_F$ be given by~\eqref{EQN::STABILIZATION-TERM}. 
Assume that the exact solution~$u$ to the continuous weak formulation~\eqref{EQN::CONTINUOUS-WEAK-FORMULATION} satisfies: $u \in X$ and~$u_{|_K} \in H^{l_K}(K)$ ($l_K > 3/2$),  for each~$K \in \Th$, and let~$\uh \in \Vp(\Th)$ be the solution to the space--time LDG formulation~\eqref{EQN::REDUCED-VARIATIONAL-DG}. Then, the following estimates hold:
\begin{alignat}{2}
\label{EQN::ERROR-LDGN-PK}
\Tnorm{u - \uh}{\LDGN}^2  \lesssim & \sum_{K \in \Th} \frac{\hK^{2\ellK - 2}}{\pK^{2 l_K - 3}} \mathfrak{G}_K,\\
\label{EQN::ERROR-LDGp-PK}
\Tnorm{u - \uh}{\LDGp}^2  \lesssim & \sum_{K \in \Th} \frac{\hK^{2\ellK - 2}}{\pK^{2 l_K- 3}} \mathfrak{H}_K,
\end{alignat}
where~$\ell_K := \min\{\pK + 1, l_K\}$ for each~$K \in \Th$, and
\begin{alignat*}{3}
\mathfrak{G}_K & := \big(\hK\pK^{-2} + \pK^{-1} + (\hK/\hhKt)(\hpK^2/\pK^2) + \big(\max_{F \in \FKtime} \eta_F \big) \frac{\hK}{\pK^{2}} \big) \Norm{\frakE u}{H^{l_K}(\calK)}^2\\
& \quad + \Big( \frac{1}{p_K} + \frac{1}{\pK^3}\Big) \Norm{\frakE \nablax u}{H^{l_K - 1}(\calK)^d}^2, \\
\mathfrak{H}_K & := \big(\hK\pK^{-2} + \pK^{-1} +  \big(\max_{F \in \FKtime} \eta_F \big) \frac{\hK}{\pK^{2}} + \pK^{-3} + \hK\pK^{-3} + (\hK^2/\hhKt)(\hpK^2/\pK^3) \big)\Norm{\frakE u}{H^{l_K}(\calK)}^2
 \\
& \quad + \Big( \frac{1}{p_K} + \frac{1}{\pK^3}\Big) \Norm{\frakE \nablax u}{H^{l_K - 1}(\calK)^d}^2.
\end{alignat*}
\end{theorem}
\begin{proof}
We postpone the proof of this theorem to Appendix~\ref{APP::PROOF-ENERGY-PK}.
\end{proof}

\begin{corollary}
Let the hypotheses of Theorem \ref{THM::ERROR-ENERGY-PK} hold and assume also uniform
elemental polynomial degrees $p_K=p\ge1$ for all $K\in\Th$, the local quasi-uniformity condition \eqref{EQN::LOCAL-QUASI-UNIFORMITY}, and the orthotropic scaling
\eqref{EQN::ORTHOTROPIC-RELATION}. If the continuous solution~$u$ to~\eqref{EQN::CONTINUOUS-WEAK-FORMULATION} belongs to~$ H^l(\QT)$ for~$l > 3/2$, then the following estimates hold:
\begin{alignat}{2}
\nonumber
\Tnorm{u - \uh}{\LDGN}  \lesssim & \frac{h^{\ell - 1}}{p^{l - \frac32}}  \Norm{u}{H^{l}(\QT)}, \\
\nonumber
\Tnorm{u - \uh}{\LDGp}  \lesssim &\frac{h^{\ell - 1}}{p^{l- \frac32}} \Norm{u}{H^{l}(\QT)},
\end{alignat}
where~$h = \max_{K\in\Th} \hK$ and $\ell = \min\{p + 1, l\}$.
\end{corollary}

\begin{remark}[Exponential convergence]
Unfortunately, the use of the Stein extension operator~$\frakE$ in the error analysis above prevents the proof of exponential convergence for the~$hp$-version of the method,
as the hidden constant in~\eqref{EQN::STEIN} depends in an unknown (and possibly bad) way on the Sobolev index~$\ell$.
\eremk
\end{remark}

\subsection{Quasi-Trefftz polynomials \label{SECT::QUASI-TREFFTZ}}
We consider a fixed polynomial degree~$p \in\IN$. For each~$K \in \Th$, given~$(\bx_K,t_K) \in K$, we define the \emph{polynomial quasi-Trefftz space} for the homogeneous equation~$\mathcal{H}u:=\dpt u -\nabla_{\bx} \cdot (\bk\nabla_{\bx}u) = 0$ in $K$ as
\begin{equation}
\label{DEF::QTK}
	\VpK(K) = \qT^p(K):=\big\{ v\in \IP^{p}(K) \, : \, D^{(\mi_{\bx},i_t)} \mathcal{H} v (\bx_K,t_K)=0 \quad \forall (\mi_{\bx},i_t)\in \IN^{d+1}_0, |\mi_{\bx}|+i_t\leq p-2 \big\},
\end{equation}
where the multi-index derivative~$D^{\mi_{\bx}, i_t} w := \partial_{x_1}^{i_1} \ldots \partial_{x_d}^{i_d} \partial_t^{i_t} w$.
The quasi-Trefftz space does not satisfy the inclusion condition in Assumption~\ref{ASM::TIME-DERIVATIVE-ASSUMPTION}.

The basis functions for the space~$\qT^p(K)$ can be constructed 
as discussed in~\cite[\S2.4]{IG_Moiola_Perinati_Stocker:2024}.
We briefly summarize the recursive procedure to compute the coefficients of the basis functions in Appendix~\ref{app:QTbasis}.
The dimension of the quasi-Trefftz space is given by
\begin{align} \label{eq:dimqt}
    \dim(\qT^p(K)) = \binom{p+d}{d} + \binom{p - 1 + d}{d},
\end{align} 
which, for large~$p$, behaves like~$p^d$, whereas~$\dim(\Pp{p}{K}) \approx p^{d+1}$ and~$\dim(\IQ^{p}(K)) \approx p^{d + 1}$. This represents a significant reduction of the total number of degrees of freedom.

Additionally, the quasi-Trefftz method can handle nonhomogeneous source term~$f$ by constructing an element-wise approximate particular solution~$u_{h,f}$ and homogenizing the system, see \cite[\S5]{IG_Moiola_Perinati_Stocker:2024} for more details.

We make the following stronger assumption on the mesh~$\Th$, which implies Assumption~\ref{ASSUMPTION::PRISMATIC-MESH} (see~\cite[Rem.~2.2]{Cangiani_Dong_Georgoulis:2017}).

\begin{assumption}[Uniform star-shapedness]
\label{ASM::UNIFOMR-STARSHAPEDNESS}
There exists~$ 0 < \rho \leq 1 $ independent of~$h$ such that each element~$K \in \Th$ is
star-shaped with respect to the ball centered at~$(\bx_K , t_K ) \in K$ and with radius~$\rho h_K$.
\end{assumption}

The approximation properties of the space~$\qT^p(K)$ come from the fact that it is defined so that the Taylor polynomial of order~$p + 1$ (and degree~$p$) centered at~$(\bx_K, t_K)$ of the continuous solution~$u$,
which we denote by~$T_{(\bx_K, t_K)}^{p + 1} u$, belongs to~$\qT^p(K)$; see~\cite[Thm.~2.4]{IG_Moiola_Perinati_Stocker:2024} and~\cite[Prop.~4]{Gomez_Moiola:2024}.
More precisely, if Assumption~\ref{ASM::UNIFOMR-STARSHAPEDNESS} holds, and~$u \in C^{p + 1}(K)$ for each~$K \in \Th$, then (see, e.g., \cite[Cor.~3.19]{Callahan_2010})
\begin{equation}
\label{EQN::ERROR-TAYLOR}
    \inf_{P\in\qT^p(K)}|u-P|_{C^q(K)} \leq 	|u-T_{(\bx_K,t_K)}^{p+1}[u]|_{C^q(K)} 
    \lp h_K^{p+1-q} |u|_{C^{p+1}(K)} \quad \forall q\in\mathbb{N}_0,\, q\le p.
\end{equation}

\begin{theorem}[Error estimate in the norm~$\Tnorm{\cdot}{\LDGN}$]
\label{THM::ERROR-ENERGY-QTK}
Given~$p\in\IN$ with~$p \geq 1$, let~$u \in X$ be the continuous solution to~\eqref{EQN::CONTINUOUS-WEAK-FORMULATION} and~$u_h\in \qT^p(\Th) + u_{h,f}$ be the solution to the LDG method~\eqref{EQN::VARIATIONAL-DG} with~$\VpK(K)$ as in~\eqref{DEF::QTK} and~$\MpK(K)$ such that Assumption~\ref{ASM::COMPATIBILITY-CONDITION} holds. 
Let also the stabilization function~$\eta_F$ be given by~\eqref{EQN::STABILIZATION-TERM}.
Under Assumption~\ref{ASM::UNIFOMR-STARSHAPEDNESS}, the local quasi-uniformity conditions in~\eqref{EQN::LOCAL-QUASI-UNIFORMITY}, and the orthotropic scaling in~\eqref{EQN::ORTHOTROPIC-RELATION}, if~$u_{|_K} \in C^{p+1}(K)$ for all~$K \in \Th$, then
\begin{equation}\label{EQN::ERROR-LDGN-QT}
    \Tnorm{u - u_h}{\LDGN}^2 \lp
 \sum_{K \in \Th} h_K^{2p} \abs{u}_{C^{p+1}(K)}^2,
\end{equation}
where the hidden constant may also depend on the maximum number of time-like facets of the elements in~$\Th$.
\end{theorem}
\begin{proof}
The proof of this theorem is postponed to Appendix~\ref{APP::PROOF-ENERGY-QTK}.
\end{proof}

\subsection{Embedded Trefftz polynomials\label{SECT::EMBEDDED-TREFFTZ}}
The embedded Trefftz method, introduced in \cite{LS_IJMNE_2023}, circumvents the explicit construction of Trefftz basis functions by embedding the Trefftz space into some standard polynomial space. 
The construction of the embedding relies on solving a small element-wise singular value problem. 
The method can also handle nonhomogeneous source terms~$f$ by constructing an element-wise particular solution using the pseudo-inverse of the already computed singular value problem.
Compared to the quasi-Trefftz space considered in the previous section, no Taylor expansion of the coefficients or of the source term is needed.

Via the embedded procedure, the discrete space $\VpK(K)$ 
is chosen as a weak Trefftz space as follows:
\begin{equation}
\label{DEF::ETK}
    \VpK(K) = \eT^p(K) :=\{v\in \Pp{p}{K}\, : \,  \Pi^{p-2} (\mathcal{H} v) = 0 \text{ in } K\},
\end{equation}
where~$\Pi^{p-2}$ is the~$L^2(K)$-orthogonal projection operator in the space~$\Pp{p - 2}{K}$, and~$\mathcal{H}$ is as in the previous section. 

The choice of the projection operator is crucial for the approximation properties of the embedded Trefftz space.
Different choices for this operator have been discussed in~\cite[Ch.~3]{ma_ch}.
The choice made here is heuristic; however, in the numerical results in Section~\ref{SECT::NUMERICAL-EXPERIMETNS} below, we observe optimal convergence rates and the same reduction of the number of degrees of freedom as for the quasi-Trefftz space, i.e., the same space dimension for~$\eT^p(K)$ as given in \eqref{eq:dimqt}.
Deriving approximation properties for embedded Trefftz spaces is a nontrivial and problem-dependent task; see, e.g.,\cite{Lozinski_2019,LLSV_ARXIV_2024}.

\section{Numerical results\label{SECT::NUMERICAL-EXPERIMETNS}}
In this section, we present some numerical experiments in $(1 + 1)$  and $(2+1)$ dimensions
to validate the theoretical results and assess numerically some additional features of the proposed method. 
Some implementation details are given in Section \ref{SECT::IMPLEMENTATION}. In Section~\ref{SECT::CONDITIONING}, we study numerically the condition number of the stiffness matrix. The accuracy
of the method is tested in Sections~\ref{SECT::SMOOTH-SOLUCTION} and~\ref{SECT::SINGULAR-SOLUTION} for smooth and singular solution, respectively.

The LDG scheme has been implemented using \texttt{NGSolve} \cite{ngsolve} and \texttt{NGSTrefftz}~\cite{ngstrefftz}.\footnote{Reproduction material is available in \cite{gomez_2024_14191529}.}
In all the numerical experiments below, the stabilization parameter~$\eta^{\star}$ is set to~$10^{-1}$ and the weight parameter~$\alpha_F$ is set to~$1/2$ for all the time-like facets.

\subsection{Implementation details}\label{SECT::IMPLEMENTATION}
We denote by~$\MT,\, \D,\, B$, and~$\Su$ the matrices associated with the bilinear forms~$\mh(\cdot, \cdot)$,\ $\bdh(\cdot, \cdot)$,\ $\bh(\cdot, \cdot)$,\ $\suh(\cdot, \cdot)$, respectively. We also denote by~$\bq,\, \bu$ the vectors associated with the linear functionals~$\lqh(\cdot)$ and~$\luh(\cdot)$, respectively. The variational problem~\eqref{EQN::VARIATIONAL-DG} can be written in matrix form as
\begin{align*}
\D\Qh - B \Uh & = \bq,\\ 
 (\MT + \Su)\Uh + B^T \Qh & = \bu.
\end{align*}
The equivalent matrix formulation of the reduced variational formulation~\eqref{EQN::REDUCED-VARIATIONAL-DG} reads
\begin{equation}
\label{EQN::REDUCED-LINEAR-SYSTEM}
    \left(\MT + \A \right)\Uh = \bu - B^T\D^{-1} \bq,
\end{equation}
where~$\A = \Su + B^T\D^{-1}B$ and~$\mathcal{B} := \MT + \A$ are the matrix representation of the bilinear forms~$\Ah(\cdot, \cdot)$ and~$\Bh(\cdot, \cdot)$, respectively.

\begin{remark}[Implicit time-stepping through time-slabs]
\label{REM::TIME-SLABS}
If the mesh elements can be collected in $N$ time-slabs, i.e., sets of the form~$Q_n := \Omega\times (t_{n-1},t_n)$ with $0 = t_0 < t_1 <\dots < t_N = T$, then, the structure of the terms on the space-like facets resulting from the use of upwind numerical fluxes in the variational formulation \eqref{EQN::REDUCED-VARIATIONAL-DG}
allows us to compute the discrete solution on the time-slab~$Q_n$ from the discrete solution on the previous time-slab~$Q_{n - 1}$. 
Hence, the global linear system \eqref{EQN::REDUCED-LINEAR-SYSTEM} can be solved as a sequence of $N$ smaller linear systems of the form~$A^{(n)} U_h^{(n)} = l^{(n)}$ for $1\leq n\leq N$, where $l^{(n)}=B^{(n)} U_h^{(n-1)}$ for $2\leq n \leq N$.
This is equivalent to an implicit time-stepping through time-slabs.
\eremk
\end{remark}

\subsection{Conditioning\label{SECT::CONDITIONING}}
We first study numerically the condition number
of the stiffness matrix. 
To do so, we consider the model problem~\eqref{EQN::MODEL-PROBLEM} in~$(1 + 1)$ dimensions with homogeneous Dirichlet boundary conditions~$(\gD = 0)$ and~$\bk = 1$ on the space--time domain~$\QT = (0, 1)^2$. 

We compute the 2-condition number of the stiffness matrix 
defined in Remark~\ref{REM::TIME-SLABS} (which in this case is the same for all time slabs) for a sequence of meshes with uniform distributions along the space and time directions and~$\hKt = \hKx = 2^{-i}$, $i = 0, \dots, 6$, and uniform polynomial degrees~$p = 2, 3, 4$.
The results are shown in Figure~\ref{fig::cond}, where we observe an asymptotic growing behavior of order~$\mathcal{O}(h^{-1})$ for all four discrete spaces and 
polynomial degrees. Similar results were obtained for a space--time DG discretization of the linear Sch\"odinger equation in~\cite[\S4]{Gomez_Moiola_Perugia_Stocker:2023}.

We have employed Legendre bases for the tensor-product space. Since the use of monomial bases for standard and quasi-Trefftz polynomials leads to ill-conditioned stiffness matrices, we applied a Gram--Schmidt orthogonalization procedure to improve their conditioning. As for the embedded Trefftz space, the \emph{orthogonal embedding} used in its construction ensures that the condition number cannot exceed that for the underlying polynomial space (see~\cite[Lemma~1]{LS_IJMNE_2023}).
\begin{figure}[htb!]
	\begin{center}    
		\resizebox{0.75\linewidth}{!}{
			\begin{tikzpicture}
				\begin{groupplot}[%
					group style={%
						group name={my plots},
						group size=2 by 2,
						horizontal sep=2cm,
                        vertical sep=2.5cm,
					},
					legend style={
						legend columns=4,
						at={(0.88, 0.12)},
                       font=\footnotesize
					},
					ymajorgrids=true,
					grid style=dashed,
					cycle list name=colorscond,
					]
        \CycleNextGruoupPlotcond{cart}{cond}
        {LDGHeat1D_condplot_new.csv}
        {eoc_cond}{$\mathrm{cond}_2(A^{(n)})$}{Tensor-product polynomials~\S\ref{SECT::TENSOR-PRODUCT}}
         \CycleNextGruoupPlotcond{poly}{cond}
        {LDGHeat1D_condplot_new.csv}
        {eoc_cond}{$\mathrm{cond}_2(A^{(n)})$}{Standard polynomials~\S\ref{SECT::STANDARD-POLY}}
       \CycleNextGruoupPlotcond{qtrefftz}{cond}
        {LDGHeat1D_condplot_new.csv}
       {eoc_cond}{$\mathrm{cond}_2(A^{(n)})$}{Quasi-Trefftz polynomials~\S\ref{SECT::QUASI-TREFFTZ}} 
        \CycleNextGruoupPlotcond{embt}{cond}
        {LDGHeat1D_condplot_new.csv}
        {eoc_cond}{$\mathrm{cond}_2(A^{(n)})$}{Embedded Trefftz polynomials~\S\ref{SECT::EMBEDDED-TREFFTZ}}
		\end{groupplot}   
		\end{tikzpicture}}
	\end{center}   
\caption{
Condition number of the stiffness matrix in~$(1 + 1)$ dimensions for
the four discrete spaces presented 
in Section~\ref{SECT::DISCRETE-SPACES}.
The numbers in the yellow boxes are the empirical algebraic rates.}
\label{fig::cond}
\end{figure}

\subsection{Smooth solution\label{SECT::SMOOTH-SOLUCTION}}
We now focus on the convergence of the method. We consider the~$(2 + 1)$-dimensional 
problem~\eqref{EQN::MODEL-PROBLEM}  with~$\bk = \mathrm{Id}_2$ on the space--time domain~$\QT=(0, 1)^2 \times (0, 1)$. We 
consider homogeneous boundary conditions ($g_D = 0$), and set the initial condition~$u_0$ and the source term~$f$ such that the exact solution~$u$ 
to~\eqref{EQN::MODEL-PROBLEM} is given by
\begin{equation}\label{EQ::NUMEXP-SMOOTHSOL}
 u(x, y, t) = e^{-t} \sin(\pi x) \sin(\pi y)  \quad \text{in } (0,1)^2 \times (0,1).
\end{equation}

We consider the four discrete spaces presented 
in Section~\ref{SECT::DISCRETE-SPACES}, i.e., the tensor-product, standard, quasi-Trefftz, and embedded Trefftz polynomial spaces. 
The numerical results obtained are compared in terms of~$h$-convergence (Section~\ref{SECT::HCONV}) and $p$-convergence (Section~\ref{SECT::PCONV}).

\subsubsection{\texorpdfstring{$h$}{h}-convergence \label{SECT::HCONV}}
We consider a uniform degree distribution with~$p = 2,\, 3,\, 4$, a sequence of unstructured simplicial meshes in space, and uniform partitions along the time direction with~$h_t \simeq h_{\bx}$.

In Figure \ref{fig:hconv1+1D}, we show 
the errors obtained for the four discrete spaces presented in Section~\ref{SECT::DISCRETE-SPACES}.
We observe convergence rates of order~$\mathcal{O}(h^p)$ for the error in the energy norms~$\Tnorm{\cdot}{\LDGN}$ and~$\Tnorm{\cdot}{\LDGp}$, and of order~$\mathcal{O}(h^{p+1})$ for the error in the~$L^2(\QT)$ norm for all the discrete spaces.  
Although the tensor-product and the standard polynomial spaces are richer, 
no significant loss of accuracy is observed for the quasi-Trefftz and the embedded Trefftz spaces. 
Moreover, the advantages of the 
Trefftz-type spaces are shown in Section~\ref{SECT::PCONV}, where the error is compared in terms of the total number of degrees of freedom.
\begin{figure}[!htb]
	\begin{center}    
		\resizebox{\linewidth}{!}{
			\begin{tikzpicture}
				\begin{groupplot}[%
					group style={%
						group name={my plots},
						group size=3 by 3,
						horizontal sep=3em,
						vertical sep=3em,
					},
					legend style={
						legend columns=1,
						at={(0.98,0.37)},
					},
					ymajorgrids=true,
					grid style=dashed,
					cycle list name=colorsh,
					]
				\CycleNextGruoupPloth{2}{error_LDGN}{LDGHeat2d2.csv}{eoc_LDGN}{$\Tnorm{u-\uh}{\LDGN}$}{$p=2$}{}
				\CycleNextGruoupPloth{3}{error_LDGN}{LDGHeat2d2.csv}{eoc_LDGN}{}{$p=3$}{}
				\CycleNextGruoupPloth{4}{error_LDGN}{LDGHeat2d2.csv}{eoc_LDGN}{}{$p=4$}{}
				\CycleNextGruoupPloth{2}{error_LDGp}{LDGHeat2d2.csv}{eoc_LDGp}{$\Tnorm{u-\uh}{\LDGp}$}{}{}
				\CycleNextGruoupPloth{3}{error_LDGp}{LDGHeat2d2.csv}{eoc_LDGp}{}{}{}
				\CycleNextGruoupPloth{4}{error_LDGp}{LDGHeat2d2.csv}{eoc_LDGp}{}{}{}
				\CycleNextGruoupPloth{2}{error_L2}{LDGHeat2d2.csv}{eoc_L2}{$\Norm{u-\uh}{L^2(\QT)}$}{}{$h$}
				\CycleNextGruoupPloth{3}{error_L2}{LDGHeat2d2.csv}{eoc_L2}{}{}{$h$}
				\CycleNextGruoupPloth{4}{error_L2}{LDGHeat2d2.csv}{eoc_L2}{}{}{$h$}
				\end{groupplot}   
		\end{tikzpicture}}
	\end{center}  
 \caption{$h$-convergence for the $(2 + 1)$-dimensional problem with exact solution~$u$ in~\eqref{EQ::NUMEXP-SMOOTHSOL}.
 The error is measured in the norms~$\Tnorm{\cdot}{\LDGN}$, $\Tnorm{\cdot}{\LDGp}$, and $\Norm{\cdot}{L^2(\QT)}$, in the corresponding rows.
 The columns correspond to polynomial degree $p=2,3,4$. 
 The numbers in the yellow boxes are the empirical algebraic convergence rates corresponding to the embedded Trefftz space.}
 \label{fig:hconv1+1D}
\end{figure}

\subsubsection{\texorpdfstring{$p$}{p}-convergence}\label{SECT::PCONV}
We now study the 
$p$-version of the method, i.e., when increasing the polynomial degree~$p$ for a fixed space--time mesh.
We denote by~$\mathrm{N}_{\mathrm{DoFs}}$ 
the total number of degrees of freedom. 

In Figure~\ref{fig:pconv1+1D}, we compare the errors obtained 
with the four choices of the discrete space~$\Vp(\Th)$ analyzed in Section~\ref{SECT::DISCRETE-SPACES}, for a coarse space--time mesh and~$p = 2, \ldots, 6$. We see that the quasi-Trefftz and the embedded Trefftz versions of the 
method lead to a higher accuracy for comparable number of degrees of freedom, especially for high polynomial degrees~$p$.
Moreover, we observe exponential decay of the error of order~$\mathcal{O}(e^{-b\sqrt{\mathrm{N}_{\mathrm{DoFs}}}})$ for the quasi-Trefftz and embedded Trefftz polynomial spaces, whereas only exponential decay of the error of order $\mathcal{O}(e^{-c\sqrt[3]{\mathrm{N}_{\mathrm{DoFs}}}})$ is expected for the tensor-product and the standard polynomial spaces. 
\begin{figure}[ht!]
	\begin{center}    
		\resizebox{\linewidth}{!}{
			\begin{tikzpicture}
				\begin{groupplot}[%
					group style={%
						group name={my plots},
						group size=3 by 1,
						horizontal sep=5em,
      	                vertical sep=5em,
					},
					legend style={
						legend columns=1,
						at={(0.98,0.98)},
					},
					ymajorgrids=true,
					grid style=dashed,
					cycle list name=colorsp,
					]
\CycleNextGruoupPlotdofs{error_LDGN}{LDGHeat2D2t_pplot.csv}{$\Norm{u-\uh}{\LDGN}$}
\CycleNextGruoupPlotdofs{error_LDGp}{LDGHeat2D2t_pplot.csv}{$\Norm{u-\uh}{\LDGp}$}
\CycleNextGruoupPlotdofs{error_L2}{LDGHeat2D2t_pplot.csv}{$\Norm{u-\uh}{L^2(\QT)}$}
\end{groupplot}   
 \end{tikzpicture}}
	\end{center}   
	 \caption{$p$-convergence in the norms~$\Norm{\cdot}{\LDGN}$ (left panel), $\Norm{\cdot}{\LDGp}$ (central panel), and~$\Norm{\cdot}{L^2(\QT)}$ (right panel) for the $(2+ 1)$-dimensional problem with exact solution~$u$ in~\eqref{EQ::NUMEXP-SMOOTHSOL}.}
 \label{fig:pconv1+1D}
\end{figure}

\subsection{Singular solutions\label{SECT::SINGULAR-SOLUTION}}
In this section, we study the convergence of the $h$- and $hp$-versions of the method for singular solutions.
In Section~\ref{SECT::SINGULAR-SOLUTION}, we consider a~$(2+1)$-dimensional problem with an initial layer, and, in Section~\ref{SECT::INCOMPATIBLE}, we consider a~$(1 + 1)$-dimensional problem with incompatible initial and boundary conditions.

\subsubsection{Singularity at initial time
\label{SECT::INITIAL-LAYER}}
First, we consider the 
numerical experiment in~\cite[\S 6.2]{Cangiani_Dong_Georgoulis:2017}, where the exact solution has an initial layer.
Let the space--time cylinder~$\QT = (0, 1)^2 \times (0, 0.1)$, 
$\bk = \mathrm{Id}_2$, $g_D = 0$, and~$u_0$ and~$f$ be such that the exact solution~$u$ to~\eqref{EQN::MODEL-PROBLEM} is given by
\begin{equation}\label{eq::singular_solution}
    u(x, y, t) = t^{\alpha} \sin(\pi x) \sin(\pi y) \qquad (\alpha = 0.75),
\end{equation} 
which belongs to the space~$H^{5/4 - \varepsilon}(0, T; C^{\infty}(\Omega))$ for all~$\varepsilon > 0$.

For the $h$-version of the method, we consider a uniform degree distribution with~$p = 2$, a
sequence of unstructured simplicial meshes in space, and uniform partitions along the time direction with~$h_t \simeq h_{\bx}$.
As for the~$hp$-version of the method, we employ 
a refinement strategy similar to the one used in~\cite[Example~2]{Cangiani_Dong_Georgoulis:2017}.
We first define a class of temporal meshes with~$t_n = \sigma^{N-n}\times 0.1$, 
$n = 1, \dots, N$, geometrically
graded towards $t_0 = 0$ with mesh grading factor~$0 < \sigma < 1$. 
Let~$\mu$ be 
a polynomial order factor, we consider temporally varying polynomial degrees,
starting from $p = 2$ on the elements belonging to the initial time slab $Q_1= \Omega \times (0, t_1)$, and linearly increasing~$p$ when moving away from~$t = 0$, according to $\pK = \lfloor \mu (n + 1)\rfloor$ for the elements~$K \in \Th$ that belong to the time slab $Q_n = \Omega \times (t_{n - 1}, t_n)$ for $n = 2, \dots, N$.
We choose~$\sigma = 0.25$, $\mu = 1$, and a fixed spatial mesh with~$h_{\bx} \approx 0.25$.
The results obtained with the~$h$- and $hp$-versions of the method are shown in Figure \ref{fig:initiallayer}. 
The $h$-version of the method exhibits only an algebraic decay of the error, whereas exponential decay  with respect to the fourth root of the total number of degrees of freedom is observed for the~$hp$-version of the method. No significant differences are observed between the four choices of discrete spaces.

\begin{figure}[ht!]
\begin{center}    
\resizebox{0.4\linewidth}{!}{
\begin{tikzpicture}
\begin{groupplot}[%
	group style={%
	group name={my plots},
	group size=2 by 1,
	vertical sep=3cm,
	horizontal sep=3cm,
					},
					legend style={
					legend columns=1,
                    legend pos=south east
					},
					ymajorgrids=true,
					grid style=dashed,
					cycle list name=colorshp,
					]    
  	\nextgroupplot[ymode=log, ylabel={$\Norm{u-\uh}{L^2(\QT)}$},xlabel={$\sqrt[4]{\mathrm{N}_{\mathrm{DoFs}}}$}
   ]
   \foreach \method in {cart,poly,qtrefftz,embt}{
\addplot+[discard if not={method}{\method}] table  [x=totaldofs4, y=error_L2, col sep=comma]
{initiallayermine.csv};
}
    \foreach \method in {cart,poly,qtrefftz,embt}{
	\addplot+[discard if not={method}{\method}] table  [x=totdofs4, y=error_L2, col sep=comma] 
{LDGHeat2D3t_hsingularmine.csv};}
    \legend{$\IQ$,$\IP$,$\qT$,$\eT$}
\end{groupplot}      
	\end{tikzpicture}}
	\end{center} 
 \caption{$h$-convergence with polynomial degree~$p = 2$ (dashed lines) and $hp$-convergence (continuous lines) of the method in the norm~$\Norm{\cdot}{L^2(\QT)}$ for the~$(2 + 1)$-dimensional problem with exact solution~$u$ in~\eqref{eq::singular_solution}.}
    \label{fig:initiallayer}
\end{figure}

\subsubsection{Singularity at the interface of the initial and boundary conditions\label{SECT::INCOMPATIBLE}}
We now consider the 
numerical experiment in~\cite[\S7.1]{Schotzau_Schwab:2000} (see also~\cite[\S5.4.1]{Gomez_Mascotto_Moiola_Perugia_2024} and~\cite[\S4]{Gomez_Mascotto_Perugia_2024}), a problem
with incompatible initial and boundary conditions.
We consider the heat equation~\eqref{EQN::MODEL-PROBLEM} in the space--time domain~$\QT = (0, 1) \times (0, 1)$ with 
$\bk=1$ and zero source term $(f = 0)$. For the initial condition~$u_0 = 1$ and homogeneous Dirichlet boundary conditions~$(g_D = 0)$, the exact solution~$u$ can be represented by the following Fourier series: 
\begin{equation}
\label{eq::geometricxt}
    u(x,t) = \sum_{n=0}^{\infty} \frac{4}{(2n+1)\pi} \sin((2n+1)\pi x) \exp(-(2n+1)^2\pi^2 t).
\end{equation} 
The above function~$u$ belongs to the space~$L^2(0, T; H^{3/2 - \varepsilon}(\Omega) \cap H_0^1(\Omega)) \cap H^{3/4 - \varepsilon}(0, T; L^2(\Omega))$ for all~$\varepsilon > 0$, 
and its time derivative~$\partial_x u$ belongs to~$L^2(0, T; H^{1/2-\varepsilon}(\Omega)) \cap H^{1/4 - \varepsilon}(0, T; L^2(\Omega))$ for all~$\varepsilon > 0$.
The series is truncated at~$n = 500$
for the computation of the error.

For the $h$-version of the method, we consider a mesh sequence with uniform partitions along the space and time directions and~$\hKt = \hKx = 2^{-i}$, $i=2, \dots, 8$, and uniform polynomial degree~$p = 2$.
In Figure \ref{fig::incompatible_algebraic}, we observe convergence rates of order~$\mathcal{O}(h^{1/4})$ for the error in the norm $\Norm{\cdot}{\LDG}$, and of order~$\mathcal{O}(h^{3/4})$ for the error in the norm~$\Norm{\cdot}{L^2(\QT)}$, which are in agreement with the regularity of~$u$ and~$\partial_x u$. 
 
As for the~$hp$-version of the method, we employ the 
refinement strategy  
in~\cite[\S4.2]{Gomez_Mascotto_Perugia_2024}.
We consider a sequence of space–time meshes geometrically graded towards $x = 0$, $x = 1$ and $t = 0$ with grading factors $\sigma_x = \sigma_t = 0.35$.
For a given number of $N$ time slabs we set $t_n=\sigma_t^{N-n}$, for $n = 1,\dots, N$, and compute until final time $t_N=1$.
Furthermore, we linearly increase the polynomial degrees when moving away from~$t = 0$.
For the elements~$K \in \Th$ that belong to the time slab~$Q_n = \Omega \times (t_{n - 1}, t_n)$ we set $\pK = n + 1$, for $n = 1,\dots, N$.
The results obtained with the $h$- and $hp$-version of the method are 
presented in Figure~\ref{fig::incompatible_hp}. 
The $h$-version of the method exhibits an algebraic decay of the error, whereas exponential decay of the errors is observed for the $hp$-version with respect to the cubic root of the total number of degrees of freedom. As for the numerical experiment in Section~\ref{SECT::INITIAL-LAYER}, no significant differences are observed
for the four choices of discrete spaces.
\begin{figure}[ht!]
	\begin{center}    
		\resizebox{0.75\linewidth}{!}{
			\begin{tikzpicture}
				\begin{groupplot}[%
					group style={%
						group name={my plots},
						group size=3 by 2,
						vertical sep=6em,
						horizontal sep=6em,
					},
					legend style={
					legend columns=1,
					at={(0.25,0.98)},
					},
					ymajorgrids=true,
					grid style=dashed,
					cycle list name=colorshsing,
					]    
     \nextgroupplot[ymode=log,xmode=log, ylabel={$\Norm{u-\uh}{\LDG}$},xlabel={$h$}]
\foreach \method in {cart,poly,qtrefftz,embt}{
\addplot+[discard if not={method}{\method},discard if ={h}{1.414213562373091},discard if ={h}{0.7071067811865455}] table  [x=h, y=error_LDG, col sep=comma] 
      {LDGHeat1D3t_hsingularmine.csv};
      }
      \foreach \method in {cart,poly,qtrefftz,embt}{
           	\addplot+[discard if not={method}{\method},discard if ={h}{1.414213562373091},discard if ={h}{0.7071067811865455},discard if={eoc_LDG}{0}, only marks,
	visualization depends on=\thisrow{eoc_LDG} \as \labela,
	nodes near coords=\pgfmathprintnumber{\labela}
	,
	every node near coord/.append style={
		black,
		draw=yellow!30,
		ellipse,
		fill=yellow!30,
		inner sep=1pt,
		xshift=3ex,
		yshift=3ex,
		scale=0.5,/pgf/number format/fixed,
        /pgf/number format/precision=2,/pgf/number format/fixed zerofill}
	] table [x=h, y=error_LDG, col sep=comma] 
  {LDGHeat1D3t_hsingularmine.csv};
  }
      \legend{$\IQ$,$\IP$,$\qT$,$\eT$}
\nextgroupplot[ymode=log,xmode=log, ylabel={$\Norm{u-\uh}{L^2(\QT)}$},xlabel={$h$}]
					\foreach \method in {cart,poly,qtrefftz,embt}{
						\addplot+[discard if not={method}{\method},discard if ={h}{1.414213562373091},discard if ={h}{0.7071067811865455}] table  [x=h, y=error_L2, col sep=comma] 
      {LDGHeat1D3t_hsingularmine.csv};
      }
      \foreach \method in {cart,poly,qtrefftz,embt}{
           	\addplot+[discard if not={method}{\method},discard if ={h}{1.414213562373091},discard if ={h}{0.7071067811865455},discard if={eoc_L2}{0}, only marks,
	visualization depends on=\thisrow{eoc_L2} \as \labela,
	nodes near coords=\pgfmathprintnumber{\labela}
	,
	every node near coord/.append style={
		black,
		draw=yellow!30,
		ellipse,
		fill=yellow!30,
		inner sep=1pt,
		xshift=3ex,
		yshift=3ex,
		scale=0.5,/pgf/number format/fixed,
        /pgf/number format/precision=2,/pgf/number format/fixed zerofill}
	] table [x=h, y=error_L2, col sep=comma] 
  {LDGHeat1D3t_hsingularmine.csv};
  }
      \legend{$\IQ$,$\IP$,$\qT$,$\eT$}
\end{groupplot}      
	\end{tikzpicture}}
	\end{center} 
 \caption{$h$-convergence of the method in the norms $\Tnorm{\cdot}{\LDG}$ and $\Norm{\cdot}{L^2(\QT)}$ for the $(1 + 1)$-dimensional problem with exact solution~$u$ in~\eqref{eq::geometricxt}. The numbers in the yellow boxes are the empirical
algebraic convergence rates.}
 \label{fig::incompatible_algebraic}
\end{figure}

\begin{figure}[ht!]
\begin{center}    
\resizebox{0.4\linewidth}{!}{\begin{tikzpicture}
\begin{groupplot}[%
					group style={%
						group name={my plots},
						group size=2 by 1,
						vertical sep=3cm,
						horizontal sep=3cm,
					},
					legend style={
					legend columns=1,
					legend pos=north east,
					},
					ymajorgrids=true,
					grid style=dashed,
					cycle list name=colorshp,
                    ymin=1e-4, ymax=0.2
					]    
  	\nextgroupplot[ymode=log, ylabel={$\Norm{u-\uh}{L^2(\QT)}$},xlabel={$\sqrt[3]{\mathrm{N}_{\mathrm{DoFs}}}$}]
		\foreach \method in {cart,poly,qtrefftz,embt}{
		\addplot+[discard if not={method}{\method}] table  [x=totaldofs3, y=error_L2, col sep=comma] 
     {LDGHeat1D_incompatiblehp.csv};
  }
    \foreach \method in {cart,poly,qtrefftz,embt}{
    \addplot+[dashed, discard if not={method}{\method}] table  [x=totdofs3, y=error_L2, col sep=comma]   
    {LDGHeat1D3t_hsi.csv}
    ;}
   \legend{$\IQ$,$\IP$,$\qT$,$\eT$}
\end{groupplot}      
	\end{tikzpicture}}
	\end{center} 
 \caption{
 $h$-convergence with polynomial degree~$p = 2$ (dashed lines) and $hp$-convergence (continuous lines) of the method in the norm~$\Norm{\cdot}{L^2(\QT)}$ for the~$(1 + 1)$-dimensional problem with exact solution~$u$ in~\eqref{eq::geometricxt}.
 }
\label{fig::incompatible_hp}
\end{figure}

\section{Conclusions \label{SECT::CONCLUSIONS}}
We introduced a space--time Local Discontinuous Galerkin 
method for the discretization of the heat equation. The method is well posed in any space dimension~$d\in \{1, 2, 3\}$  for very general prismatic space--time meshes and discrete spaces, 
even if polynomial inverse estimates
are not available. 
Moreover, for piecewise polynomial spaces satisfying an additional mild condition, we showed an inf-sup condition that provides an additional control of the time derivative of the discrete solution. 
We have also derived~$hp$-a priori
error bounds in some energy norms, and proven~$hp$-error estimates for
standard and tensor-product polynomial spaces, and~$h$-error estimates for quasi-Trefftz spaces.
In the numerical experiments presented, we have observed 
optimal convergence rates of order~$\mathcal{O}(h^{p})$ for the error in the energy norms, and of order~$\mathcal{O}(h^{p+1})$ for the error in the~$L^2(\QT)$ norm for the four choices of discrete spaces presented: tensor-product, standard, quasi-Trefftz, and embedded Trefftz polynomial spaces.
The two 
latter spaces allow for a significant reduction of the number of degrees of freedom.
We also have assessed the performance of the~$hp$-version of the method for some singular solutions.

\section*{Acknowledgements}
We are grateful to Ilaria Perugia (University of Vienna) for her very valuable
suggestions to improve the presentation of this work, and to Lorenzo Mascotto (University of Milano-Bicocca) for the helpful discussions on~$hp$-convergence.
The first author acknowledges support from the Italian Ministry of University and Research through the project PRIN2020 ``Advanced polyhedral discretizations of heterogeneous PDEs for multiphysics problems", and from the INdAM-GNCS through the
project CUP E53C23001670001.
The second author acknowledges support from the PRIN project ``ASTICE" (202292JW3F) and from INdAM-GNCS.
This research was funded in part by the Austrian Science Fund (FWF) 
\href{https://doi.org/10.55776/F65}{10.55776/F65} and 
\href{https://doi.org/10.55776/ESP4389824}{10.55776/ESP4389824}.
For open access purposes, the authors have applied a CC BY public copyright license to any author-accepted manuscript version arising from this submission.

\paragraph{Conflict of interest.}
The authors declare no competing interests.

\paragraph{Data availability.} Reproduction material is available in~\cite{gomez_2024_14191529}.

\appendix
\section{Construction of quasi-Trefftz basis}\label{app:QTbasis}

We describe the explicit construction of a basis for the polynomial quasi-Trefftz space defined in~\eqref{DEF::QTK}. 
For simplicity, we consider the case~$\bk = \kappa \mathrm{Id}$, for a strictly positive~$\kappa \in \mathbb{R}$, i.e., for the equation~$\dpt v -\kappa\Delta_\bx v =0$.

For all $(\mi_{\bx},i_t) \in \IN^{d+1}_0$ such that $|\mi_{\bx} |+i_t\leq p-2$, we have
\begin{align}\label{DerOp}
D^{(\mi_{\bx},i_t)}(\dpt v -\kappa \Delta_\bx v) =&D^{(\mi_{\bx},i_t)} \left(D^{(\bzero,1)}v-\kappa\sum_{j=1}^d D^{(2\ej,0)}v\right)=D^{(\mi_\bx,i_t+1)}v-\kappa\sum_{j=1}^d D^{(\mi_\bx+2\ej,i_t)}v.
\end{align}

Since each function $v\in\qT^p(K)$ is a polynomial of degree at most $p$, it can be expressed as a linear combination of scaled monomials centred at $(\bx_K,t_K)\in K$ as follows:
\begin{equation}\label{lincombmon}
	v(\bx,t)=\sum_{\substack{(\mj_\bx,j_t)\in\IN_0^{d+1},\\
 |\mj_\bx|+j_t\leq p}} a_{(\mj_\bx,j_t)} \left(\frac{\bx-\bx_K}{\hKx}\right)^{\mj_\bx}\left(\frac{t-t_K}{\hKt}\right)^{j_t}.
\end{equation}

Evaluating \eqref{DerOp} in $(\bx_K,t_K)$ and using that $D^{(\mj_\bx,j_t)} v(\bx_K,t_K)=a_{(\mj_\bx,j_t)} \frac{\mj_\bx ! j_t !}{\hKx^{\abs{\mj_\bx}}\hKt^{j_t}}$, we obtain
\begin{align*}
D^{(\mi_{\bx},i_t)}(\dpt v -\kappa \Delta_\bx v)(\bx_K,t_K) = & 
D^{(\mi_\bx,i_t+1)}v(\bx_K,t_K)-\kappa\sum_{j=1}^d D^{(\mi_\bx+2\ej,i_t)}v(\bx_K,t_K)\\=&
	a_{(\mi_\bx,i_t+1)}\frac{\mi_\bx ! (i_t+1) !}{\hKx^{\abs{\mi_\bx}}\hKt^{i_t+1}}-\kappa\sum_{j=1}^d a_{(\mi_\bx+2\ej,i_t)}\frac{(\mi_\bx+2\ej )! i_t !}{\hKx^{\abs{\mi_\bx}+2}\hKt^{i_t}}.
\end{align*}

Using the condition~$D^{(\mi_\bx,i_t)}	\mathcal{H} v(\bx_K, t_K) = 0$ for all $\abs{\mi_\bx}+i_t\leq p-2$ in \eqref{DEF::QTK}, 
from the above expression we can compute the term $a_{(\mi_\bx+2\e_1,i_t)}$ as follows:
\begin{align*}
    \begin{split}
a_{(\mi_\bx+2\e_1,i_t)}=&
\frac{\hKx^{\abs{\mi_\bx}+2}\hKt^{i_t}}{(\mi_\bx+2\e_1 )! i_t !}
	\Biggl( \frac{1}{\kappa}	a_{(\mi_\bx,i_t+1)}\frac{\mi_\bx ! (i_t+1) !}{\hKx^{\abs{\mi_\bx}}\hKt^{i_t+1}}-
	\sum_{j=2}^d a_{(\mi_\bx+2\ej,i_t)}\frac{(\mi_\bx+2\ej )! i_t !}{\hKx^{\abs{\mi_\bx}+2}\hKt^{i_t}}\Biggl)\\=&
	\frac{1}{(\mi_\bx^{(1)}+2)(\mi_\bx^{(1)}+1)}
	\Biggl( \frac{1}{\kappa}	a_{(\mi_\bx,i_t+1)}\frac{ \hKx^{2}(i_t+1) }{\hKt}-
	\sum_{j=2}^d a_{(\mi_\bx+2\ej,i_t)}(\mi_\bx^{(j)}+2)(\mi_\bx^{(j)}+1)\Biggl).
    \end{split}
\end{align*}
Therefore, any polynomial in the quasi-Trefftz space is uniquely determined by its values, and the values of the first-order partial derivatives~$\partial_{x_1}$ at the point~$(\bx_K, t_K)$.

Given~$\{\psi_\alpha\}_{\alpha = 1, \dots, \binom{p+d}{d}}$ and~$\{\phi_\beta\}_{\beta = 1, \ldots, \binom{p-1+d}{d}}$ bases of~$\IP^p(\IR^d)$ and $\IP^{p-1}(\IR^d)$, respectively, we define the following quasi-Trefftz functions:
\begin{equation*}
 \left\{b_{J}\in\qT^p(K)  \bigg\vert
	\begin{tabular}{ll}
 $	b_{J}(\bx_K^{(1)},\cdot)=	\psi_{J} \text{ and } \partial_{x_1}b_{J}(\bx_K^{(1)},\cdot)=0$ & $\text{ for } 1\leq J\leq \binom{p+d}{d}, $\\[0.2cm]
	$	b_{J}(\bx_K^{(1)},\cdot)= 0 \text{ and } \partial_{x_1} b_{J}(\bx_K^{(1)},\cdot)=\phi_{J-\binom{p+d}{d}}$  & $\text{ for } \binom{p+d}{d}<J\leq \binom{p+d}{d}+\binom{p-1+d}{d}.$
	\end{tabular}
	\right\}.
\end{equation*}
It can be proven that this set of functions forms a basis for the space $\qT^p(K)$ (see e.g. \cite[Prop. 5]{Gomez_Moiola:2024}, \cite[Prop. 2.6]{IG_Moiola_Perinati_Stocker:2024}).

\section{Proofs of the~\emph{a priori} error estimats}
\subsection{Proof of Theorem~\ref{THM::ERROR-ENERGY-QK} for the tensor-product polynomial space\label{APP::PROOF-ENERGY-QK}}
\paragraph{Proof of~\eqref{EQN::ERROR-LDGN-QK}.}
We denote by~$\dpth(\cdot)$ the broken first-order time derivative operator. The~\emph{a priori} error bound~\eqref{EQ::STRANG_NEWTON}, together with the definition in~\eqref{EQN::DG-NORMS-3} of the norm~$\Tnorm{\cdot}{\LDGN}$, the bound in Lemma~\ref{LEMMA::BOUND-Lh} for the lifting operator~$\Lu$, the bound in Proposition~\ref{PROP::BOUND-Nh-LDG} for the discrete Newton potential~$\Nh(\cdot)$, and the inconsistency bound in Lemma~\ref{LEMMA::INCONSISTENCY-BOUND}, gives
\begin{alignat}{3}
\nonumber
\Tnorm{u - \uh}{\LDGN} \lesssim 
& \ \SemiNorm{u - \vh}{\sf J} + \Norm{\nablaxh(u - \vh)}{L^2(\QT)} + \Norm{\dpth (u - \vh)}{L^2(QT)} \\
\nonumber 
& + \Norm{\jump{\lambdah^{-\frac12 } (u - \vh)}_t}{L^2(\Fspa)} + \Norm{\lambdah^{-\frac12} (u - \vh)}{L^2(\FO)} \\
\nonumber
& + \Norm{\eta_F^{\frac12} \jump{u - \vh}_{\sf N}}{L^2(\Ftime)^d} + \Norm{\eta_F^{\frac12} (u - \vh)}{L^2(\FD)} \\
\nonumber
& + \Norm{\eta_F^{-\frac12} \mvl{\bk (\nablaxh u - \PiO \nablaxh u)}_{1 - \alpha_F}}{L^2(\Ftime)^d} \\
\nonumber
& + \Norm{\eta_F^{-\frac12} \bk(\nablaxh u - \PiO \nablaxh u)}{L^2(\FD)}  
\\
\label{EQN::ERROR-LDGN-SPLIT-QK}
& =: I_1 + I_2 + I_3 + I_4 + I_5 + I_6 + I_7 + I_8 + I_9 & & \qquad \forall \vh \in \Vp(\Th).
\end{alignat}

Let~$\vh\in\Vp(\Th)$ be defined as
$$\vh{}_{|_K} = \pi_{\pK}^t \TPix u_{|_{K}} \qquad \forall K \in \Th.$$  
We now bound each term~$\{I_i\}_{i = 1}^9$ on the right-hand side of~\eqref{EQN::ERROR-LDGN-SPLIT-QK}.

Using the triangle inequality, the trace inequality~\eqref{EQN::TRACE-INEQUALITY-CONTINUOUS-1D}, 
estimates~\eqref{EQN::ESTIMATE-TPI-VOLUME} and~\eqref{EQN::ESTIMATE-TPI-STABILITY} for~$\TPix$, 
and estimate~\eqref{EQN::ESTIMATE-PI-TIME-VOLUME} for~$\pi_{\pK}^t$, we get 
\begin{align}
\nonumber
I_1^2 & = \SemiNorm{u-\vh}{\J}^2 
= \frac12 \big(\Norm{u - \vh}{L^2(\FT)}^2 + \Norm{\jump{u - \vh}_t}{L^2(\Fspa)}^2 + \Norm{u - \vh}{L^2(\FO)}^2\big) \\
\nonumber
& \lesssim \sum_{K \in \Th} \sum_{F \in \FKspace} \Norm{u - \pi_{\pK}^t \TPix u}{L^2(F)}^2 \\
\nonumber
& \lesssim \sum_{K \in \Th} \Big(
\frac{\pK}{\hKt} \Norm{u - \pi_{\pK}^t \TPix u}{L^2(K)}^2 + \frac{\hKt}{\pK} \Norm{\dpt u - \dpt \pi_{\pK}^t \TPix u}{L^2(K)}^2 \Big) \\
\nonumber
& \lesssim \sum_{K \in \Th} \Big( 
\frac{\pK}{\hKt} \Big(\Norm{u - \TPix u}{L^2(K)}^2 + \Norm{\TPix(u - \pi_{\pK}^t u)}{L^2(K)}^2 \Big) \\
\nonumber
& \quad + \frac{\hKt}{\pK} \Big(\Norm{\dpt u - \TPix \dpt u}{L^2(K)}^2 + \Norm{\TPix (\dpt u - \dpt \pi_{\pK}^t u)}{L^2(K)}^2\Big)
\Big) \\
\nonumber
& \lesssim \sum_{K \in \Th} \Big( 
\frac{\pK}{\hKt} \Big(\Norm{u - \TPix u}{L^2(K)}^2 + \frac{\hKx^2}{\pK^2}\Norm{\frakEx u - \pi_{\pK}^t \frakEx u}{L^2(\Kt; H^1(\calKx))}^2 + \Norm{u - \pi_{\pK}^t u}{L^2(K)}^2 \Big) \\
\nonumber
& \quad + \frac{\hKt}{\pK} \Big(\Norm{\dpt u - \TPix \dpt u}{L^2(K)}^2 + \frac{\hKx^2}{\pK^2}\Norm{\dpt \frakEx u - \dpt \pi_{\pK}^t \frakEx u}{L^2(\Kt; H^1(\calKx))}^2 + \Norm{\dpt u - \dpt \pi_{\pK}^t u}{L^2(K)}^2\Big)
\Big) \\
\nonumber
& \lesssim \sum_{K \in \Th} \Big(
\frac{\hKx^{2\ellK}}{\hKt \pK^{2 l_K - 1}} \Norm{\frakEx u}{L^2(\Kt; H^{l_K}(\calKx))}^2 + \frac{\hKx^2 \hKt^{2\tsK - 3}}{\pK^{2 s_K - 1}} \Norm{\frakEx u}{H^{s_K- 1}(\Kt; H^1(\calKx))}^2 \\
\nonumber
& \quad   + \frac{\hKt^{2\tsK - 1}}{\pK^{2 s_K - 1}} \Norm{u}{H^{s_K}(\Kt; L^2(\Kx))}^2 + \frac{\hKt \hKx^{2\ellK - 2}}{\pK^{2 l_K - 1}}\Norm{\frakEx u}{H^1(\Kt; H^{l_K - 1}(\calKx))}^2 \\
\label{EQN::ESTIMATE-I1-QK}
& \quad + \frac{\hKx^2 \hKt^{2\tsK - 3}}{\pK^{2s_K- 1}} \Norm{\frakEx u}{H^{s_K- 1}(\Kt; H^1(\calKx))}^2 \Big). 
\end{align}

Using the triangle inequality, estimate~\eqref{EQN::ESTIMATE-PI-TIME-VOLUME} for~$\pi_{\pK}^t$ and its stability properties, and estimate~\eqref{EQN::ESTIMATE-TPI-VOLUME} for~$\TPix$, we obtain
\begin{align}
\nonumber
I_2^2 & = \Norm{\nablaxh(u - \vh)}{L^2(\QT)}^2 = \sum_{K \in \Th} \Norm{\nablax (u - \pi_{\pK}^t \TPix u)}{L^2(K)^d}^2 \\
\nonumber 
& \lesssim \sum_{K \in \Th} \big(\Norm{\nablax u - \pi_{\pK}^t \nablax u}{L^2(K)^d}^2 + \Norm{\pi_{\pK}^t \nablax (u - \TPix u)}{L^2(K)^d}^2 \big) \\
\label{EQN::ESTIMATE-I2-QK}
& \lesssim \sum_{K \in \Th} \Big(\frac{\hKt^{2 \tsK- 2}}{\pK^{2 s_K- 2}} \Norm{u}{H^{s_K - 1}(\Kt; H^1(\Kx))}^2 + \frac{\hKx^{2\ellK - 2}}{\pK^{2 l_K - 2}} \Norm{\frakEx u}{L^2(\Kt; H^{l_K}(\calKx))}^2 \Big).
\end{align}

The third term on the right-hand side of~\eqref{EQN::ERROR-LDGN-SPLIT-QK} can be bounded using the triangle inequality, estimates~\eqref{EQN::ESTIMATE-TPI-VOLUME} and~\eqref{EQN::ESTIMATE-TPI-STABILITY} for~$\TPix$, and estimate~\eqref{EQN::ESTIMATE-PI-TIME-VOLUME} for~$\pi_{\pK}^t$ as follows:
\begin{align}
\nonumber
I_3^2 & =  \Norm{\dpth (u - \vh)}{L^2(\QT)}^2 = \sum_{K \in \Th} \Norm{\dpt (u - \pi_{\pK}^t \TPix u)}{L^2(K)}^2 \\
\nonumber 
& \lesssim \sum_{K \in \Th} \big(\Norm{\dpt u - \TPix \dpt u}{L^2(K)}^2 + \Norm{\TPix (\dpt u - \dpt \pi_{\pK}^t u)}{L^2(K)}^2 \big) \\
\nonumber
& \lesssim \sum_{K \in \Th} \Big( \frac{\hKx^{2\ellK - 2} }{\pK^{2 l_K - 2} }\Norm{\frakEx u}{H^1(\Kt; H^{l_K - 1}(\calKx))}^2 + \frac{\hKx^2}{\pK^2}\Norm{\dpt \frakEx u - \dpt \pi_{\pK}^t \frakEx u}{L^2(\Kt; H^1(\calKx))}^2 \\
\nonumber
& \quad + \Norm{\dpt (u - \pi_{\pK}^t u)}{L^2(K)}^2 \Big) \\
\nonumber
& \lesssim \sum_{K \in \Th} \Big( \frac{\hKx^{2\ellK - 2} }{\pK^{2 l_K - 2}}  \Norm{\frakEx u}{H^1(\Kt; H^{l_K- 1}(\calKx))}^2 + \frac{\hKx^2 \hKt^{2\tsK - 4}}{\pK^{2s_K - 2}} \Norm{\frakEx u}{H^{s_K - 1}(\Kt; H^1(\calKx))}^2 \\
\label{EQN::ESTIMATE-I3-QK}
& \quad + \frac{\hKt^{2\tsK - 2}}{\pK^{2s_K - 2}}  \Norm{u}{H^{s_K}(\Kt; L^2(\Kx))}^2\Big).
\end{align}

Proceeding as for the estimate of~$I_1$ in~\eqref{EQN::ESTIMATE-I1-QK} and using the definition of~$\lK$ in~\eqref{EQN::LambdaK}, we have
\begin{align}
\nonumber
I_4^2 + I_5^2 & = \Norm{\jump{\lambdah^{-\frac12 } (u - \vh)}_t}{L^2(\Fspa)}^2 + \Norm{\lambdah^{-\frac12} (u - \vh)}{L^2(\FO)}^2 \\
\nonumber
& \lesssim \sum_{K \in \Th} \sum_{F \in \FKspace} \frac{\hpK^2}{\hhKt} \Norm{u - \pi_{\pK}^t \TPix u}{L^2(F)}^2 \\
\nonumber
& \lesssim \sum_{K \in \Th} \frac{\hpK^2}{\hhKt} \Big(
\frac{\hKx^{2\ellK}}{\hKt \pK^{2 l_K - 1}} \Norm{\frakEx u}{L^2(\Kt; H^{l_K}(\calKx))}^2 + \frac{\hKx^2 \hKt^{2\tsK - 3}}{\pK^{2 s_K - 1}} \Norm{\frakEx u}{H^{s_K - 1}(\Kt; H^1(\calKx))}^2 \\
\nonumber
& \quad   + \frac{\hKt^{2\tsK - 1}}{\pK^{2 s_K - 1}} \Norm{u}{H^{s_K}(\Kt; L^2(\Kx))}^2 + \frac{\hKt \hKx^{2\ellK - 2}}{\pK^{2 l_K - 1}}\Norm{\frakEx u}{H^1(\Kt; H^{l_K- 1}(\calKx))}^2 \\
\label{EQN::ESTIMATE-I4-I5-QK}
& \quad + \frac{\hKx^2 \hKt^{2\tsK - 3}}{\pK^{2 s_K - 1}} \Norm{\frakEx u}{H^{s_K- 1}(\Kt; H^1(\calKx))}^2 \Big). 
\end{align}

Using the multiplicative trace inequality in~\eqref{EQN::TRACE-INEQUALITY-CONTINUOUS} for continuous functions and the Young inequality, we deduce that, for all~$w \in L^2(\Kt; H^1(\Kx))$ and~$K \in \Th$, it holds
\begin{equation}
\label{EQN::AUXILIARY-TRACE-INEQUALITY}
\Norm{w}{L^2(F)} \lesssim \frac{|\Fx|}{|\sKxF|} \Big(\pK \Norm{w}{L^2(\Ft; L^2(\sKxF))}^2 + \frac{\hKx^2}{\pK} \Norm{v}{L^2(\Ft; H^1(\sKxF))}^2\Big) \qquad \forall F = \Fx \times \Ft \in \FKtime,
\end{equation}
where~$\sKxF$ is as in Assumption~\ref{ASSUMPTION::PRISMATIC-MESH}.

Using the triangle inequality, the trace inequality~\eqref{EQN::AUXILIARY-TRACE-INEQUALITY},
Assumption~\ref{ASSUMPTION::PRISMATIC-MESH} on~$\Th$, estimate~\eqref{EQN::ESTIMATE-PI-TIME-VOLUME} for~$\pi_{\pK}^t$ and its stability properties, and estimate~\eqref{EQN::ESTIMATE-TPI-VOLUME} for~$\TPix$, we get
\begin{align}
\nonumber
I_6^2 + I_7^2 & = \Norm{\eta_F^\frac12 \jump{u-\vh}_{\bN}}{L^2(\Ftime)^{d}}^2 +\Norm{\eta_F^{\frac12} (u-\vh)}{L^2    (\FD)}^2
\\
\nonumber
& \lesssim \sum_{K \in \Th} \sum_{F \in \FKtime} \eta_F \frac{|\Fx|}{|\sKxF|}\Big( \pK \Norm{u - \pi_{\pK}^t \TPix u}{L^2(\Ft; L^2(\sKxF))}^2 \\
\nonumber
& \quad + \frac{\hKx^2}{\pK} \Norm{\nablax u - \pi_{\pK}^t \nablax \TPix u}{L^2(\Ft; L^2(\sKxF)^d)}^2 \Big) \\ 
\nonumber
& \lesssim 
\sum_{K \in \Th} \Big(\max_{F \in \FKtime} \eta_F \Big) \frac{1}{\hKx} \Big(\pK \Norm{u - \pi_{\pK}^t u}{L^2(K)}^2 + \pK \Norm{u - \TPix u}{L^2(K)}^2  \\
\nonumber
& \quad + \frac{\hKx^2}{\pK}\ \Norm{\nablax u - \pi_{\pK}^t \nablax u}{L^2(K)^d}^2 + \frac{\hKx^2}{\pK} \Norm{\nablax u - \nablax \TPix u}{L^2(K)^d}^2 \Big) \\
\nonumber
& \lesssim \sum_{K \in \Th} \Big(\max_{F \in \FKtime} \eta_F \Big) \Big( \frac{\hKt^{2\tsK}}{\hKx \pK^{2 s_K - 1}} \Norm{u}{H^{s_K}(\Kt; L^2(\Kx))}^2 + \frac{\hKx^{2\ellK - 1}}{\pK^{2 l_K- 1}} \Norm{\frakEx u}{L^2(\Kt; H^{l_K}(\calKx))}^2 \\
\label{EQN::ESTIMATE-I6-I7-QK}
& \quad + \frac{\hKx \hKt^{2\tsK - 2}}{\pK^{2 s_K - 1}} \Norm{u}{H^{s_K- 1}(\Kt; H^1(\Kx))}^2\Big).
\end{align}

As for the last two terms on the right-hand side of~\eqref{EQN::ERROR-LDGN-SPLIT-QK}, we use the triangle inequality, the trace inequality~\eqref{EQN::AUXILIARY-TRACE-INEQUALITY}, the polynomial trace inequality~\eqref{EQN::TRACE-INEQUALITY-1}, 
the definition in~\eqref{EQN::STABILIZATION-TERM} of the stabilization function~$\eta_F$, Assumption~\ref{ASSUMPTION::PRISMATIC-MESH} on~$\Th$, the stability properties of~$\PiO$ and~$\pi_{\pK}^t$, estimate~\eqref{EQN::ESTIMATE-TPI-VOLUME} for~$\TPix$, and estimate~\eqref{EQN::ESTIMATE-PI-TIME-VOLUME} for~$\pi_{\pK}^t$ to obtain
\begin{align}
\nonumber
I_8^2 + I_9^2 
& = \Norm{\eta_F^{-\frac12} \mvl{\bk (\nablaxh u - \PiO \nablaxh u)}_{1 - \alpha_F}}{L^2(\Ftime)^d}^2 + \Norm{\eta_F^{-\frac12} \bk(\nablaxh u - \PiO \nablaxh u)}{L^2(\FD)}^2 \\
\nonumber
& \lesssim  \Norm{\eta_F^{-\frac12} \mvl{\bk (\nablaxh u - \pi_{\pK}^t \TPix \nablaxh u)}_{1 - \alpha_F}}{L^2(\Ftime)^d}^2 + \Norm{\eta_F^{-\frac12} \bk(\nablaxh u - \pi_{\pK}^t \TPix \nablaxh u)}{L^2(\FD)}^2 \\
\nonumber
& \quad + \Norm{\eta_F^{-\frac12} \mvl{\bk \PiO (\nablaxh u - \pi_{\pK}^t \TPix \nablaxh u)}_{1 - \alpha_F}}{L^2(\Ftime)^d}^2 + \Norm{\eta_F^{-\frac12} \bk \PiO (\nablaxh u - \pi_{\pK}^t \TPix \nablaxh u)}{L^2(\FD)}^2 \\
\nonumber
& \lesssim \sum_{K \in \Th} \sum_{F = \Fx \times \Ft \in \FKtime} \!\!\!\!\! \eta_F^{-1} \Norm{\nablax u - \pi_{\pK}^t \TPix \nablax u}{L^2(F)^d}^2 \\
\nonumber
& \quad + \sum_{K \in \Th} \sum_{F = \Fx \times \Ft \in \FKtime} \!\!\!\!\! \eta_F^{-1} \frac{(\pK + 1)(\pK + d)}{d} \frac{|\Fx|}{|\sKxF|} \Norm{\PiO(\nablax u - \pi_{\pK}^t \TPix \nablax u)}{L^2(\Ft; L^2(\sKxF)^d)}^2 \\
\nonumber
& \lesssim \sum_{K \in \Th} \sum_{F = \Fx \times \Ft \in \FKtime} \eta_F^{-1} \frac{|\Fx|}{|\sKxF|}  \Big( \pK \Norm{\nablax u - \pi_{\pK}^t \TPix \nablax u}{L^2(\Ft; L^2(\sKxF)^d)}^2 \\
\nonumber 
& \quad + \frac{\hKx^2}{\pK} \Norm{\nablax u - \pi_{\pK}^t \TPix \nablax u}{L^2(\Ft; H^1(\sKxF)^d)}^2 \Big)  + \sum_{K \in \Th} \Norm{\PiO(\nablax u - \pi_{\pK}^t \TPix \nablax u)}{L^2(K)^d}^2 \\
\nonumber
& \lesssim \sum_{K \in \Th} \frac{1}{\pK^2} \Big( \pK \Norm{\nablax u - \pi_{\pK}^t \nablax u}{L^2(K)^d}^2 + \pK \Norm{\nablax u - \TPix \nablax u }{L^2(K)^d}^2 \\
\nonumber
& \quad + \frac{\hKx^2}{\pK} \Norm{\nablax u - \pi_{\pK}^t \nablax u}{L^2(\Kt; H^1(\Kx)^d)}^2 
+ \frac{\hKx^2}{\pK} \Norm{\nablax u - \TPix \nablax u}{L^2(\Kt; H^1(\Kx)^d)}^2 \Big) \\
\nonumber
& \quad + \sum_{K \in \Th} \Big(\Norm{\nablax u - \pi_{\pK}^t \nablax u }{L^2(K)^d}^2 + \Norm{\nablax u - \TPix \nablax u}{L^2(K)^d}^2 \Big) \\
\nonumber
& \lesssim \sum_{K \in \Th} \Big( 
\frac{\hKt^{2\tsK - 2}}{\pK^{2 s_K- 1}} \Norm{u}{H^{s_K- 1}(\Kt; H^1(\Kx))}^2 + \frac{\hKx^{2\ellK - 2}}{\pK^{2 l_K - 1}} \Norm{\frakEx \nablax u}{L^2(\Kt; H^{l_K - 1}(\calKx)^d)}^2 \\
\nonumber
& \quad + \frac{\hKx^2 \hKt^{2\theta_K}}{\pK^{2\vartheta_K + 3} } \Norm{u}{H^{\vartheta_{K}}(\Kt; H^2(\Kx))}^2 + \frac{\hKx^{2\ellK - 2}}{\pK^{2 l_K - 1}} \Norm{\frakEx \nablax u}{L^2(\Kt; H^{l_K- 1}(\calKx)^d)}^2  \\
\label{EQN::ESTIMATE-I8-I9-QK}
& \quad + \frac{\hKt^{2\tsK - 2}}{\pK^{2s_K - 2}} \Norm{u}{H^{s_K - 1}(\Kt; H^1(\Kx))}^2 + \frac{\hKx^{2\ellK - 2}}{\pK^{2 l_K - 2}} \Norm{\frakEx \nablax u}{L^2(\Kt; H^{l_K - 1}(\calKx)^d)}^2 
\Big).
\end{align}
Combining estimates~\eqref{EQN::ESTIMATE-I1-QK}, \eqref{EQN::ESTIMATE-I2-QK}, \eqref{EQN::ESTIMATE-I3-QK}, \eqref{EQN::ESTIMATE-I4-I5-QK}, \eqref{EQN::ESTIMATE-I6-I7-QK}, and~\eqref{EQN::ESTIMATE-I8-I9-QK} with~\eqref{EQN::ERROR-LDGN-SPLIT-QK}, we obtain~\eqref{EQN::ERROR-LDGN-QK}.

\paragraph{Proof of~\eqref{EQN::ERROR-LDGp-QK}.}
The~\emph{a priori} error bound~\eqref{EQ::STRANG_PLUS}, together with the definitions in~\eqref{EQN::DG-NORMS-4} and~\eqref{EQN::DG-NORMS-5} of the norms~$\Norm{\cdot}{\LDGp}$ and~$\Norm{\cdot}{\LDGs}$, respectively, the bound in Lemma~\ref{LEMMA::BOUND-Lh} for the lifting operator~$\Lu$, and the bound in Lemma~\ref{LEMMA::INCONSISTENCY-BOUND} for the inconsistency term~$\Rh(\cdot, \cdot)$ gives
\begin{alignat}{3}
\nonumber
\Tnorm{u - \uh}{\LDGp} \lesssim 
& \ \SemiNorm{u - \vh}{\sf J} + \Norm{\nablaxh(u - \vh)}{L^2(\QT)} + \Norm{\lambdah^{\frac12}\dpth (u - \vh)}{L^2(\QT)} \\
\nonumber
& +
\Norm{\lambdah^{-\frac12} (u-\vh)}{L^2(\QT)}^2 +\Norm{ (u-\vh)^{-}}{L^2(\Fspa)}^2+\Norm{u-\vh}{L^2(\FT)}^2\\
\nonumber
& +
\Norm{\eta_F^\frac12 \jump{u-\vh}_{\bN}}{L^2(\Ftime)^{d}}^2 +\Norm{\eta_F^{\frac12} (u-\vh)}{L^2(\FD)}^2\\
\nonumber
& + \Norm{\eta_F^{-\frac12} \mvl{\bk (\nablaxh u - \PiO \nablaxh u)}_{1 - \alpha_F}}{L^2(\Ftime)^d} \\
\nonumber
& + \Norm{\eta_F^{-\frac12} \bk(\nablaxh u - \PiO \nablaxh u)}{L^2(\FD)}  \\
\label{EQN::ERROR-LDGp-SPLIT-QK}
& =: I_1 + I_2 + J_1 + J_2 + J_3 + J_4 + I_6 + I_7 + I_8 + I_9 & & \forall \vh \in \Vp(\Th).
\end{alignat}

Let~$\vh\in\Vp(\Th)$ be defined as above. We focus on the new terms~$\{J_i\}_{i = 1}^4$ on the right-hand side of~\eqref{EQN::ERROR-LDGp-SPLIT-QK}.

Proceeding as for the estimate in~\eqref{EQN::ESTIMATE-I3-QK} of~$I_3$, we have
\begin{align}
\nonumber
J_1^2 & =\Norm{\lambdah^{\frac12}\dpth (u - \vh)}{L^2(\QT)}^2 \\
\nonumber
& \lesssim \sum_{K \in \Th} \Big(\frac{\hKt \hKx^{2\ellK - 2}}{\pK^{2 l_K}} \Norm{\frakEx u}{H^1(\Kt; H^{l_K- 1}(\calKx))}^2 + \frac{\hKx^2 \hKt^{2\tsK - 3}}{\pK^{2 s_K}} \Norm{\frakEx u}{H^{s_K- 1}(\Kt; H^1(\calKx))}^2 \\
\label{EQN::ESTIMATE-J1-QK}
& \quad + \frac{\hKt^{2\tsK - 1}}{\pK^{2 s_K}} \Norm{u}{H^{s_K}(\Kt; L^2(\Kx))}^2\Big).
\end{align}

The definition of~$\lK$ in~\eqref{EQN::LambdaK}, estimate~\eqref{EQN::ESTIMATE-TPI-VOLUME} for~$\TPix$, and estimate~\eqref{EQN::ESTIMATE-PI-TIME-VOLUME} for~$\pi_{\pK}^t$ lead to
\begin{align}
\nonumber
J_2^2 & = \Norm{\lambdah^{-\frac12}(u - \vh)}{L^2(\QT)}^2 \\
\label{EQN::ESTIMATE-J2-QK}
& \lesssim \sum_{K \in \Th} \frac{\hpK^2}{\hhKt} \Big( \frac{\hKt^{2\tsK}}{\pK^{2s_K}} \Norm{u}{H^{s_K}(\Kt; L^2(\Kx))}^2 + \frac{\hKx^{2\ellK}}{\pK^{2l_K}} \Norm{\frakEx u}{L^2(\Kt; H^{l_K}(\calKx))}^2 \Big).
\end{align}

Similarly as for the estimate in~\eqref{EQN::ESTIMATE-I1-QK} of~$I_1$, we have
\begin{align}
\nonumber
J_3^2 + J_4^2 & =\Norm{(u - \vh)^{-}}{L^2(\Fspa)}^2 + \Norm{u - \vh}{L^2(\FT)}^2 \\
\nonumber
& \lesssim \sum_{K \in \Th} \Big(
\frac{\hKx^{2\ellK}}{\hKt \pK^{2 l_K - 1}} \Norm{\frakEx u}{L^2(\Kt; H^{l_K}(\calKx))}^2 + \frac{\hKx^2 \hKt^{2\tsK - 3}}{\pK^{2 s_K - 1}} \Norm{\frakEx u}{H^{s_K- 1}(\Kt; H^1(\calKx))}^2 \\
\nonumber
& \quad   + \frac{\hKt^{2\tsK - 1}}{\pK^{2 s_K - 1}} \Norm{u}{H^{s_K}(\Kt; L^2(\Kx))}^2 + \frac{\hKt \hKx^{2\ellK - 2}}{\pK^{2 l_K - 1}}\Norm{\frakEx u}{H^1(\Kt; H^{l_K - 1}(\calKx))}^2 \\
\label{EQN::ESTIMATE-J3-J4-QK}
& \quad + \frac{\hKx^2 \hKt^{2\tsK - 3}}{\pK^{2 s_K - 1}} \Norm{\frakEx u}{H^{s_K - 1}(\Kt; H^1(\calKx))}^2 \Big). 
\end{align}

The error estimate~\eqref{EQN::ERROR-LDGp-QK} then follows by combining~\eqref{EQN::ERROR-LDGp-SPLIT-QK} with estimates~\eqref{EQN::ESTIMATE-I1-QK}, \eqref{EQN::ESTIMATE-I2-QK}, \eqref{EQN::ESTIMATE-I6-I7-QK}, \eqref{EQN::ESTIMATE-I8-I9-QK}, \eqref{EQN::ESTIMATE-J1-QK}, \eqref{EQN::ESTIMATE-J2-QK}, and~\eqref{EQN::ESTIMATE-J3-J4-QK}.

\subsection{Proof of Theorem~\ref{THM::ERROR-ENERGY-PK} for the standard polynomial space\label{APP::PROOF-ENERGY-PK}}
\paragraph{Proof of~\eqref{EQN::ERROR-LDGN-PK}.}
Let~$\vh\in\Vp(\Th)$ be defined as~$\vh{}_{|_K} = \TPi v_{|_{K}}$ for all~$K \in \Th$. 

We proceed as in the proof of Theorem~\ref{THM::ERROR-ENERGY-QK}. Therefore, it is enough to bound the terms~$\{I_i\}_{i = 1}^9$ on the right-hand side of~\eqref{EQN::ERROR-LDGN-SPLIT-QK}.

Using the approximation results in Lemma~\ref{LEMMA::ESTIMATE-TPI-XT} for~$\TPi$, the definition of~$\lK$ in~\eqref{EQN::LambdaK}, Assumption~\ref{ASM::SHAPE-REGULARITY}, and the triangle inequality, we get
\begin{align}
\nonumber
I_1^2 & = \SemiNorm{u-\vh}{\J}^2 
= \frac12 \big(\Norm{u-\vh}{L^2(\FT)}^2 + \Norm{\jump{u-\vh}_t}{L^2(\Fspa)}^2 +\Norm{u-\vh}{L^2(\FO)}^2\big) \\
\label{EQN::ESTIMATE-I1}
& \lesssim \sum_{K \in \Th} \sum_{F \in \FKspace} \Norm{u - \TPi u}{L^2(F)}^2 
\lesssim 
\sum_{K \in \Th}  \frac{\hK^{2\ellK - 1}}{\pK^{2 l_K- 1}} \Norm{\frakE u}{H^{l_K}(\calK )}^2, \\
\nonumber
I_2^2 + I_3^2 & = \Norm{\nablaxh(u - \vh)}{L^2(\QT)}^2 + \Norm{\dpth (u - \vh)}{L^2(\QT)}^2 \\
\label{EQN::ESTIMATE-I2-I3}
& = \sum_{K \in \Th} \Norm{u - \TPi u}{H^1(K)}^2 
\lesssim \sum_{K \in \Th} \frac{\hK^{2\ellK - 2}}{\pK^{2 l_K - 2}} \Norm{\frakE u}{H^{l_K}(\calK )}^2, \\
\nonumber
I_4^2 + I_5^2 & = \Norm{\jump{\lambdah^{-\frac12 } (u - \vh)}_t}{L^2(\Fspa)}^2 + \Norm{\lambdah^{-\frac12} (u - \vh)}{L^2(\FO)}^2 \\
\label{EQN::ESTIMATE-I4-I5}
& \lesssim \sum_{K \in \Th} \sum_{F \in \FKspace} \frac{\hpK^2}{\hhKt} \Norm{u - \TPi u}{L^2(F)}^2 \lesssim \sum_{K \in \Th} \frac{\hpK^2}{\hhKt} \frac{\hK^{2\ellK - 1}}{\pK^{2 l_K - 1}} \Norm{\frakE u}{H^{l_K}(\calK )}^2, \\
\nonumber
I_6^2 + I_7^2 & = \Norm{\eta_F^\frac12 \jump{u-\vh}_{\bN}}{L^2(\Ftime)^{d}}^2 +\Norm{\eta_F^{\frac12} (u-\vh)}{L^2(\FD)}^2
\\
\nonumber
& \lesssim \sum_{K \in \Th} \sum_{F \in \FKtime} \eta_F \Norm{u - \TPi u}{L^2(F)}^2 \\ 
\label{EQN::ESTIMATE-I6-I7}
& \lesssim 
\sum_{K \in \Th} \big(\max_{F \in \FKtime} \eta_F \big) \frac{\hK^{2\ellK - 1}}{\pK^{2 l_K - 1}} \Norm{\frakE u}{H^{l_K}(\calK )}^2.
\end{align}

As for the last two terms on the right-hand side of~\eqref{EQN::ERROR-LDGN-SPLIT-QK}, we use the triangle inequality, the trace inequality for simplices in~\eqref{EQN::TRACE-INEQUALITY-CONTINUOUS}, 
the polynomial trace inequality in~\eqref{EQN::TRACE-INEQUALITY-1}, the definition of the stabilization function~$\eta_F$ in~\eqref{EQN::STABILIZATION-TERM},
Assumptions~\ref{ASSUMPTION::PRISMATIC-MESH} and~\ref{ASM::SHAPE-REGULARITY} on~$\Th$, the fact that~$\diam(\sKxF) \le \hKx$ for all~$\sKxF$, the stability properties of~$\PiO$, and the estimates for~$\TPi$ in Lemma~\ref{LEMMA::ESTIMATE-TPI-XT} to obtain
\begin{align}
\nonumber
I_8^2 + I_9^2 
& = \Norm{\eta_F^{-\frac12} \mvl{\bk (\nablaxh u - \PiO \nablaxh u)}_{1 - \alpha_F}}{L^2(\Ftime)^d}^2 + \Norm{\eta_F^{-\frac12} \bk(\nablaxh u - \PiO \nablaxh u)}{L^2(\FD)}^2 \\
\nonumber
& \lesssim  \Norm{\eta_F^{-\frac12} \mvl{\bk (\nablaxh u - \TPi \nablaxh u)}_{1 - \alpha_F}}{L^2(\Ftime)^d}^2 + \Norm{\eta_F^{-\frac12} \bk(\nablaxh u - \TPi \nablaxh u)}{L^2(\FD)}^2 \\
\nonumber
& \quad + \Norm{\eta_F^{-\frac12} \mvl{\bk \PiO (\nablaxh u - \TPi \nablaxh u)}_{1 - \alpha_F}}{L^2(\Ftime)^d}^2 + \Norm{\eta_F^{-\frac12} \bk \PiO (\nablaxh u - \TPi \nablaxh u)}{L^2(\FD)}^2 \\
\nonumber
& \lesssim \sum_{K \in \Th} \sum_{F = \Fx \times \Ft \in \FKtime} \!\!\!\!\! \eta_F^{-1} \Norm{\nablax u - \TPi \nablax u}{L^2(F)^d}^2 \\
\nonumber
& \quad + \sum_{K \in \Th} \sum_{F = \Fx \times \Ft \in \FKtime} \!\!\!\!\! \eta_F^{-1} \frac{(\pK + 1)(\pK + d)}{d} \frac{|\Fx|}{|\sKxF|} \Norm{\PiO(\nablax u - \TPi \nablax u)}{L^2(\Ft; L^2(\sKxF)^d)}^2 \\
\nonumber
& \lesssim \sum_{K \in \Th} \sum_{F = \Fx \times \Ft \in \FKtime} \eta_F^{-1} \frac{|\Fx|}{|\sKxF|} \big(\Norm{\nablax u - \TPi \nablax u}{L^2(\Ft; L^2(\sKxF)^d)}^2 \\
\nonumber 
& \quad + \hKx^2 \Norm{\nablax u - \TPi \nablax u}{L^2(\Ft; H^1(\sKxF)^d)}^2 \big)  + \sum_{K \in \Th} \Norm{\PiO(\nablax u - \TPi \nablax u)}{L^2(K)^d}^2 \\
\nonumber
& \lesssim \sum_{K \in \Th} \Big(\frac{1}{\pK^2} \Norm{\nablax u - \TPi \nablax u}{L^2(K)^d}^2   + \frac{\hKx^2}{\pK^2} \Norm{\nablax u -  \TPi \nablax u}{L^2(\Kt; H^1(\Kx)^d)}^2 \\
\nonumber
& \quad +  \Norm{\nablax u - \TPi \nablax u}{L^2(K)^d}^2 \Big)  \\
\label{EQN::ESTIMATE-I8-I9}
& \lesssim 
\sum_{K \in \Th} \frac{\hK^{2\ellK - 2}}{\pK^{2 l_K- 2}}\Big(1 + \frac{1}{\pK^2}\Big) \Norm{\frakE \nablax u}{H^{l_K - 1}(\calK)^d}^2.
\end{align}

Combining estimates~\eqref{EQN::ESTIMATE-I1}, \eqref{EQN::ESTIMATE-I2-I3}, \eqref{EQN::ESTIMATE-I4-I5}, \eqref{EQN::ESTIMATE-I6-I7}, and~\eqref{EQN::ESTIMATE-I8-I9} with~\eqref{EQN::ERROR-LDGN-SPLIT-QK}, we obtain~\eqref{EQN::ERROR-LDGN-PK}.

\paragraph{Proof of~\eqref{EQN::ERROR-LDGp-PK}.}
Let~$\vh\in\Vp(\Th)$ be defined as above. 
As before, we only need to estimate the terms~$\{J_i\}_{i = 1}^4$ on the right-hand side of~\eqref{EQN::ERROR-LDGp-SPLIT-QK}.

Using the approximation results in Lemma~\ref{LEMMA::ESTIMATE-TPI-XT} for~$\TPi$ and Assumption~\ref{ASM::SHAPE-REGULARITY}, we get
\begin{align}
\label{EQN::ESTIMATE-J1}
J_1^2 & =\Norm{\lambdah^{\frac12}\dpth (u - \vh)}{L^2(\QT)}^2 \lesssim \sum_{K \in \Th} \frac{\hhKt}{\hpK^2}\frac{\hK^{2\ellK -2}}{\pK^{2 l_K-2}}  \Norm{\frakE u}{H^{l_K}(\calK )}^2 \lesssim \sum_{K \in \Th} \frac{\hK^{2\ellK-1}}{\pK^{2 l_K}}  \Norm{\frakE u}{H^{l_K}(\calK )}^2, \\
\label{EQN::ESTIMATE-J2}
J_2^2 & =\Norm{\lambdah^{-\frac12}(u - \vh)}{L^2(\QT)}^2 \lesssim \sum_{K \in \Th} \frac{\hpK^2}{\hhKt}\frac{\hK^{2\ellK}}{\pK^{2 l_K}} \Norm{\frakE u}{H^{l_K}(\calK )}^2, \\
\label{EQN::ESTIMATE-J3-J4}
J_3^2 + J_4^2 & =\Norm{ (u-\vh)^{-}}{L^2(\Fspa)}^2+\Norm{u-\vh}{L^2(\FT)}^2\lesssim \sum_{K \in \Th} \frac{\hK^{2\ellK -1}}{\pK^{2 l_K- 1}} \Norm{\frakE u}{H^{l_K}(\calK )}^2.
\end{align}

The error estimate~\eqref{EQN::ERROR-LDGp-PK} then follows by combining~\eqref{EQN::ERROR-LDGN-SPLIT-QK} with estimates~\eqref{EQN::ESTIMATE-I1}, \eqref{EQN::ESTIMATE-I2-I3}, \eqref{EQN::ESTIMATE-I6-I7}, \eqref{EQN::ESTIMATE-I8-I9}, \eqref{EQN::ESTIMATE-J1}, \eqref{EQN::ESTIMATE-J2}, and~\eqref{EQN::ESTIMATE-J3-J4}.

\subsection{Proof of Theorem~\ref{THM::ERROR-ENERGY-QTK} for the quasi-Trefftz polynomial space\label{APP::PROOF-ENERGY-QTK}}
Let~$\vh\in\Vp(\Th)$ be defined as~$\vh{}_{|_K} = T^{p+1}_{(\bx_K,t_K)}[u_{|_{K}}]$ for all~$K \in \Th$.
Using the approximation results in~\eqref{EQN::ERROR-TAYLOR} for~$T^{p+1}_{(\bx_K,t_K)}$, the definition of~$\lK$ in~\eqref{EQN::LambdaK}, the triangle inequality, 
the trace inequalities in~\cite[Thm.~1.6.6]{Brenner_Scott_2007} and~\cite[Lemma 2]{Moiola_Perugia_2018}, 
the definition in~\eqref{EQN::STABILIZATION-TERM} of~$\eta_F$, the local quasi-uniformity conditions in~\eqref{EQN::LOCAL-QUASI-UNIFORMITY}, and the orthotropic scaling in~\eqref{EQN::ORTHOTROPIC-RELATION}, we bound the terms on the right-hand side of~\eqref{EQN::ERROR-LDGN-SPLIT-QK} as follows:
\begin{align*}
I_1^2 & = \SemiNorm{u - \vh}{\J}^2 
= \frac12 \big(\Norm{u - \vh}{L^2(\FT)}^2 + \Norm{\jump{u - \vh}_t}{L^2(\Fspa)}^2 + \Norm{u-\vh}{L^2(\FO)}^2\big) 
 \\
 & \lp \sum_{K \in \Th}  (\hKt^{-1}\Norm{u-\vh}{L^2(K)}^2+\hKt \Norm{\dpt  (u-\vh)}{L^2(K)}^2)\\
 &\lp \sum_{K \in \Th} (\hK^{2p + 1} \SemiNorm{u}{C^{p+1}(K)}^2 + h_K^{2p + 1} \abs{u}_{C^{p+1}(K)}^2),\\
I_2^2 + I_3^2 & = \Norm{\nablaxh(u - \vh)}{L^2(\QT)}^2 + \Norm{\dpth (u - \vh)}{L^2(\QT)}^2 = \sum_{K \in \Th} \Norm{u - \vh}{H^1(K)}^2 \\
& \lp \sum_{K \in \Th}  \hK^{2p} \Norm{u}{C^{p+1}(K)}^2, \\
I_4^2 + I_5^2 & = \Norm{\jump{\lambdah^{-\frac12 } (u - \vh)}_t}{L^2(\Fspa)}^2 + \Norm{\lambdah^{-\frac12} (u - \vh)}{L^2(\FO)}^2 \\
& \lp \sum_{K \in \Th} \lambda_K^{-1}(\hKt^{-1}\Norm{u-\vh}{L^2(K)}^2+\hKt \Norm{\dpt  (u-\vh)}{L^2(K)}^2) \\
 &\lp \sum_{K \in \Th} \hK^{2p + 2} \SemiNorm{u}{C^{p+1}(K)}^2,\\
I_6^2 + I_7^2 & = \Norm{\eta_F^\frac12 \jump{u-\vh}_{\bN}}{L^2(\Ftime)^{d}}^2 +\Norm{\eta_F^{\frac12} (u-\vh)}{L^2(\FD)}^2
\\
& \lesssim \sum_{K \in \Th} \sum_{F \in \FKtime} \eta_F \Norm{u - \vh}{L^2(F)}^2 \\
& \lp \sum_{K \in \Th}\Big(\max_{F \in \FKtime} \eta_F \Big) 
(\hKx^{-1} \hK^{2p+2} \SemiNorm{u}{C^{p+1}(K)}^2 + \hKx \hK^{2p}\abs{u}_{C^{p+1}(K)}^2) \\
& \lp \sum_{K \in \Th} \hK^{2p}\abs{u}_{C^{p+1}(K)}^2,\\
I_8^2 + I_9^2 
& = \Norm{\eta_F^{-\frac12} \mvl{\bk(\nabla_{\bx, h} u - \PiO \nabla_{\bx, h} u)}_{1 - \alpha_F}}{L^2(\Ftime)^d}^2 + \Norm{\eta_F^{-\frac12} \bk(\nabla_{\bx, h} u - \PiO \nabla_{\bx, h} u) \cdot \bnOmega}{L^2(\FD)}^2 \\
& \lesssim \sum_{K \in \Th} \sum_{F \in \FKtime} \eta_{F}^{-1} \Norm{\nablax u - \PiO \nablax u}{L^2(F)^d}^2 \\
& \lp \sum_{K \in \Th} \hKx \big(\hKx^{-1} \Norm{\nablax u - \PiO \nablax u }{L^2(K)^d}^2 + \hKx \Norm{\nablax u - \PiO \nablax u}{L^2(\Kt; H^1(\Kx)^d)}^2 \big) \\
& \lp \sum_{K \in \Th} \hK^{2p} \SemiNorm{u}{C^{p + 1}(K)}^2.
\end{align*}

Therefore, combining the above estimates with~\eqref{EQN::ERROR-LDGN-SPLIT-QK}, we obtain~\eqref{EQN::ERROR-LDGN-QT}.
\end{document}